\documentclass[12pt]{article}

\usepackage[utf8]{inputenc}
\usepackage[T1]{fontenc}
\usepackage{lmodern}
\usepackage{amsmath,amssymb,amsthm,mathtools}
\usepackage{geometry}
\usepackage{enumitem}
\usepackage{booktabs}
\usepackage{array}
\usepackage{tikz}
\usetikzlibrary{arrows.meta,positioning,shapes.geometric,calc,fit}
\usepackage{algorithm}
\usepackage{algpseudocode}
\usepackage{hyperref}
\usepackage[nameinlink,capitalise]{cleveref}
\usepackage{tabularx}
\usepackage{booktabs}
\usepackage{makecell}

\hypersetup{colorlinks=true,linkcolor=blue,citecolor=blue,urlcolor=blue}
\setlist[itemize]{leftmargin=2em}
\setlist[enumerate]{leftmargin=2em}

\newtheorem{theorem}{Theorem}[section]
\newtheorem{lemma}[theorem]{Lemma}
\newtheorem{proposition}[theorem]{Proposition}
\newtheorem{corollary}[theorem]{Corollary}

\newtheorem{openproblem}[theorem]{Open Problem}
\theoremstyle{definition}
\newtheorem{definition}[theorem]{Definition}
\newtheorem{example}[theorem]{Example}
\newtheorem{remark}[theorem]{Remark}

\renewenvironment{proof}{{\bf \noindent Proof.}}{\qed}
\everymath{\displaystyle}

\title{Betti numbers for cochordal zero-divisor graphs of commutative rings}
\author{Bilal Ahmad Rather\\[2mm]
        \small School of Mathematics and Statistics, Shandong University of Technology, Zibo 255049, China\\
        \small School of Mechanical Engineering, Shandong University of Technology, Zibo 255049, China\\
        \texttt{\href{mailto:bilalahmadrr@gmail.com}{bilalahmadrr@gmail.com}}
			}
\date{}
\begin{document}
\maketitle

\begin{abstract}
	This paper studies the zero-divisor graphs attached to several finite chain-ring families and computes the homological invariants of their edge ideals by using cochordal constructible systems.  We begin with a general layered graph $C(q,L)$, whose vertices are arranged according to valuation layers and whose adjacency is governed by the single rule $k+\ell\ge L$, form some integers $k$ and $\ell$.  This graph models the zero-divisor structure of a finite chain ring with residue field of order $q$ and nilpotency index $L$.  We prove that $C(q,L)$ is cochordal, determine its type sequence, then correct and refine the Betti formula of its edge ideal [Dung and Vu, Cochordal zero divisor graphs and Betti numbers of their edge ideals, Comm. Algebra  54(2) (2026) 736--744]. The results are then specialized to the Gaussian quotient rings $\mathbb Z_{2^m}[i]$ and to the truncated polynomial rings $\mathbb Z_p[x]/(x^c)$.  We compute projective dimension, regularity, independence number, height, Hilbert series, and Cohen--Macaulay behavior.  The computations show that these quotient rings have $2$-linear resolutions, while Cohen--Macaulayness occurs only in the expected degenerate or complete-graph cases.
\end{abstract}


\medskip


\medskip
\tableofcontents
\section{Introduction}\label{sec:introduction}

The zero-divisor graph of a commutative ring records multiplication-to-zero by a simple graph.  Since Beck's original study of colorings of commutative rings \cite{Beck1988} and the subsequent structural formulation of Anderson and Livingston \cite{AndersonLivingston1999}, this graph has served as a useful interface between ring theory and finite graph theory.  For the residue class ring $\mathbb{Z}/n\mathbb{Z}$, the graph is especially concrete, as its vertices are the nonzero residues whose greatest common divisor with $n$ is not one, and two distinct vertices $x,y$ are adjacent precisely when $n$ divides $xy$. The study of zero-divisor graphs began with Beck’s work on coloring problems for commutative rings \cite{Beck1988}, and was later developed systematically by Anderson and Livingston, who introduced the standard zero-divisor graph associated with a commutative ring \cite{AndersonLivingston1999}. Subsequent work investigated structural properties of these graphs, for example, Akbari, Maimani and Yassemi characterized important cases in which zero-divisor graphs are planar \cite{AkbariMaimaniYassemi2003}, while Anderson and Badawi further refined the theory by studying zero-divisor graphs from the perspective of ring-theoretic annihilation and adjacency relations \cite{AndersonBadawi2008}. In \cite{ArunkumaraCameronKavaskarTamizh2023}, the authors showed that zero divisor graphs are universal graphs.

\medskip

The graph-theoretic simplicity of $\Gamma(\mathbb{Z}/n\mathbb{Z})$ hides a rather rich homological problem.  If $G$ is a finite simple graph on vertices $x_1,\ldots,x_s$, its edge ideal is the squarefree quadratic monomial ideal $I(G)=(x_ix_j:\{x_i,x_j\}\in E(G))$ in a standard graded polynomial ring.  The minimal free resolution of $S/I(G)$ measures the arrangement of edges, induced subgraphs, and vertex covers of $G$ in a way that is sensitive enough to detect induced matchings and chordality, see the standard references \cite{HerzogHibi2011,Villarreal2015}.

\medskip

A decisive point is Fröberg's theorem: the edge ideal $I(G)$ has a $2$-linear resolution if and only if the complement graph $\overline{G}$ is chordal \cite{Froberg1990}.  Thus, a cochordal graph is not only a graph with a forbidden-cycle condition in its complement, but also a graph whose edge ideal has the cleanest possible homological shape.  For such a graph all nonzero graded Betti numbers of $S/I(G)$ lie on the linear strand $\beta_{i,i+1}$, and hence the total Betti number $\beta_i$ is already the relevant graded invariant.

\medskip

The residue class graphs $\Gamma(\mathbb{Z}/n\mathbb{Z})$ therefore provide a natural test family for the interaction between arithmetic and homological algebra.  The arithmetic supplies valuation classes, and the edge relation becomes an inequality on the exponents of prime factors.  The homological task is to convert this valuation data into a free resolution.  Earlier work by Rather, Imran and Diene \cite{RatherImranDiene2024} and by Rather, Imran and Pirzada \cite{RatherImranPirzada2024} computed parts of the linear strand in several small arithmetic families, including cases such as $p^4$, $p^2q$, and $pqr$. Pirzada and Rather \cite{PirzadaRather} studied  the linear strand of edge ideals of some zero divisor graphs.

\medskip

A recent manuscript of Dung and Vu \cite{DungVu2026} proposed a broader method based on cochordal constructible systems and type sequences.  The manuscript proves, in essence, that $\Gamma(\mathbb{Z}/n\mathbb{Z})$ is cochordal exactly when $n$ has one of the forms $p^a$, $p^a q$, or $pqr$, where the primes appearing are distinct.  It also proposes closed Betti-number formulae for the corresponding edge ideals. The present paper validates the useful part of that approach and corrects the parts that are not compatible with the actual edge relation.  The first correction is global and affects every family: the type-sequence Betti formula contains a subtractive term $\binom{k}{i+1}$, where $k$ is the number of center steps in the construction.  When the arithmetic type sequence is evaluated, this term does not disappear.  Omitting it already gives the wrong answer for mentioned values in $\Gamma(\mathbb{Z}/n\mathbb{Z})$, whose edge ideal is principal but for which the uncorrected expression gives two first syzygy generators.
 The second correction is structural.  In the family $n=p^a q$, a class represented by $p^i q$ is not a clique unless $2i\geq a$.  Similarly, in the square-free family $n=pqr$, the classes represented by $pq$, $pr$, and $qr$ are independent sets, not cliques, because the square of such a representative misses one prime factor.  Any constructible system that incorporates previously introduced vertices from the same independent class in the cover generates edges absent from the graph.  This changes the type sequences and, consequently, the Betti formulae.

\medskip

The corrected formulae require a mild refinement of the type calculus.  A clique block contributes repeated identical binomial terms after the type entries are shifted by their positions.  An independent repeated block contributes a hockey-stick sum (binomial-coefficient summation formula like $
\sum_{j=r}^{m} \binom{j}{r}=\binom{m+1}{r+1}$) of binomial coefficients, because the type entries are constant while the positional shifts increase.  This distinction is the technical source of the new formulae in Sections \ref{sec:two-prime}, and \ref{sec:three-prime}.

\medskip

The problem studied here can now be stated precisely.  Let $n\geq 2$, let $R=\mathbb{Z}/n\mathbb{Z}$, let $G_n=\Gamma(R)$, and let $I_n=I(G_n)\subseteq S_n$, where $S_n$ is the polynomial ring whose variables correspond to the nonzero zero-divisors of $R$.  We ask: for which $n$ is $G_n$ cochordal, and, in those cases, what are the exact graded Betti numbers of $S_n/I_n$?

\medskip

Our answer is complete for the cochordal cases.  The classification theorem remains $n=p^a$, $n=p^a q$, or $n=pqr$.  For $p^a$ the type blocks from the valuation classes with exponent at least $\lceil a/2\rceil$ are clique blocks, so the corrected Betti formula is the earlier block sum minus the missing global binomial term.  For $p^a q$ the high valuation blocks are clique blocks but the low valuation blocks are independent repeated blocks, giving a mixed formula.  For $pqr$ all three center classes are independent repeated blocks, and the formula is a sum of three hockey-stick differences minus the same global correction.

\medskip

The corrected results also clarify projective dimension.  In the prime-power case, the projective dimension is still $p^{a-1}-2$ for $a\geq 2$, with the convention that the graph is edgeless when this number is negative.  In the two-prime case, it is the maximum of the high-clique and low-independent block endpoints, and therefore it is not generally $q p^{a-1}-2$.  In the square-free three-prime case, for $p<q<r$, the projective dimension remains $qr+p-3$, but the full Betti table differs from the uncorrected one.

\medskip

The purpose of this study is to carry out an analogous program for other zero divisor graphs: (1) $\Gamma(R_m)$ of the	Gaussian integer ring modulo powers of $2$, that is, $R_m = \mathbb Z_{2^{m}}[i] \cong \mathbb Z[i]/(2^{m}),$	where $m$ is a positive integer, and (2) for zero divisor graph $\Gamma(R_{p,c}) $ of the truncated polynomial rings 	$R_{p,c}=\mathbb Z_p[x]/\langle x^c\rangle$, where $p$ is prime and $c\ge2$ is a positive integer. Our goals are the following:\\
(1)  Determine the type sequence of $\Gamma(R_m)$ explicitly and, compute
all Betti numbers of the edge ideal $I(\Gamma(R_m))$, which gives a closed
formulas for the projective dimension and regularity. \\
(2) We study the Cohen--Macaulay property	and the Hilbert series of $S/I(\Gamma(R_m))$.  We show that $S/I(\Gamma(R_m))$
is Cohen--Macaulay only for $m=1$, while for $m\ge2$ the depth is $1$ and the
dimension is at least~$2$.  We further give a combinatorial description of
the Hilbert series of $S/I(\Gamma(R_m))$ in terms of independent sets in
$\Gamma(R_m)$, viewed as the Stanley–Reisner ring of the independence
complex, in the spirit of Stanley \cite{Stanley1996}. A similar type of study is carried for the zero divisor graph $\Gamma(R_{p,c})$.\medskip

\medskip
Several foundational and recent works provide the algebraic and homological background for the present study. Auslander and Buchsbaum established important homological methods in local algebra \cite{AuslanderBuchsbaum1957}, while Matsumura and Stanley developed standard tools from commutative algebra and combinatorics that are frequently used in the study of monomial and edge ideals \cite{Matsumura1989,Stanley1996}. Hochster’s work connected Cohen--Macaulay theory with simplicial complexes, giving fundamental techniques for studying graded Betti numbers \cite{Hochster1977}. In the graph-theoretic and combinatorial direction, Chen studied minimal free resolutions of linear edge ideals \cite{Chen2010}, and Dochtermann analyzed chordal clutters through exposed circuits and linear quotients \cite{Dochtermann2021}. Other recent papers on edge ideals are those of Eisenbud on syzygies \cite{Eisenbud2005}, Fernandez-Ramos and Gimenez on regularity (3) for bipartite graphs \cite{FernandezRamosGimenez2014}, Fröberg on Betti numbers and edge rings \cite{Froberg2022,Froberg2023}, Jacques on Betti numbers of graph ideals \cite{Jacques2004}, Mohammadi and Moradi on resolutions of unmixed bipartite graphs \cite{MohammadiMoradi2015}, Singh and Verma on Betti numbers of split graphs \cite{SinghVerma2020} and Wang on zero-divisor graphs of finite commutative rings \cite{Wang2006}. Banerjee \cite{Banerjee2015} gave results on the regularity of powers of edge ideals. Rather studied the homological invariants of edge ideals of graphs that are Wollastonite, providing more examples of graph ideals and their associated algebraic properties \cite{Rather2026,Rather2024}. Related developments on edge ideals of power graphs of finite groups and comaximal graphs of commutative rings, with particular attention to Betti numbers and linear strands can be seen in \cite{RatherWang2026,RatherImranDiene2024}. The independence polynomial of zero divisor graphs which is related to Hilbert series  is discussed in \cite{bilalac,bilaltcs,bilalsc}.

\medskip

The paper is organized as: Section \ref{sec:preliminaries} fixes notation, recalls zero-divisor graphs, chordal and cochordal graphs, edge ideals, Betti splittings, and the type-sequence formula used later.  Section \ref{sec:type-calculus} develops a corrected type-sequence calculus.  It proves the block-evaluation rules used in all later formulae and gives an algorithm that turns a valid constructible system into Betti numbers. Section \ref{sec:classification} proves the arithmetic classification of cochordal zero-divisor graphs.  The necessity is obtained from induced matching obstructions, while sufficiency follows from explicit corrected constructible systems in the three surviving arithmetic families.
Section  \ref{sec:prime-powers} treats $n=p^a$.  We prove the corrected type sequence, the corrected Betti formula, projective dimension, and several small numerical checks. Section \ref{sec:two-prime} treats $n=p^a q$.  The main point is the mixed nature of the type sequence: high $p$-valuation $q$-classes are clique blocks, whereas low ones are independent repeated blocks.
 Section \ref{sec:three-prime} treats $n=pqr$.  We prove that all three center blocks are independent repeated blocks and derive the corrected closed formula for the whole linear strand.
Section \ref{sec:hilbert-series} derives Hilbert series from the corrected linear resolutions.  It gives a uniform numerator formula, family-specific vertex counts, numerical Hilbert-series examples.
Section \ref{sec:cohen-macaulay} studies Cohen--Macaulayness.  It proves that, among the cochordal residue-class cases, the quotient $S/I(\Gamma(\mathbb{Z}/n\mathbb{Z}))$ is Cohen--Macaulay only in the trivial prime case and in the complete-graph case $n=p^2$.
 Section~\ref{sec:chain-template} develops a uniform template for finite chain rings.  Sections~\ref{sec:gaussian}--\ref{sec:gaussian-hilbert} specialize the template to \(\mathbb{Z}_{2^m}[i]\).  Sections~\ref{sec:poly-ring}--\ref{sec:pc-CM-Hilbert} treat the truncated polynomial rings \(\mathbb{Z}_p[x]/(x^c)\).  The final section summarizes the results and points to possible extensions.
\section{Preliminaries}\label{sec:preliminaries}

Throughout the paper all graphs are finite and simple.  If $G$ is a graph, then $V(G)$ and $E(G)$ denote its vertex set and edge set. For a vertex $v\in V(G)$, the \emph{degree} of $v$ in $G$, denoted by
$\deg_G(v)$, is the number of vertices adjacent to $v$.
The minimum and maximum degrees of $G$ are denoted by
$ \delta(G)=\min\{\deg_G(v):v\in V(G)\} $
and
$
\Delta(G)=\max\{\deg_G(v):v\in V(G)\}.
$
  For a vertex $v\in V(G)$, its open neighborhood is
 $N_G(v)=\{u\in V(G):\{u,v\}\in E(G)\}.$  For a set $A\subseteq V(G)$, we use
 $N_G(A)=\left(\bigcup_{v\in A}N_G(v)\right)\setminus A$ 
for the open neighborhood of $A$, and
 $N_G[A]=A\cup N_G(A)$ 
for its closed neighborhood. Let $G$ be a simple graph with vertex set $V(G)$ and edge set $E(G)$.
 The \emph{complement} of $G$, denoted by $\overline{G}$, is the graph with the same
vertex set $V(\overline{G})=V(G)$, where two distinct vertices $u,v\in V(G)$ are
adjacent in $\overline{G}$ if and only if they are not adjacent in $G$. Equivalently,
\[
E(\overline{G})=\bigl\{\{u,v\}:u,v\in V(G),\ u\ne v,\ \{u,v\}\notin E(G)\bigr\}.
\]
 A subset $A\subseteq V(G)$ is called an \emph{independent set} of $G$ if no two
distinct vertices of $A$ are adjacent in $G$, or equivalently,
 $\{u,v\}\notin E(G)$ for all distinct $u,v\in A.$ A subset $C\subseteq V(G)$ is called a \emph{clique} of $G$ if every two distinct
vertices of $C$ are adjacent in $G$. 
A simple graph $G$ is called \emph{chordal} if every cycle of length at least $4$
has a chord. Equivalently, $G$ is chordal if there is no induced cycle of length
greater than $3$. Here, a \emph{chord} of a cycle is an edge joining two non-consecutive vertices
of the cycle. A simple graph $G$ is called \emph{cochordal} if its complement $\overline{G}$ is
chordal. That is, $G$ is cochordal if every induced cycle of length at least $4$
is absent from $\overline{G}$.
 A subset $U\subseteq V(G)$ is called a \emph{vertex cover} of $G$ if every edge
of $G$ has at least one endpoint in $U$. Equivalently, for every edge
$\{u,v\}\in E(G)$, one has $u\in U$ or $v\in U$. The minimum cardinality of a
vertex cover of $G$ is called the \emph{vertex cover number} of $G$ and is
denoted by $\tau(G)$. The \emph{independence number} of $G$, denoted by $\alpha(G)$, is the maximum
cardinality of an independent set of $G$. Thus
$$
\alpha(G)=\max\{|A|:A\subseteq V(G)\text{ is independent}\}.
$$
 The \emph{clique number} of $G$, denoted by $\omega(G)$, is the maximum
cardinality of a clique of $G$. Thus
$$
\omega(G)=\max\{|C|:C\subseteq V(G)\text{ is a clique}\}.
$$
 A \emph{matching} in $G$ is a set of pairwise disjoint edges, that is, no two
edges in a matching share a common endpoint. The maximum cardinality of a
matching in $G$ is called the \emph{matching number} of $G$ and is denoted by
$\nu(G)$.
 An \emph{induced matching} in $G$ is a matching $M$ such that the subgraph of
$G$ induced by the vertices appearing in the edges of $M$ has edge set exactly
$M$. Equivalently, if $\{u,v\}$ and $\{x,y\}$ are two distinct edges in $M$,
then there is no edge of $G$ joining a vertex of $\{u,v\}$ to a vertex of
$\{x,y\}$. The maximum cardinality of an induced matching in $G$ is called the
\emph{induced matching number} of $G$ and is denoted by $\operatorname{im}(G)$.
A graph $G$ is called a \emph{threshold graph} if there exist real numbers
$w(v)$ for $v\in V(G)$ and a real number $T$ such that, for any two distinct
vertices $u,v\in V(G)$,
$$
\{u,v\}\in E(G)\quad\Longleftrightarrow\quad w(u)+w(v)\ge T.
$$
The numbers $w(v)$ are called threshold weights, and $T$ is called a threshold. 

Let $R$ be a commutative ring with unity, and let $Z(R)$ denote the set of zero
divisors of $R$. The \emph{zero-divisor graph} of $R$, denoted by $\Gamma(R)$,
is the simple graph whose vertex set is the set of nonzero zero divisors of $R$,
that is,
$ V(\Gamma(R))=Z(R)\setminus\{0\}. $
Two distinct vertices $x,y\in V(\Gamma(R))$ are adjacent if and only if their
product is zero in $R$. Equivalently,
$$
\{x,y\}\in E(\Gamma(R))
\quad\Longleftrightarrow\quad
xy=0.
$$
Thus, $\Gamma(R)$ records the multiplicative annihilation relation among the
nonzero zero divisors of $R$. Since loops are not allowed, a vertex $x$ is not
joined to itself, even when $x^2=0$.

The \emph{independence polynomial} of $G$ is the generating function
$$
F_G(y)=\sum_{A\in \operatorname{Ind}(G)} y^{|A|},
$$
where $\operatorname{Ind}(G)$ denotes the set of all independent sets of $G$.
Equivalently, if $f_{r-1}$ denotes the number of independent sets of cardinality
$r$, then
$$
F_G(y)=\sum_{r\ge 0} f_{r-1}y^r.
$$

The \emph{independence complex} of $G$, denoted by $\operatorname{Ind}(G)$, is
the simplicial complex whose faces are precisely the independent sets of $G$.
Thus
$$
\operatorname{Ind}(G)=\{A\subseteq V(G):A\text{ is an independent set of }G\}.
$$

A \emph{facet} of $\operatorname{Ind}(G)$ is a maximal independent set of $G$,
where maximal means maximal with respect to inclusion. A maximum independent set
is an independent set of cardinality $\alpha(G)$. Every maximum independent set
is maximal, but a maximal independent set need not have maximum cardinality.

\medskip

Let $S=\mathbb{K}[x_v:v\in V(G)]$ be a standard graded polynomial ring over a field $\mathbb{K}$.  The edge ideal of $G$ is
$$
I(G)=(x_u x_v:\{u,v\}\in E(G))\subseteq S.
$$
For a finitely generated graded $S$-module $M$, the graded Betti numbers are
$$
\beta_{i,j}^{S}(M)=\dim_{\mathbb{K}}\operatorname{Tor}_i^S(\mathbb{K},M)_j.
$$
We write $\beta_i(M)=\sum_j\beta_{i,j}(M)$, $\operatorname{pd}(M)=\sup\{i:\beta_i(M)\neq 0\}$, and $\operatorname{reg}(M)=\sup\{j-i:\beta_{i,j}(M)\neq 0\}$.

For an edge ideal $I(G)$, the height of $I(G)$ is equal to the vertex cover
number of $G$, we have 
$
\operatorname{height} I(G)=\tau(G).
$
 and $
\dim S/I(G)=\alpha(G).
$
Indeed, the minimal primes of $I(G)$ correspond to the minimal vertex covers of
$G$, while the faces of the Stanley--Reisner complex of $S/I(G)$ are the
independent sets of $G$.
\medskip
A minimal graded free resolution of $S/I$ records the syzygies among the generators of $I$ and encodes the homological invariants of the quotient through its graded Betti numbers.

A minimal graded free resolution of $S/I$ has the form
$$
0\longrightarrow
\bigoplus_j S(-j)^{\beta_{p,j}(S/I)}
\longrightarrow
\cdots
\longrightarrow
\bigoplus_j S(-j)^{\beta_{1,j}(S/I)}
\longrightarrow
S
\longrightarrow
S/I
\longrightarrow 0,
$$
where $p=\operatorname{pd}_S(S/I)$ and $\beta_{r,j}(S/I)$ are the graded Betti numbers.

\medskip

If $I$ is generated in degree $2$ and has a $2$-linear resolution, then the resolution of $S/I$ is
$$
0\longrightarrow
S(-(p+1))^{\beta_p(S/I)}
\longrightarrow
\cdots
\longrightarrow
S(-3)^{\beta_2(S/I)}
\longrightarrow
S(-2)^{\beta_1(S/I)}
\longrightarrow
S
\longrightarrow
S/I
\longrightarrow 0.
$$
Equivalently, the minimal graded free resolution of $I$ itself is
{\footnotesize $$
0\longrightarrow
S(-(p+2))^{\beta_{p+1}(S/I)}
\longrightarrow
\cdots
\longrightarrow
S(-4)^{\beta_3(S/I)}
\longrightarrow
S(-3)^{\beta_2(S/I)}
\longrightarrow
S(-2)^{\beta_1(S/I)}
\longrightarrow
I
\longrightarrow 0.
$$} 
The ideal $I$ has a $2$-linear resolution if all its nonzero graded Betti numbers occur in degree $r+2$, equivalently $\beta_{r,j}(I)=0$ for $j\ne r+2$. For the quotient $S/I$, this corresponds to nonzero Betti numbers only in degree $r+1$.
\medskip

For a positive integer $n$, the zero-divisor graph $\Gamma(\mathbb{Z}/n\mathbb{Z})$ has as vertices the nonzero zero-divisors of $\mathbb{Z}/n\mathbb{Z}$.  Two distinct vertices $x,y$ are adjacent when $xy\equiv 0\pmod n$.  Equivalently, if representatives are chosen in $\{1,\ldots,n-1\}$, then $\{x,y\}$ is an edge exactly when $n\mid xy$.

\medskip

If $R$ is a commutative ring and $x\in R$, then $\operatorname{ann}(x)=\{r\in R:rx=0\}$.  The compressed zero-divisor graph $\Gamma_E(R)$ has vertices as the equivalence classes of nonzero zero-divisors with respect to equality of annihilators, and two classes are adjacent if the product of their representatives is zero.  This graph is useful because, for $R=\mathbb{Z}/n\mathbb{Z}$, the classes are described by the divisor type of $\gcd(x,n)$.

\medskip

For $n=p^a$, the nonzero zero-divisors are partitioned into classes
$$
V_i=\{x\in \mathbb{Z}/n\mathbb{Z}:[x]=[p^i]\},\qquad 1\leq i\leq a-1,
$$
and $|V_i|=\varphi(p^{a-i})=p^{a-i-1}(p-1)$.  If $x\in V_i$ and $y\in V_j$, then $x$ and $y$ are adjacent exactly when $i+j\geq a$.

\medskip

For $n=p^a q$, where $p$ and $q$ are distinct primes, we use the classes
$$
A_t=\{x:[x]=[p^t]\},\quad 1\leq t\leq a,
\qquad
B_i=\{x:[x]=[p^i q]\},\quad 0\leq i\leq a-1.
$$
The sizes are $|A_a|=q-1$, $|A_t|=p^{a-t-1}(p-1)(q-1)$ for $1\leq t\leq a-1$, and $|B_i|=p^{a-i-1}(p-1)$.  The edges satisfy $A_tA_u\notin E$, while $A_tB_i\in E$ exactly when $t+i\geq a$, and $B_iB_j\in E$ exactly when $i+j\geq a$ for distinct vertices.

\medskip

For $n=pqr$ with distinct primes, we use the classes represented by $p,q,r,pq,pr,qr$.  Edges occur precisely when the product of the two labels contains all three primes.  In particular, the classes represented by $pq$, $pr$, and $qr$ are independent classes, even though different such classes are mutually adjacent.

A \textit{valuation class} is the set of residue classes in $\mathbb{Z}/n\mathbb{Z}$ having the same divisibility pattern with respect to the prime powers dividing $n$. If $n=p_1^{a_1}p_2^{a_2}\cdots p_t^{a_t},$ then for a nonzero residue $x\in \mathbb{Z}/n\mathbb{Z}$, define the $p_i$-adic valuation truncated at $a_i$ by
 $$
v_{p_i}(x)=\max\{e: p_i^e\mid x\},\qquad 0\leq e\leq a_i,
$$
where $v_{p_i}(0)$ is treated as $a_i$ modulo $p_i^{a_i}$. Then the valuation vector of $x$ is
 $$
\nu(x)=\bigl(v_{p_1}(x),\ldots,v_{p_t}(x)\bigr).
$$
The valuation class corresponding to a vector $\alpha=(\alpha_1,\ldots,\alpha_t)$ is
$$
C_\alpha=\{x\in \mathbb{Z}/n\mathbb{Z}: \nu(x)=\alpha\}.
$$
So two elements $x,y$ lie in the same valuation class exactly when they are divisible by the same powers of each prime divisor of $n$. For example, if $n=p^aq$, the relevant classes include
$
C_{(i,0)}=\{x:v_p(x)=i,\ v_q(x)=0\},
$
and
$
C_{(i,1)}=\{x:v_p(x)=i,\ v_q(x)=1\}.
$
 in our case, these correspond roughly to classes represented by divisors such as $[p^i]$ and $[p^iq]$.

The reason valuation classes are useful is that adjacency in the zero-divisor graph can be read directly from valuations. In $\Gamma(\mathbb{Z}/n\mathbb{Z})$, two nonzero zero divisors $x,y$ are adjacent precisely when
$ xy\equiv 0 \pmod n, $ which is equivalent to
$$
v_{p_i}(x)+v_{p_i}(y)\geq a_i
\quad\text{for every }i.
$$
Thus, all vertices in the same valuation class have the same external adjacency behavior. This makes valuation classes a natural way to organize the zero-divisor graph.

\medskip

A substantial body of work examines the construction and properties of these algebraic invariants, for example, see \cite{FernandezRamosGimenez2014, Froberg1990, Jacques2004, Villarreal2015}.  
Fat forests were studied by Fröberg \cite{Froberg2022}, who described their Betti numbers and Alexander duals.  
Minimal free resolutions for unmixed bipartite graphs were established by Mohammadi and Moradi \cite{MohammadiMoradi2015}.  
Several families of split graphs have been analyzed in \cite{SinghVerma2020}.
In a related direction, Corso and Nagel \cite{CorsoNagel2008,CorsoNagel2009} proved that Ferrers graphs have edge ideals with $2$-linear resolutions, and Chen \cite{Chen2010} gave an explicit description of the minimal $2$-linear resolutions for all such ideals.  
Finally, Fröberg \cite{Froberg2023} explored several conjectures about $2$-linear resolutions in edge rings.

\medskip

We now recall the external results used later.  The following known result characterizes graphs whose edge ideals have a linear resolution.
\begin{theorem}[Fröberg, see \cite{Froberg1990}]\label{thm:froberg}
For a finite simple graph $G$, the edge ideal $I(G)$ has a $2$-linear resolution if and only if $\overline{G}$ is chordal.  Equivalently, $G$ is cochordal if and only if $I(G)$ has a $2$-linear resolution.
\end{theorem}

\medskip

The following known form of Dirac's theorem is used only through its consequence for chordal elimination.
\begin{theorem}[Dirac, see \cite{Dirac1961,HerzogHibiZheng2004}]\label{thm:dirac}
Every noncomplete chordal graph has at least two nonadjacent simplicial vertices.  Consequently, chordal graphs admit perfect elimination orderings.
\end{theorem}

\medskip

The following result is the standard Betti-splitting principle for monomial ideals.
\begin{theorem}[Francisco--H\`a--Van Tuyl, see \cite{FranciscoHaVanTuyl2009}]\label{thm:betti-splitting-known}
Let $I=J+K$ be a splitting of a monomial ideal satisfying the usual splitting-function hypotheses.  Then, for all $i,j$,
$$
\beta_{i,j}(I)=\beta_{i,j}(J)+\beta_{i,j}(K)+\beta_{i-1,j}(J\cap K).
$$
In particular, the total Betti numbers satisfy the corresponding ungraded equality.
\end{theorem}

\medskip

The following known specialized splitting result is the form needed for quadratic edge ideals.
\begin{lemma}[Nguyen--Vu, see \cite{NguyenVu2016}]\label{lem:x-partition}
Let $I$ be a quadratic monomial ideal and let $x$ be a variable.  If $I=xP+J$ is the $x$-partition of $I$, where $P$ is generated by variables and $J$ is generated by the quadratic monomials not divisible by $x$, then this decomposition is a Betti splitting whenever the hypotheses of the quadratic splitting criterion hold; in particular, it applies in the cochordal elimination steps used below.
\end{lemma}

\medskip

The following known observation identifies a useful colon ideal at a cochordal elimination step.
\begin{proposition}[Jaramillo--Villarreal, see \cite{JaramilloVillarreal2021}]\label{prop:colon-prime-known}
If $G$ is cochordal, then the edge ideal $I(G)$ admits an elimination step $I(G)=xP+J$ in which $P=I(G):x$ is a monomial prime generated by variables corresponding to a vertex cover of the remaining graph.
\end{proposition}

\medskip

The following known lower bound explains why induced matchings obstruct cochordality.
\begin{theorem}[H\`a--Van Tuyl and Woodroofe, see \cite{HaVanTuyl2008,Woodroofe2014}]\label{thm:induced-matching-bound}
If $G$ is a finite simple graph, then the induced matching number of $G$ gives a lower bound on the regularity of $I(G)$.  In particular, if $I(G)$ has a $2$-linear resolution, then $G$ has no induced matching of size two.
\end{theorem}

\medskip

We shall use the following constructible-system formulation.  It is essentially the cochordal construction of \cite{DungVu2026}, but the proof below keeps the global subtraction term visible.

\medskip

\noindent The following definition records the construction of a graph by repeatedly adding a star whose leaf set is a vertex cover of the graph already built.
\begin{definition}[Cochordal constructible system, \cite{DungVu2026}]\label{def:constructible}
Let $P=(u_k,U_k,u_{k-1},U_{k-1},\ldots,u_1,U_1)$ be an ordered list such that $u_j\notin U_{\ell}$ whenever $\ell\leq j$.  Set $G_1=K_{\{u_1\},U_1}$ and, recursively,
$$
V(G_j)=V(G_{j-1})\cup \{u_j\}\cup U_j,
\qquad
E(G_j)=E(G_{j-1})\cup E(K_{\{u_j\},U_j}).
$$
The list $P$ is a cochordal constructible system for $G$ if $G=G_k$ and $U_j$ is a vertex cover of $G_{j-1}$ for $j\geq 2$.  Its type is the sequence $(a_k,\ldots,a_1)$, where $a_j=|U_j|$.
\end{definition}

\medskip

\noindent The following known theorem links constructible systems with cochordal graphs.
\begin{theorem}[Constructible characterization, see \cite{DungVu2026}]\label{thm:constructible-characterization}
A finite simple graph is cochordal if and only if it admits a cochordal constructible system.
\end{theorem}

\medskip

\noindent The following formula is the central Betti-number identity used in the paper \cite{DungVu2026}.
\begin{theorem}[Type-sequence Betti formula, see \cite{DungVu2026}]\label{thm:type-formula}
Let $G$ be a cochordal graph with a constructible system of type $(a_k,\ldots,a_1)$.  Then $S/I(G)$ has a linear resolution and, for every $i\geq 1$,
$$
\beta_{i,i+1}(S/I(G))=\beta_i(S/I(G))=\binom{a_k}{i}+\binom{a_{k-1}+1}{i}+\cdots+\binom{a_1+k-1}{i}-\binom{k}{i+1}.
$$
Consequently,
$$
\operatorname{pd}(S/I(G))=\max\{a_k,a_{k-1}+1,\ldots,a_1+k-1\}.
$$
\end{theorem}

\medskip

The literature supplies the construction principle and the general type formula, but it does not automatically validate a proposed arithmetic type sequence.  The point that must be checked is whether every vertex placed in a cover is truly adjacent to the new center.  In $n\in {p^a q, pqr}$ case, some valuation classes in $\Gamma(\mathbb{Z}/n\mathbb{Z})$ are independent rather than cliques, and including earlier vertices from the same class creates false edges.  Moreover, the term $-\binom{k}{i+1}$ in Theorem \ref{thm:type-formula} is essential and cannot be dropped when evaluating arithmetic blocks. Furthermore, we extend theory for new classes of zero divisor graphs.

\medskip

\section{Corrected type calculus and diagnostic principles}\label{sec:type-calculus}

The results in this section do not change the abstract constructible-system (Theorem \ref{thm:constructible-characterization}, \cite{DungVu2026}). Instead, we give the corrected arithmetic evaluation rules needed for zero-divisor graphs.  The distinction between decreasing clique blocks and constant independent blocks is absent from the formulae in \cite{DungVu2026}.

\medskip

Before applying the type formula form Definition \ref{def:constructible}, we shall often group the type sequence
$(a_k,\ldots,a_1)$ into consecutive blocks.  If $(a_k,\ldots,a_1)=(T_s,T_{s-1},\ldots,T_1),$ then each $T_i$ is called a \emph{type block}.  Thus $T_i$ is  only a convenient notation for the
subsequence of type entries arising from one valuation layer or one arithmetic
class.  If the vertices in a layer $V_i$ are ordered as
$V_i=\{v_{i,1},\ldots,v_{i,n_i}\}$ with corresponding cover sets
$U_{i,1},\ldots,U_{i,n_i}$, then the associated type block is
\[
T_i=(|U_{i,n_i}|,|U_{i,n_i-1}|,\ldots,|U_{i,1}|),
\]
because the type sequence is written in reverse order from the construction
order.

\medskip

\noindent The following result gives the corrected evaluation of a constructible type sequence in block form.
\begin{theorem}\label{prop:block-evaluation}
Let a valid constructible system have type obtained by concatenating blocks $T_s,\ldots,T_1$.  Suppose block $T_h$ has length $\ell_h$, and let $K_{>h}=\ell_s+\cdots+\ell_{h+1}$.  If the entries of $T_h$ are $c_h,c_h-1,\ldots,c_h-\ell_h+1$, then its contribution to $\beta_i(S/I(G))$ before the global correction is
$ \ell_h\binom{c_h+K_{>h}}{i}. $
If the entries of $T_h$ are all equal to $d_h$, then its contribution before the global correction is
$$
\sum_{t=0}^{\ell_h-1}\binom{d_h+K_{>h}+t}{i}
=\binom{d_h+K_{>h}+\ell_h}{i+1}-\binom{d_h+K_{>h}}{i+1}.
$$
The total Betti number is the sum of all block contributions minus $\binom{k}{i+1}$, where $k=\ell_s+\cdots+\ell_1$.
\end{theorem}

\begin{proof}
	Let the type sequence of the constructible system be
	$ (a_k,a_{k-1},\ldots,a_1), $
	where $k=\ell_s+\ell_{s-1}+\cdots+\ell_1$.  By Theorem
	\ref{thm:type-formula}, the uncorrected part of the $i$-th Betti number is
	$ \sum_{j=1}^{k}\binom{a_j+k-j}{i}. $
	Equivalently, if we read the type sequence from left to right, the entry in
	the first position contributes with shift $0$, the entry in the second
	position contributes with shift $1$, and so on.  Thus, if the entry in the
	$t$-th position from the left is denoted by $b_t$, where
	$0\leq t\leq k-1$, then its contribution before the global correction is
	$ \binom{b_t+t}{i}. $
	This is the same formula as
	$ \binom{a_j+k-j}{i}, $ since the entry $a_j$ occupies the position $t=k-j$ from the left.
	
	Now fix a block $T_h$.  Since the full type sequence is obtained by
	concatenating the blocks
	$ T_s,T_{s-1},\ldots,T_1, $
	the number of entries lying strictly to the left of $T_h$ is
	$ K_{>h}=\ell_s+\ell_{s-1}+\cdots+\ell_{h+1}. $
	Therefore the entries of $T_h$ occupy the left-to-right positions
	$ K_{>h},\ K_{>h}+1,\ \ldots,\ K_{>h}+\ell_h-1. $
	 Suppose that
	$ T_h=(c_h,c_h-1,\ldots,c_h-\ell_h+1). $
	The $r$-th entry of this block, counted from $r=0$ to $r=\ell_h-1$, is
	$ c_h-r, $
	and its position from the left is
	$ K_{>h}+r. $
	Hence its contribution before the global correction is
	$$
	\binom{(c_h-r)+(K_{>h}+r)}{i}
	=
	\binom{c_h+K_{>h}}{i}.
	$$
	The dependence on $r$ cancels, as the entry decreases by $1$ at each step,
	while the positional shift increases by $1$ at each step.  Since there are
	$\ell_h$ entries in the block, the whole block contributes
	$$
	\sum_{r=0}^{\ell_h-1}\binom{c_h+K_{>h}}{i}
	=
	\ell_h\binom{c_h+K_{>h}}{i}.
	$$
	
	 Next suppose that
	$ T_h=(d_h,d_h,\ldots,d_h) $ is a constant block of length $\ell_h$.  The $r$-th entry of this block is
	$d_h$, and again its position from the left is $K_{>h}+r$.  Therefore its
	contribution is $
	\binom{d_h+K_{>h}+r}{i}. $
	Now, summing over all entries of the block, we have
	$$
	\sum_{r=0}^{\ell_h-1}\binom{d_h+K_{>h}+r}{i}.
	$$
	Let $
	N=d_h+K_{>h}. $
	Then this sum becomes
	$$
	\sum_{r=0}^{\ell_h-1}\binom{N+r}{i}
	=
	\sum_{m=N}^{N+\ell_h-1}\binom{m}{i}.
	$$
	By the hockey-stick identity, we have
	$$
	\sum_{m=N}^{N+\ell_h-1}\binom{m}{i}
	=
	\binom{N+\ell_h}{i+1}-\binom{N}{i+1}.
	$$
	Substituting back $N=d_h+K_{>h}$, we derive
	$$
	\sum_{t=0}^{\ell_h-1}\binom{d_h+K_{>h}+t}{i}
	=
	\binom{d_h+K_{>h}+\ell_h}{i+1}
	-
	\binom{d_h+K_{>h}}{i+1}.
	$$
	 Finally, Theorem \ref{thm:type-formula} states that the complete Betti
	number is obtained by subtracting the global correction term
	$ \binom{k}{i+1}, $
	where
	$ k=\ell_s+\ell_{s-1}+\cdots+\ell_1. $
	Hence the total value of $\beta_i(S/I(G))$ is the sum of the block
	contributions described above, minus $\binom{k}{i+1}$.  This proves the
	claim.
\end{proof}

\medskip

The above proposition is new in its use here, as it separates the two arithmetic phenomena that were previously conflated.  Ferrers graphs studied by Corso and Nagel \cite{CorsoNagel2009,CorsoNagel2008} naturally lead to decreasing blocks. Zero-divisor graphs also produce constant blocks when a valuation class is independent.

\medskip

Let $n=\prod_{t=1}^s p_t^{\alpha_t}$ and let $R=\mathbb{Z}/n\mathbb{Z}$.  For a residue $x$, define
$ \nu_t(x)=\min\{v_{p_t}(x),\alpha_t\}. $
Thus $\nu_t(x)=\alpha_t$ means that $p_t^{\alpha_t}$ divides the chosen representative of $x$.

\medskip

\noindent The following lemma identifies the ordinary zero-divisor vertices with  annihilators  and valuation classes.
\begin{lemma}\label{lem:valuation-annihilator}
	Let $x,y\in \mathbb{Z}/n\mathbb{Z}$ be nonzero zero divisors.  Then $\operatorname{ann}(x)=\operatorname{ann}(y)$ if and only if $\nu_t(x)=\nu_t(y)$ for every $t$.  Moreover, the number of residues in the valuation class $\mathbf e=(e_1,\ldots,e_s)$ is $|C(\mathbf e)|=\prod_{t=1}^s c_t(e_t),$ where
	$$
	 c_t(e_t)=
	\begin{cases}
		p_t^{\alpha_t-e_t-1}(p_t-1),&0\leq e_t<\alpha_t,\\
		1,&e_t=\alpha_t.
	\end{cases}
	$$
	The all-maximal vector $(\alpha_1,\ldots,\alpha_s)$ corresponds to the zero residue and is omitted from the vertex set.
\end{lemma}

\begin{proof}
	Choose integer representatives of $x$ and $y$, with 
	$ n=\prod_{t=1}^{s}p_t^{\alpha_t}. $
	For a residue $x$, let
	$ d_x=\gcd(x,n). $
	Then
	$ d_x=\prod_{t=1}^{s}p_t^{\nu_t(x)}. $
	Clearly, the exponent of $p_t$ in $\gcd(x,n)$ is exactly the truncated
	$p_t$-adic valuation of $x$ modulo $n$. We first determine the annihilator of $x$.  A residue $z\in \mathbb{Z}/n\mathbb{Z}$ belongs
	to $\operatorname{ann}(x)$ if and only if
	 $xz\equiv 0 \pmod n, $
	or equivalently,
	$ n\mid xz,$  which is equivalent to
	$ \frac{n}{d_x}\mid z. $
	Hence
	$ \operatorname{ann}(x)=\left(\frac{n}{d_x}\right)
	\subseteq \mathbb{Z}/n\mathbb{Z}. $
	Therefore $\operatorname{ann}(x)$ is determined by $d_x$, and $d_x$ is determined by the
	valuation vector
	$ \nu(x)=(\nu_1(x),\ldots,\nu_s(x)). $
	Consequently, if $\nu_t(x)=\nu_t(y)$ for every $t$, then $d_x=d_y$, and hence
	$\operatorname{ann}(x)=\operatorname{ann}(y)$.
	
	 Conversely, suppose that $\operatorname{ann}(x)=\operatorname{ann}(y)$.  Since
	$ \operatorname{ann}(x)=\left(\frac{n}{d_x}\right)$ and $ \operatorname{ann}(y)=\left(\frac{n}{d_y}\right), $ so
	the two principal ideals generated by $n/d_x$ and $n/d_y$ in $\mathbb{Z}/n\mathbb{Z}$ are
	equal.  In $\mathbb{Z}/n\mathbb{Z}$, ideals are determined by divisors of $n$, so we have 
	$ \frac{n}{d_x}=\frac{n}{d_y}, $
	and hence $d_x=d_y$.  Comparing the prime-power decompositions of $d_x$ and
	$d_y$ gives
	$ \nu_t(x)=\nu_t(y) $ for every $t. $
	This proves the equivalence between equality of annihilators and equality of
	valuation vectors.
	
	 For counting the size of a valuation class, we use Chinese remainder theorem,
	$ \mathbb{Z}/n\mathbb{Z}\cong \prod_{t=1}^{s}\mathbb{Z}/p_t^{\alpha_t}\mathbb{Z}.$
	Thus the choices in different prime-power components are independent.  In the
	component $\mathbb{Z}/p_t^{\alpha_t}\mathbb{Z}$, the number of residues with exact
	valuation $e_t<\alpha_t$ is the number of multiples of $p_t^{e_t}$ which are
	not multiples of $p_t^{e_t+1}$, that is, such a number is
	$$
	p_t^{\alpha_t-e_t}-p_t^{\alpha_t-e_t-1}
	=
	p_t^{\alpha_t-e_t-1}(p_t-1).
	$$
	If $e_t=\alpha_t$, then there is only one residue with that valuation, namely
	the zero residue modulo $p_t^{\alpha_t}$. Therefore, for
	$ \mathbf e=(e_1,\ldots,e_s), $
	the number of residues with valuation vector $\mathbf e$ is
	$ |C(\mathbf e)|
	=
	\prod_{t=1}^{s}c_t(e_t), $
	where
	$$
	c_t(e_t)=
	\begin{cases}
		p_t^{\alpha_t-e_t-1}(p_t-1),&0\leq e_t<\alpha_t,\\
		1,&e_t=\alpha_t.
	\end{cases}
	$$
	Finally, the vector
	$ (\alpha_1,\ldots,\alpha_s) $
	means that every prime-power component is zero, and hence it represents the
	zero residue modulo $n$.  Since the zero-divisor graph uses only nonzero zero
	divisors as vertices, this class is omitted from the vertex set.
\end{proof}

\medskip

The above lemma explains why the same notation can be used for compressed classes and for ordinary vertex blocks.  The homological calculation, however, must use the size and internal adjacency of the lifted block.

\medskip

\noindent The following result gives a quick test for whether same-class vertices may be inserted into a cover.
\begin{lemma}\label{lem:same-class-test}
Let $n=\prod_{t=1}^s p_t^{\alpha_t}$, and let $C(\mathbf e)$ be the class of residues whose representative has valuation vector $\mathbf e=(e_1,\ldots,e_s)$ with respect to the primes $p_t$.  Two distinct vertices in $C(\mathbf e)$ are adjacent in $\Gamma(\mathbb{Z}/n\mathbb{Z})$ only if $2e_t\geq \alpha_t$ for every $t$.
\end{lemma}

\begin{proof}
	Choose two distinct vertices $x,y\in C(\mathbf e)$.  By the definition of the
	valuation class, both $x$ and $y$ have the same valuation vector
	$ \mathbf e=(e_1,\ldots,e_s), $
	that is,
	$ v_{p_t}(x)=e_t$ and $	v_{p_t}(y)=e_t
	$
	for every $t=1,\ldots,s$, where the valuations are understood in the usual
	truncated sense modulo $p_t^{\alpha_t}$. 
	As
	$
	n=\prod_{t=1}^{s}p_t^{\alpha_t},
	$
	the congruence of adjacency in $\Gamma(\mathbb{Z}/n\mathbb{Z})$ is equivalent to the divisibility condition
	$ p_t^{\alpha_t}\mid xy $ for every $t=1,\ldots,s. $
	In terms of $p_t$-adic valuations, this means
	$ v_{p_t}(xy)\geq \alpha_t $ for every $t. $ Now fix an index $t$.  If $e_t<\alpha_t$, then the representatives of $x$ and
	$y$ may be written in the form
	$ x=p_t^{e_t}u,$ and $ y=p_t^{e_t}v, $
	where neither $u$ nor $v$ is divisible by $p_t$.  Hence,
	$ xy=p_t^{2e_t}uv, $
	with $p_t\nmid uv$, and 
	$ v_{p_t}(xy)=2e_t. $
	Thus, in order that $p_t^{\alpha_t}$ divide $xy$, it is necessary that
	$ 2e_t\geq \alpha_t. $ If $e_t=\alpha_t$, then the same inequality is automatic.  Consequently, for
	every prime factor $p_t$ of $n$, adjacency of two distinct vertices in the
	same valuation class forces
	$ 2e_t\geq \alpha_t. $
\end{proof}

\medskip

Compared with the compressed-graph viewpoint of Mulay \cite{Mulay2002} and Spiroff--Wickham \cite{SpiroffWickham2011}, Lemma \ref{lem:same-class-test} is deliberately elementary.  Its role is to prevent a compressed adjacency from being mistakenly interpreted as a clique inside a lifted equivalence class.

\medskip

\noindent The following result identifies precisely when a repeated block is decreasing or constant in the families studied later.
\begin{corollary}\label{cor:clique-independent-block}
In $\Gamma(\mathbb{Z}/p^a\mathbb{Z})$, the class $[p^i]$ is a clique if and only if $2i\geq a$.  In $\Gamma(\mathbb{Z}/p^a q\mathbb{Z})$, the class $[p^i q]$ is a clique if and only if $2i\geq a$.  In $\Gamma(\mathbb{Z}/pqr\mathbb{Z})$, the classes $[pq]$, $[pr]$, and $[qr]$ are independent sets.
\end{corollary}

\begin{proof}
By Lemma \ref{lem:same-class-test}, for $p^a$ the condition is $2i\geq a$.  For $p^a q$, the $q$-exponent in $[p^i q]$ already satisfies $2\geq 1$, so the only condition is again $2i\geq a$.  For $pqr$, the class $[pq]$ has valuation vector $(1,1,0)$, whose square still has $r$-valuation zero, and the other two cases are identical by symmetry.
\end{proof}

\medskip

The above corollary is the first point at which the present  paper departs from the invalid type sequences in the \cite{DungVu2026}.  It implies, for example, that two distinct vertices in the class $[q]$ of $\Gamma(\mathbb{Z}/pq\mathbb{Z})$ are not adjacent, so a cover for one such vertex cannot contain another such vertex.

\medskip

\noindent The following theorem gives a corrected validation rule for proposed constructible systems in zero-divisor graphs.
\begin{theorem}\label{thm:validity-criterion}
Consider a proposed constructible step in $\Gamma(\mathbb{Z}/n\mathbb{Z})$ with center in valuation class $C(\mathbf e)$ and cover equal to a union of previously available valuation classes and possibly previous vertices from $C(\mathbf e)$.  The step is valid only if every included class $C(\mathbf f)$ satisfies $e_t+f_t\geq \alpha_t$ for all $t$, and previous vertices from $C(\mathbf e)$ may be included only if $2e_t\geq \alpha_t$ for all $t$.
\end{theorem}
\begin{proof}
	Let $
	n=p_1^{\alpha_1}\cdots p_m^{\alpha_m}, $
	and let the chosen center be a vertex $x\in C(\mathbf e)$, where
	$\mathbf e=(e_1,\ldots,e_m)$.  By definition of the valuation class, this means
	that $
	v_{p_t}(x)=e_t $ or every $t=1,\ldots,m. $
	
	 In a constructible step with center $x$, one attaches the complete bipartite
	piece between $x$ and the prescribed cover.  Thus, for each vertex $y$ placed
	in the cover, the construction asserts that $\{x,y\}$ is an edge of
	$\Gamma(\mathbb{Z}/n\mathbb{Z})$.  Since the zero-divisor graph contains an edge precisely
	when the product is zero modulo $n$, this forces
	$ xy\equiv 0 \pmod n,$ or  equivalently,
	$
	n\mid xy.
	$
	 
	 Now suppose that $y\in C(\mathbf f)$, where
	$\mathbf f=(f_1,\ldots,f_m)$.  Then
	$ v_{p_t}(y)=f_t $ for every $t.$ 	The divisibility condition $n\mid xy$ is equivalent to saying that each prime
	power $p_t^{\alpha_t}$ divides $xy$.  In valuation form this is
	$ v_{p_t}(xy)\geq \alpha_t$  for every$t.$
	Since valuations add under multiplication, we have
	$$
	v_{p_t}(xy)=v_{p_t}(x)+v_{p_t}(y)=e_t+f_t.
	$$
	Therefore, every class $C(\mathbf f)$ included in the cover must satisfy
	$ e_t+f_t\geq \alpha_t $  for all $t=1,\ldots,m. $
	
	 It remains only to discuss the case where the cover contains earlier vertices
	from the same valuation class as the center.  If $y$ is another vertex of
	$C(\mathbf e)$, then
	$ v_{p_t}(y)=e_t $ for every $t. $
	Thus the edge condition between $x$ and $y$ becomes
	$$
	v_{p_t}(xy)=e_t+e_t=2e_t\geq \alpha_t
	\qquad\text{for every }t.
	$$
	Hence previous vertices from $C(\mathbf e)$ can be included in the cover only
	when
	$ 2e_t\geq \alpha_t$  for all $t. $ This is exactly the same-class adjacency criterion stated in Lemma
	\ref{lem:same-class-test}. Consequently, a proposed constructible step cannot be valid unless all
	valuation classes placed in its cover are genuine neighbor classes of the
	center, and unless any same-class vertices included in the cover are adjacent
	to the center.  These are precisely the two stated conditions.
\end{proof}

\medskip

The above theorem is a diagnostic tool.  It does not replace chordal-elimination theory. It checks whether the arithmetic data used in such an elimination is faithful to the graph before any Betti computation is attempted.

\medskip

\noindent The following corollary records the corrected global subtraction in a form useful for examples.
\begin{corollary}\label{cor:global-correction}
Let a valid constructible system have $k$ center steps.  Any formula obtained by summing only the shifted cover-size terms overcounts the $i$th total Betti number by $\binom{k}{i+1}$.
\end{corollary}

\begin{proof}
This is immediate from Theorem \ref{thm:type-formula}.  The term $\binom{k}{i+1}$ depends only on the number of center steps and is independent of the cover sizes.  It is therefore still present after all arithmetic simplifications.
\end{proof}

\medskip

The correction is already visible for a single edge.  The graph $\Gamma(\mathbb{Z}/9\mathbb{Z})$ is $K_2$, so $I(G)$ is principal and $\beta_1(S/I(G))=1$.  The raw shifted block sum gives $2$, and the missing term $\binom{2}{2}=1$ restores the correct value.

\medskip

\begin{example}\label{ex:diagnostic}
Let $n=15$.  The class $[5]$ has two vertices, namely $5$ and $10$, and $(5)(10)=50\not\equiv 0\pmod {15}$.  Thus $[5]$ is independent.  The graph is the complete bipartite graph $K_{4,2}$ with parts $[3]$ and $[5]$, see Figure \ref{fig:z15}.  A valid construction may use the two $[5]$-vertices as centers, but each cover has only the four vertices in $[3]$, and the second cover cannot include the first $[5]$-vertex.
\end{example}

\medskip

The example demonstrates the two errors simultaneously: same-class vertices from $[5]$ cannot be inserted, and the two center steps still produce the global subtraction $\binom{2}{i+1}$.

\medskip

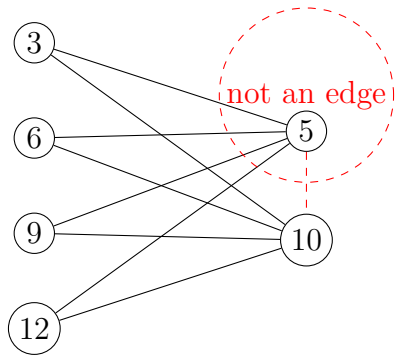
\begin{figure}[h]
\centering
\begin{tikzpicture}[scale=0.9, every node/.style={circle,draw,inner sep=2pt,minimum size=5mm}]
\node (a1) at (0,2.1) {$3$};
\node (a2) at (0,0.7) {$6$};
\node (a3) at (0,-0.7) {$9$};
\node (a4) at (0,-2.1) {$12$};
\node (b1) at (4,0.8) {$5$};
\node (b2) at (4,-0.8) {$10$};
\foreach \a in {a1,a2,a3,a4}{\draw (\a)--(b1); \draw (\a)--(b2);}
\draw[dashed,red] (b1)--(b2) node[midway,above] {not an edge};
\end{tikzpicture}
\caption{The graph $\Gamma(\mathbb{Z}/15\mathbb{Z})=K_{4,2}$.  The dashed segment marks the false same-class edge that an invalid constructible cover would create.}
\label{fig:z15}
\end{figure}

\medskip

Figure \ref{fig:z15} shows that compressed classes must be lifted with care.  A class can contribute many vertices without becoming a clique in the ordinary zero-divisor graph.

\medskip

\begin{table}[H]
	\begin{center}
\begin{tabular}{cccc}
\toprule
Graph & valid type & raw shifted sum for $\beta_1$ & corrected $\beta_1$\\
\midrule
$\Gamma(\mathbb{Z}/9\mathbb{Z})$ & $(1,0)$ & $2$ & $2-\binom{2}{2}=1$\\[2mm]
$\Gamma(\mathbb{Z}/15\mathbb{Z})$ & $(4,4)$ & $9$ & $9-\binom{2}{2}=8$\\
$\Gamma(\mathbb{Z}/8\mathbb{Z})$ & $(2)$ & $2$ & $2$\\
\bottomrule
\end{tabular}
\end{center}
\label{tab 1}
\caption{Raw block contribution with the corrected first Betti number}
\end{table}
\medskip

Table \ref{tab 1} compares the raw block contribution with the corrected first Betti number.  The corrected value equals the actual number of edges in each graph.

\section{Classification of cochordal zero-divisor graphs}\label{sec:classification}

The classification below agrees with the recent classification in \cite{DungVu2026}, but the proof is rewritten to avoid an invalid induced-matching choice in the two-prime obstruction and to make the dependence on induced matching number explicit.

\medskip

\noindent The following theorem gives the complete arithmetic classification.
\begin{theorem}\label{thm:classification}
Let $n\geq 2$.  Then $\Gamma(\mathbb{Z}/n\mathbb{Z})$ is cochordal if and only if $n$ has one of the following forms:
\begin{enumerate}[label=\textup{(\roman*)}]
\item $n=p^a$;
\item $n=p^a q$;
\item $n=pqr$;
\end{enumerate}
where $p,q,r$ are distinct primes and $a\geq 1$.
\end{theorem}

\begin{proof}
Assume first that $\Gamma(\mathbb{Z}/n\mathbb{Z})$ is cochordal.  By Theorem \ref{thm:froberg}, the edge ideal has a $2$-linear resolution, and by Theorem \ref{thm:induced-matching-bound} the graph has no induced matching of size two.

\medskip

If $n$ has at least four distinct prime divisors, write $n=p_1^{\alpha_1}p_2^{\alpha_2}p_3^{\alpha_3}p_4^{\alpha_4}m$, where $m$ is divisible by no $p_i$ necessarily after absorbing the remaining prime powers. Let $x=p_1^{\alpha_1}p_2^{\alpha_2}m$, $y=p_3^{\alpha_3}p_4^{\alpha_4}m$, $u=p_1^{\alpha_1}p_3^{\alpha_3}m$, and $v=p_2^{\alpha_2}p_4^{\alpha_4}m$.  Then $xy$ and $uv$ are divisible by $n$, while the four cross products miss one of the prime powers $p_i^{\alpha_i}$.  Hence $\{x,y\}$ and $\{u,v\}$ form an induced matching, a contradiction.

\medskip

If $n$ has exactly three distinct prime divisors and some exponent exceeds one, say $n=p^a q^b r^c$ with $a>1$, choose $x=p^a$, $y=q^b r^c$, $u=p^{a-1}q^b$, and $v=pr^c$.  Then $xy$ and $uv$ are zero modulo $n$, while every cross product misses one required prime power.  This again gives an induced matching of size two.  Therefore a cochordal case with three prime divisors must be square-free, namely $pqr$.

\medskip

If $n$ has exactly two distinct prime divisors and both exponents exceed one, write $n=p^a q^b$ with $a,b\geq 2$.  The edge $\{p^a,q^b\}$ is present.  If $(a,b)\neq (2,2)$, then $u=p^{a-1}q^{b-1}$ and $v=pq$ are distinct vertices and $uv$ is zero modulo $n$. So all cross products with $p^a$ and $q^b$ miss one full prime power.  If $(a,b)=(2,2)$, choose two distinct unit multiples $u=pq$ and $v=cpq$, where $c$ is coprime to $pq$ and $c\not\equiv 1\pmod {pq}$, such a $c$ exists because $\varphi(pq)\geq 2$.  The same cross-product check applies.  Thus both exponents cannot exceed one.  Hence a cochordal two-prime case has the form $p^a q$.

\medskip

It remains to prove sufficiency.  For $p^a$, $p^a q$, and $pqr$, explicit constructible systems are given in Theorems \ref{thm:prime-power-type}, \ref{thm:two-prime-type} and \ref{thm:three-prime-type}.  By Theorem \ref{thm:constructible-characterization}, each corresponding zero-divisor graph is cochordal.
\end{proof}

\medskip

The proof validates the classification but corrects the two-prime obstruction by using vertices that are always residues of the ring and always distinct after the stated small exception is handled.  This avoids the unsupported use of an auxiliary prime divisor of $pq-1$.

\medskip

\noindent The following lemma isolates the four-prime obstruction used in the classification.
\begin{lemma}\label{lem:four-prime}
If $n$ has at least four distinct prime divisors, then $\Gamma(\mathbb{Z}/n\mathbb{Z})$ contains an induced matching of size two.
\end{lemma}

\begin{proof}
The first paragraph of the proof of Theorem \ref{thm:classification} gives explicit vertices.  The two displayed products $xy$ and $uv$ contain every prime power dividing $n$, whereas each cross product omits one of $p_1^{\alpha_1}$, $p_2^{\alpha_2}$, $p_3^{\alpha_3}$, or $p_4^{\alpha_4}$.  Thus exactly the two matching edges occur among those four vertices.
\end{proof}

\medskip

This lemma is the arithmetic reason that a cochordal zero-divisor graph cannot encode too many independent prime directions.  It is stronger than merely saying that the complement has an induced cycle, as it exhibits the homological obstruction detected by regularity.

\medskip

\noindent The following lemma treats the three-prime nonsquare-free obstruction.
\begin{lemma}\label{lem:three-prime-obstruction}
If $n=p^a q^b r^c$ and at least one of $a,b,c$ is larger than one, then $\Gamma(\mathbb{Z}/n\mathbb{Z})$ is not cochordal.
\end{lemma}

\begin{proof}
By symmetry assume $a>1$, so the vertices $p^a$, $q^b r^c$, $p^{a-1}q^b$, and $pr^c$ induce exactly two disjoint edges, as verified in the second paragraph of Theorem \ref{thm:classification}.  Hence the graph has an induced matching of size two and cannot be cochordal.
\end{proof}

\medskip

The above lemma explains why $pqr$ is the only surviving three-prime case.  Any repeated exponent supplies enough room to split the prime powers into two disjoint zero products.

\medskip

\noindent The following lemma treats the two-prime obstruction.
\begin{lemma}\label{lem:two-prime-obstruction}
If $n=p^a q^b$ with distinct primes and $a,b\geq 2$, then $\Gamma(\mathbb{Z}/n\mathbb{Z})$ is not cochordal.
\end{lemma}

\begin{proof}
The third paragraph of Theorem \ref{thm:classification} constructs an induced matching of size two.  The exceptional case $a=b=2$ is handled by taking two distinct unit multiples of $pq$ in the class $[pq]$.
\end{proof}

\medskip

\noindent The following proposition summarizes the sufficiency mechanism.
\begin{proposition}\label{prop:sufficiency-summary}
Each graph $\Gamma(\mathbb{Z}/p^a\mathbb{Z})$, $\Gamma(\mathbb{Z}/p^a q\mathbb{Z})$, and $\Gamma(\mathbb{Z}/pqr\mathbb{Z})$ admits a valid cochordal constructible system.
\end{proposition}

\begin{proof}
For prime powers, the system is built from classes $[p^i]$ with $i\geq\lceil a/2\rceil$, and same-class vertices may be used because those classes are cliques.  For $p^a q$, the system uses the classes $[p^i q]$ in increasing construction order, and same-class previous vertices are included only when $i\geq\lceil a/2\rceil$.  For $pqr$, the system uses the classes $[pq]$, then $[pr]$, then $[qr]$ in construction order, and no same-class previous vertices are included.  The detailed verification appears in Sections \ref{sec:prime-powers}, \ref{sec:two-prime} and \ref{sec:three-prime}.
\end{proof}

\medskip

Proposition \ref{prop:sufficiency-summary} highlights that the corrected construction is not a cosmetic modification.  The validity of sufficiency depends exactly on the same-class rule of Theorem \ref{cor:clique-independent-block}.

\medskip

\begin{example}[Classification at small values]\label{ex:classification-small}
For $n\leq 30$, the non-prime values producing cochordal zero-divisor graphs include $4,8,9,12,16,18,20,25,27,28$ and $30$, together with products $pq$ such as $6,10,14,15,21,22,26$.  The value $36=2^2\cdot 3^2$ is excluded by Theorem \ref{lem:two-prime-obstruction}, as $\{4,9\}$ and $\{6,30\}$ form an induced matching.
\end{example}

\medskip

The example shows how quickly the obstruction appears.  The excluded value $36$ is the smallest case in which both exponents in a two-prime factorization exceed one.

\medskip

\begin{figure}[h]
\centering
\begin{tikzpicture}[>=Latex,node distance=7mm]
\tikzstyle{box}=[rectangle,draw,rounded corners,align=center,minimum width=31mm,minimum height=8mm]
\node[box] (n) {factor $n$};
\node[box,below left=of n,xshift=-7mm] (four) {$\geq4$ primes\ induced matching};
\node[box,below=of n] (three) {$3$ primes\ square-free?};
\node[box,below right=of n,xshift=7mm] (two) {$2$ primes\ one exponent $1$?};
\node[box,below=of three,yshift=-5mm] (yes) {$p^a$, $p^a q$, $pqr$};
\draw[->] (n)--(four);
\draw[->] (n)--(three);
\draw[->] (n)--(two);
\draw[->] (three)--(yes);
\draw[->] (two)--(yes);
\end{tikzpicture}
\caption{Arithmetic decision tree for cochordality of $\Gamma(\mathbb{Z}/n\mathbb{Z})$.}
\label{fig:classification-tree}
\end{figure}
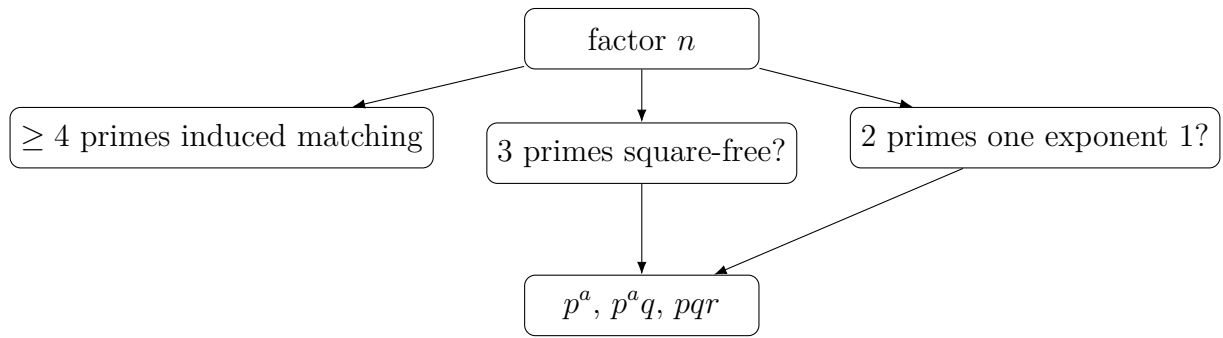

Figure \ref{fig:classification-tree} gives the block diagram for the cochordality of $\Gamma(\mathbb{Z}/n\mathbb{Z})$. The block diagram separates necessity from sufficiency.  The left branches are ruled out by induced matchings, while the surviving leaves are later handled by constructible systems. Table \ref{tab 2} gives the numerical validation for cochordality of $\Gamma(\mathbb{Z}/n\mathbb{Z})$ with  $n\in \{18,30,36,210\}.$

\medskip
\begin{table}[H]
\begin{center}
\begin{tabular}{cccl}
\toprule
$n$ & factorization & is cochordal? & reason\\
\midrule
$18$ & $2\cdot 3^2$ & yes & $p^a q$ form\\
$30$ & $2\cdot 3\cdot 5$ & yes & square-free three-prime form\\
$36$ & $2^2\cdot 3^2$ & no & two-prime exponent obstruction\\
$60$ & $2^2\cdot 3\cdot 5$ & no & three-prime exponent obstruction\\
$210$ & $2\cdot 3\cdot 5\cdot 7$ & no & four-prime obstruction\\
\bottomrule
\end{tabular}
\end{center}
\caption{Classification of cochordality of $\Gamma(\mathbb{Z}/n\mathbb{Z})$.}
\label{tab 2}
\end{table}

\medskip

Table \ref{tab 2} compares the classification across the first genuinely different arithmetic patterns.  It is a useful check because all four obstruction types appear among small integers.

\section{The prime-power case}\label{sec:prime-powers}
 The constructible system for $p^a$ is valid as in \cite{DungVu2026}, but the Betti formula must retain the global subtraction.  The corrected formula below is the first closed expression in this paper and already changes numerical values for $p>2$ and $a=2$.

\medskip

Let $n=p^a$ with $a\geq 2$.  For $1\leq i\leq a-1$, let
$ V_i=\{x\in \mathbb{Z}/p^a\mathbb{Z}:[x]=[p^i]\}, $ with $\phi_i=|V_i|=p^{a-i-1}(p-1)$ and let $m=\lceil a/2\rceil$.

\medskip

\noindent The following theorem gives the valid type sequence for prime powers.
\begin{theorem}\label{thm:prime-power-type}
The graph $\Gamma(\mathbb{Z}/p^a\mathbb{Z})$ is cochordal.  A valid constructible system has one decreasing block for each $i=m,m+1,\ldots,a-1$, namely
$$
T_i=\bigl(p^i-p^{a-i-1}-1,\ p^i-p^{a-i-1}-2,\ \ldots,\ p^i-p^{a-i}\bigr),
$$
of length $\phi_i$.  The type sequence is $T_{a-1},T_{a-2},\ldots,T_m$.
\end{theorem}

\begin{proof}
Order the vertices of each $V_i$ as $v_{i,1},\ldots,v_{i,\phi_i}$.  For $i\geq m$ and $1\leq j\leq\phi_i$, let
$$
U_{i,j}=V_{a-i}\cup V_{a-i+1}\cup\cdots\cup V_{i-1}\cup\{v_{i,1},\ldots,v_{i,j-1}\}.
$$
A vertex in $V_i$ is adjacent to every vertex in $V_t$ with $t\geq a-i$, and, because $2i\geq a$, it is also adjacent to earlier vertices of $V_i$.  Thus each listed cover contains only true neighbors of the new center.  The same valuation inequality shows that $U_{i,j}$ covers the graph constructed before that step.  More explicitly, if an earlier edge has endpoints in $V_s$ and $V_t$ with $s\leq t$, then $s+t\geq a$ forces $t\geq a-i$ whenever both endpoints have already appeared before the $V_i$-step, so at least one endpoint lies in the displayed union or among the earlier same-class vertices.  Hence, these data form a cochordal constructible system.
 The cover size is
$$
|U_{i,j}|=\sum_{t=a-i}^{i-1}p^{a-t-1}(p-1)+j-1=p^i-p^{a-i}+j-1.
$$
When the vertices in $V_i$ are listed in reverse order in the type sequence, the block is exactly $T_i$.
\end{proof}

\medskip

The above theorem agrees with the valid part of the earlier construction.  Its importance here is that it supplies decreasing blocks, so Theorem \ref{prop:block-evaluation} applies in its first case.

\medskip

\noindent The following theorem gives the corrected Betti numbers for the prime-power case.
\begin{theorem}\label{thm:prime-power-betti}
Let $a\geq 2$, $n=p^a$, $G=\Gamma(\mathbb{Z}/n\mathbb{Z})$, $I=I(G)$, $m=\lceil a/2\rceil$, and let
$$
K_{p,a}=\sum_{j=m}^{a-1}p^{a-j-1}(p-1)=p^{a-m}-1.
$$
Then $S/I$ has a linear resolution and, for every $i\geq 1$,
$$
\beta_{i,i+1}(S/I)=\beta_i(S/I)=\sum_{j=m}^{a-1}p^{a-j-1}(p-1)\binom{p^j-2}{i}-\binom{K_{p,a}}{i+1}.
$$
\end{theorem}

\begin{proof}
For the block $T_j$ in Theorem \ref{thm:prime-power-type}, the number of earlier blocks to its left in the type sequence is
$$
K_{>j}=\sum_{\ell=j+1}^{a-1}p^{a-\ell-1}(p-1)=p^{a-j-1}-1.
$$
The first entry of $T_j$ is $p^j-p^{a-j-1}-1$, so the shifted value is $p^j-2$.  Since $T_j$ is decreasing, Theorem \ref{prop:block-evaluation} gives the contribution $\phi_j\binom{p^j-2}{i}$.  Now, summing over $j=m,\ldots,a-1$ and subtracting the global term $\binom{K_{p,a}}{i+1}$ proves the formula.  The linearity of the resolution follows from cochordality and Theorem \ref{thm:froberg}.
\end{proof}

\medskip

The novelty compared with \cite{DungVu2026} is the final term $-\binom{K_{p,a}}{i+1}$.  Without it the formula overcounts, beginning with the graph $K_{p-1}$ obtained from $n=p^2$.

\medskip

\noindent The following corollary gives projective dimension and regularity.
\begin{corollary}\label{cor:prime-power-pd}
If $a\geq 2$ and $G=\Gamma(\mathbb{Z}/p^a\mathbb{Z})$, then
$ \operatorname{pd}(S/I(G))=p^{a-1}-2,$ and $
\operatorname{reg}(S/I(G))=1, $
whenever $G$ has at least one edge.  If $p=2$ and $a=2$, then $G$ is edgeless and $\operatorname{pd}(S/I(G))=0$.
\end{corollary}

\begin{proof}
The projective dimension is the maximum shifted value in Theorem \ref{thm:type-formula}.  By the proof of Theorem  \ref{thm:prime-power-betti}, these shifted values are $p^j-2$ for $j=m,\ldots,a-1$, whose maximum is $p^{a-1}-2$.  The regularity statement follows, since the resolution is linear for $S/I(G)$, so the nonzero Betti numbers lie in degrees $i+1$.
\end{proof}

\medskip

The above corollary confirms the projective dimension stated in the earlier manuscript for prime powers, but the agreement of projective dimension does not imply agreement of the full Betti table.

\medskip

\noindent The following proposition gives a sharp small-case $a=2$.
\begin{proposition}\label{prop:prime-power-a2}
For $n=p^2$, the graph $\Gamma(\mathbb{Z}/p^2\mathbb{Z})$ is the complete graph $K_{p-1}$, and
$$
\beta_i(S/I(\Gamma(\mathbb{Z}/p^2\mathbb{Z})))=i\binom{p-1}{i+1}
$$
for $i\geq 1$.
\end{proposition}

\begin{proof}
The vertices are the nonzero multiples of $p$, and any two distinct such vertices multiply to zero modulo $p^2$.  Hence the graph is $K_{p-1}$.  Substituting $a=2$ into Theorem \ref{thm:prime-power-betti}, we have
$$
\beta_i=(p-1)\binom{p-2}{i}-\binom{p-1}{i+1}.
$$
With $(p-1)\binom{p-2}{i}=(i+1)\binom{p-1}{i+1}$, we get $i\binom{p-1}{i+1}$ after subtracting $\binom{p-1}{i+1}$.
\end{proof}

\medskip

The above proposition makes the correction transparent.  For $p=3$, the graph is one edge, and the corrected formula gives $\beta_1=1$ rather than $2$.

\medskip

\noindent The following corollary gives the first noncomplete prime-power family (case $a=3$)
\begin{corollary}\label{cor:prime-power-a3}
For $n=p^3$,
$$
\beta_i(S/I(\Gamma(\mathbb{Z}/p^3\mathbb{Z})))=(p-1)\binom{p^2-2}{i}-\binom{p-1}{i+1}.
$$
In particular, for $p=2$ one obtains $\beta_1=2$ and $\beta_2=1$.
\end{corollary}

\begin{proof}
Here $m=2$, so there is only the block $T_2$, of length $p-1$.  The formula follows directly from Theorem \ref{thm:prime-power-betti}.  For $p=2$, the graph is the star with center $4$ and leaves $2,6$ in $\mathbb{Z}/8\mathbb{Z}$ (see Figure \ref{fig:z8}), so $I=x_4(x_2,x_6)$ and the Betti numbers are $2,1$.
\end{proof}

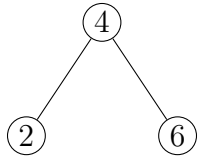
\begin{figure}[H]
	\centering
	\begin{tikzpicture}[scale=1, every node/.style={circle,draw,minimum size=5mm,inner sep=1pt}]
		\node (v2) at (0,0) {$2$};
		\node (v6) at (2,0) {$6$};
		\node (v4) at (1,1.5) {$4$};
		\draw (v4)--(v2);
		\draw (v4)--(v6);
	\end{tikzpicture}
	\caption{The graph $\Gamma(\mathbb{Z}/8\mathbb{Z})$, a two-edge star.}
	\label{fig:z8}
\end{figure}
\medskip

This small case is useful because it is not complete but is still simple enough to verify by inspection. Figure \ref{fig:z8} illustrates the case $p=2$, $a=3$.  It also shows why the edge ideal has the form $x_4(x_2,x_6)$ and hence a two-step linear resolution.

\medskip

\begin{example}\label{ex:prime-power-numerics}
For $n=16$, one has $p=2$, $a=4$, $m=2$, and $K_{2,4}=3$. The graph $\Gamma(\mathbb{Z}/16\mathbb{Z})$ is shown in Figure \ref{fig 0}. 
The formula gives
$$
\beta_i=2\binom{2}{i}+\binom{6}{i}-\binom{3}{i+1}.
$$
Thus $\beta_1=2\cdot 2+6-3=7$, $\beta_2=2\cdot 1+15-1=16$, $\beta_3=20$, $\beta_4=15$, $\beta_5=6$, and $\beta_6=1$.
\end{example}
Theorem 1.8 \cite{DungVu2026} gives $(1,\ 10,\ 17,\ 20,\ 15,\ 6,\ 1)$, which differs from above calculation as well from Macaulay2  calculation, see Table \ref{tab:betti_numbers 0}.  The graph $\Gamma(\mathbb{Z}/16\mathbb{Z})$ has $7$ edges not $10$ as per Theorem 1.8 \cite{DungVu2026}.
\begin{table}[H]
	\centering
	\ttfamily 
	\begin{tabular}{r c c c c c c c}
		& 0 & 1 & 2  & 3  & 4  & 5 & 6 \\
		total: & 1 & 7 & 16 & 20 & 15 & 6 & 1 \\
		0: & 1 & . & .  & .  & .  & . & . \\
		1: & . & 7 & 16 & 20 & 15 & 6 & 1 \\
	\end{tabular}
	\normalfont 
	\caption{Graded Betti numbers for the ideal.}
	\label{tab:betti_numbers 0}
\end{table}
\begin{figure}[H]
	\centering
	\begin{tikzpicture}[scale=1, every node/.style={circle,draw,minimum size=5mm,inner sep=1pt}]
	
	\node (8) at (0, 0) {8};
	
	\node  (4)  at (-1.5, 2) {4};
	\node  (12) at ( 1.5, 2) {12};
	
	\node  (2)  at (-3, -1.5) {2};
	\node  (6)  at (-1, -2)   {6};
	\node  (10) at ( 1, -2)   {10};
	\node  (14) at ( 3, -1.5) {14};
	
	\draw (4) -- (12);
	
	\foreach \v in {2, 4, 6, 10, 12, 14} {
		\draw (8) -- (\v);
	}
	
\end{tikzpicture}
\caption{Zero divisor graph of $\Gamma(\mathbb{Z}/16\mathbb{Z})$.}
\label{fig 0}
\end{figure}
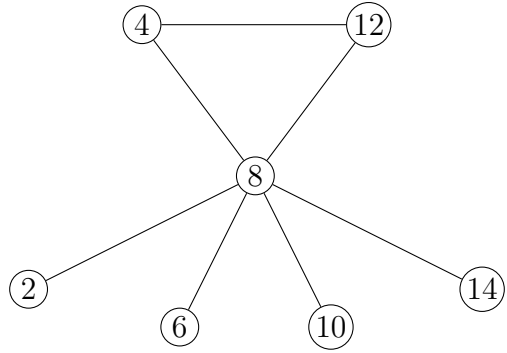
\medskip

The example displays a substantial linear strand even for a small ring.  The largest nonzero Betti number occurs before the last homological degree, while the projective dimension is $6$.

\medskip

\begin{table}[H]
	\begin{center}
\begin{tabular}{ccccc}
\toprule
$n$ & type blocks & $K$ & corrected Betti sequence & $\operatorname{pd}$\\
\midrule
$8$ & $(2)$ & $1$ & $(2,1)$ & $2$\\
$9$ & $(1,0)$ & $2$ & $(1)$ & $1$\\
$16$ & $(6),(2,1)$ & $3$ & $(7,16,20,15,6,1)$ & $6$\\
\bottomrule
\end{tabular}
\end{center}
\caption{Betti numbers and projective dimension for some $n$ in $\Gamma(\mathbb{Z}/n\mathbb{Z})$}
\label{tab 3}
\end{table}

\medskip

Table \ref{tab 3} confirms that the correction is not optional.  The case $n=9$ would be wrong without subtracting $\binom{2}{2}$ from the raw first Betti number (Theorem 1.8 \cite{DungVu2026}).
\section{The two-prime case $p^a q$}\label{sec:two-prime}
 In this section, we give the principal correction to the two-prime formula (Theorem 1.9 \cite{DungVu2026}).  The low blocks $[p^i q]$ with $i<\lceil a/2\rceil$ are independent repeated blocks, not decreasing clique blocks, so their Betti contributions are hockey-stick differences rather than repeated identical binomial terms.

\medskip

Let $n=p^a q$, where $p$ and $q$ are distinct primes and $a\geq 1$.  Use the classes $A_t=[p^t]$ for $1\leq t\leq a$ and $B_i=[p^i q]$ for $0\leq i\leq a-1$.  Let 
$ \phi_i=|B_i|=p^{a-i-1}(p-1),$ with $ m=\lceil a/2\rceil. $
The total number of center vertices in the construction below is
$$
K_{p,a}^{(q)}=\sum_{i=0}^{a-1}\phi_i=p^a-1.
$$

\medskip

\noindent The following theorem gives the corrected type sequence for $p^a q$.
\begin{theorem}\label{thm:two-prime-type}
The graph $\Gamma(\mathbb{Z}/p^a q\mathbb{Z})$ is cochordal.  A valid constructible system has, for each $i\geq m$, a decreasing block
$$
T_i=\bigl(qp^i-p^{a-i-1}-1,\ qp^i-p^{a-i-1}-2,\ \ldots,\ qp^i-p^{a-i}\bigr)
$$
of length $\phi_i$, and, for each $0\leq i<m$, a constant block
$$
T_i=\bigl((q-1)p^i,\ (q-1)p^i,\ \ldots,\ (q-1)p^i\bigr)
$$
of length $\phi_i$.  The type sequence is $T_{a-1},T_{a-2},\ldots,T_0$.
\end{theorem}

\begin{proof}
We order $B_i$ as $b_{i,1},\ldots,b_{i,\phi_i}$.  In construction order, the blocks are added from $B_0$ up to $B_{a-1}$.  For a center $b_{i,j}$, the true neighbors among the $A$-classes are exactly $A_{a-i}\cup A_{a-i+1}\cup\cdots\cup A_a$, whose total size is $(q-1)p^i$.

\medskip

Among previously constructed $B$-classes, the center $b_{i,j}$ is adjacent exactly to the classes $B_t$ with $a-i\leq t\leq i-1$.  This range is empty when $i<m$.  When $i\geq m$, it has total size $p^i-p^{a-i}$.  Finally, previous vertices of the same class $B_i$ may be included exactly when $2i\geq a$, again equivalent to $i\geq m$.

\medskip

Thus for $i<m$ the cover size is the constant $(q-1)p^i$ for all $j$.  For $i\geq m$ the cover size is
$$
(q-1)p^i+(p^i-p^{a-i})+j-1=qp^i-p^{a-i}+j-1.
$$
Listing the vertices inside each block in reverse order gives the displayed type sequence.  The same valuation inequalities show that each cover is a vertex cover of the previously constructed graph.  Indeed, every earlier edge either joins an earlier $B$-center to an eligible $A$-class already listed in the cover, or joins two earlier $B$-classes whose valuation sum is at least $a$. In the latter case the endpoint with larger index lies among the $B_t$ satisfying $t\geq a-i$ and is therefore included.  Hence the system is valid.
\end{proof}

\medskip

The above theorem corrects the low block in the earlier proposed type sequence in \cite{DungVu2026}.  The expression that decreases through a low block would add same-class edges that do not exist when $2i<a$.

\medskip

\noindent The following theorem gives the corrected two-prime Betti formula.
\begin{theorem}\label{thm:two-prime-betti}
Let $G=\Gamma(\mathbb{Z}/p^a q\mathbb{Z})$  with  $I=I(G)$, and for every $h\geq 1$,
$$
\begin{aligned}
\beta_{h,h+1}(S/I)=\beta_h(S/I)
=&\sum_{j=m}^{a-1}p^{a-j-1}(p-1)\binom{qp^j-2}{h}  +\sum_{j=0}^{m-1}\Bigg[\binom{(q-1)p^j+p^{a-j}-1}{h+1}\\
&-\binom{(q-1)p^j+p^{a-j-1}-1}{h+1}\Bigg]-\binom{p^a-1}{h+1}.
\end{aligned}
$$
Here the first sum is empty if $a=1$.
\end{theorem}

\begin{proof}
For a high block $j\geq m$, the first entry is $qp^j-p^{a-j-1}-1$, and the number of blocks to its left is
$$
K_{>j}=\sum_{\ell=j+1}^{a-1}p^{a-\ell-1}(p-1)=p^{a-j-1}-1.
$$
Hence the shifted value is $qp^j-2$, and Theorem \ref{prop:block-evaluation} gives the first sum.

\medskip

For a low block $0\leq j<m$, the repeated value is $d_j=(q-1)p^j$.  The number of entries to its left is again $K_{>j}=p^{a-j-1}-1$, and the length is $\phi_j=p^{a-j-1}(p-1)$.  Thus its contribution is
$$
\binom{d_j+K_{>j}+\phi_j}{h+1}-\binom{d_j+K_{>j}}{h+1}
=\binom{(q-1)p^j+p^{a-j}-1}{h+1}-\binom{(q-1)p^j+p^{a-j-1}-1}{h+1}.
$$
The total number of center steps is $p^a-1$, so Theorem \ref{prop:block-evaluation} gives the final subtraction.
\end{proof}

\medskip

The formula is new relative to \cite{DungVu2026} in two ways, as the low part is a hockey-stick difference, and the global correction remains.  In the special case $a=1$ it reduces to the known complete bipartite formula.

\medskip

\noindent The following corollary records the projective dimension.
\begin{corollary}\label{cor:two-prime-pd}
With the notation of Theorem \ref{thm:two-prime-betti},
$$
\operatorname{pd}(S/I)=\max\left(\{qp^j-2:m\leq j\leq a-1\}\cup\{(q-1)p^j+p^{a-j}-2:0\leq j<m\}\right),
$$
where the first set is omitted when $a=1$.  Also $\operatorname{reg}(S/I)=1$ whenever the graph has an edge.
\end{corollary}

\begin{proof}
By Theorem \ref{thm:type-formula}, the projective dimension is the maximum shifted type value.  High blocks give shifted values $qp^j-2$.  Low constant blocks have shifted values from $(q-1)p^j+p^{a-j-1}-1$ up to $(q-1)p^j+p^{a-j}-2$, so their maximum is the second displayed family.  The linearity is a consequence of cochordality.
\end{proof}

\medskip

This statement corrects the claim that the projective dimension is always $qp^{a-1}-2$ for $a\geq 2$.  That value can fail when a low independent block is larger, for instance, $n=18$ gives projective dimension $8$, not $4$.

\medskip

\noindent The following corollary recovers the complete bipartite case (case $a=1$).
\begin{corollary}\label{cor:bipartite-case}
If $n=pq$, then $\Gamma(\mathbb{Z}/pq\mathbb{Z})=K_{q-1,p-1}$ and
$$
\beta_h(S/I)=\binom{p+q-2}{h+1}-\binom{q-1}{h+1}-\binom{p-1}{h+1}.
$$
\end{corollary}

\begin{proof}
When $a=1$, the first sum in Theorem \ref{thm:two-prime-betti} is empty, and the second sum has only $j=0$.  Now, substitution gives
$$
\binom{p+q-2}{h+1}-\binom{q-1}{h+1}-\binom{p-1}{h+1}.
$$
The graph description is immediate from the edge relation, as every nonzero multiple of $p$ is adjacent to every nonzero multiple of $q$, and there are no edges inside either part.
\end{proof}

\medskip

The above corollary is a useful consistency check because the Betti numbers of complete bipartite edge ideals are classical and follow from the product ideal $(X)(Y)$.

\medskip

\noindent The following example gives a concrete correction test.
\begin{example}\label{prop:z18-test}
For $n=18=3^2\cdot 2$, the corrected Betti sequence of $S/I(\Gamma(\mathbb{Z}/18\mathbb{Z}))$ is
$$
(13,\ 39,\ 64,\ 72,\ 56,\ 28,\ 8,\ 1).
$$
In particular, $\operatorname{pd}(S/I)=8$.
\end{example}

As $p=3$, $q=2$, and $a=2$.  Then $m=1$.  The high contribution is $2\binom{4}{h}$, the low contribution is
$$
\binom{9}{h+1}-\binom{3}{h+1},
$$
and the global subtraction is $\binom{8}{h+1}$.  Evaluating for $h=1,\ldots,8$, we obtain the Betti sequence (13,\ 39,\ 64,\ 72,\ 56,\ 28,\ 8,\ 1).

\medskip

This example shows why low independent blocks cannot be ignored.  The six vertices in $[2]$ are not mutually adjacent, but their repeated covers still create high projective dimension through the positional shifts.

\medskip

\begin{example}\label{ex:z12}
Let $n=12=2^2\cdot 3$.  The classes are $A_2=[4]=\{4,8\}$, $A_1=[2]=\{2,10\}$, $B_0=[3]=\{3,9\}$, and $B_1=[6]=\{6\}$, see Figure \ref{fig:z12}.  The corrected formula gives
$$
\beta_h=\binom{4}{h}+\binom{5}{h+1}-\binom{3}{h+1}-\binom{3}{h+1}.
$$
Hence the Betti sequence is $(1, 8,14,9,2)$, which verifies computer calculation by Macaulay2 \cite{Macaulay2} (see Table \ref{tab:betti_numbers}).
\end{example}
\begin{table}[H]
	\centering
	\ttfamily 
	\begin{tabular}{r c c c c c}
		& 0 & 1 & 2  & 3 & 4 \\
		total: & 1 & 8 & 14 & 9 & 2 \\
		0: & 1 & . & .  & . & . \\
		1: & . & 8 & 14 & 9 & 2 \\
	\end{tabular}
	\normalfont 
	\caption{Graded Betti numbers for the ideal.}
	\label{tab:betti_numbers}
\end{table}
While Theorem 1.9 \cite{DungVu2026} gives $(1,\ 12,\ 18,\ 12,\ 3)$, which is not correct at first place as there are $8$ edges not $12$.
\medskip

The example has both a high clique block and a low independent block.  The class $[3]$ contributes through a hockey-stick difference rather than through a decreasing type block.

\medskip

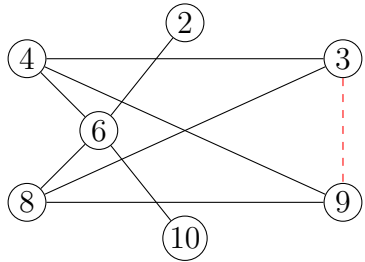
\begin{figure}[H]
\centering
\begin{tikzpicture}[scale=0.95, every node/.style={circle,draw,minimum size=5mm,inner sep=1pt}]
\node (a21) at (0,1) {$4$};
\node (a22) at (0,-1) {$8$};
\node (a11) at (2.2,1.5) {$2$};
\node (a12) at (2.2,-1.5) {$10$};
\node (b01) at (4.4,1) {$3$};
\node (b02) at (4.4,-1) {$9$};
\node (b1) at (1,0) {$6$};
\draw (b01)--(a21); \draw (b01)--(a22); \draw (b02)--(a21); \draw (b02)--(a22);
\draw (b1)--(a21); \draw (b1)--(a22); \draw (b1)--(a11); \draw (b1)--(a12);
\draw[dashed,red] (b01)--(b02);
\end{tikzpicture}
\caption{The graph $\Gamma(\mathbb{Z}/12\mathbb{Z})$.  The dashed segment between $3$ and $9$ is not an edge.}
\label{fig:z12}
\end{figure}

\medskip

Figure \ref{fig:z12} displays the mixed structure of the two-prime case.  The low class $[3]$ is independent, while the high class $[6]$ is a singleton clique block.

\medskip

\begin{table}[H]
	\begin{center}
\begin{tabular}{cccccc}
\toprule
$n$ & $(p,a,q)$ & high part & low part & correction & Betti sequence\\
\midrule
$12$ & $(2,2,3)$ & $\binom{4}{h}$ & $\binom{5}{h+1}-\binom{3}{h+1}$ & $-\binom{3}{h+1}$ & $(8,14,9,2)$\\[4mm]
$18$ & $(3,2,2)$ & $2\binom{4}{h}$ & $\binom{9}{h+1}-\binom{3}{h+1}$ & $-\binom{8}{h+1}$ & $(13,39,64,72,56,28,8,1)$\\ [4mm]
$15$ & $(3,1,5)$ & none & $\binom{6}{h+1}-\binom{4}{h+1}$ & $-\binom{2}{h+1}$ & $(8,16,14,6,1)$\\
\bottomrule
\end{tabular}
\end{center}
\label{tab 4}
\caption{Betti sequence of $\Gamma(\mathbb{Z}/p^a q\mathbb{Z})$ for some values.}
\end{table}
Table \ref{tab 4}  compares three two-prime rings.  The row $n=15$ confirms the complete bipartite specialization, while $n=18$ shows the failure of the shorter projective-dimension formula.

\medskip

Algorithm \ref{alg:two-prime} is about calculating Betti sequence for two prime cases in $\Gamma(\mathbb{Z}/p^a q\mathbb{Z})$.
\begin{algorithm}[H]
\caption{Two-prime Betti numbers}
\label{alg:two-prime}
\begin{algorithmic}[1]
\Require Distinct primes $p,q$, integer $a\geq 1$, degree $h$.
\Ensure $\beta_h(S/I(\Gamma(\mathbb{Z}/p^a q\mathbb{Z})))$.
\State $m\gets \lceil a/2\rceil$, $B\gets 0$.
\For{$j=m$ to $a-1$}
    \State $B\gets B+p^{a-j-1}(p-1)\binom{qp^j-2}{h}$.
\EndFor
\For{$j=0$ to $m-1$}
    \State $B\gets B+\binom{(q-1)p^j+p^{a-j}-1}{h+1}-\binom{(q-1)p^j+p^{a-j-1}-1}{h+1}$.
\EndFor
\State \Return $B-\binom{p^a-1}{h+1}$.
\end{algorithmic}
\end{algorithm}

\medskip

Algorithm \ref{alg:two-prime} runs in $O(a)$ arithmetic steps.  It avoids constructing the graph, whose vertex count is $p^a q-p^{a-1}(p-1)(q-1)-1$ after excluding units and zero.

\medskip

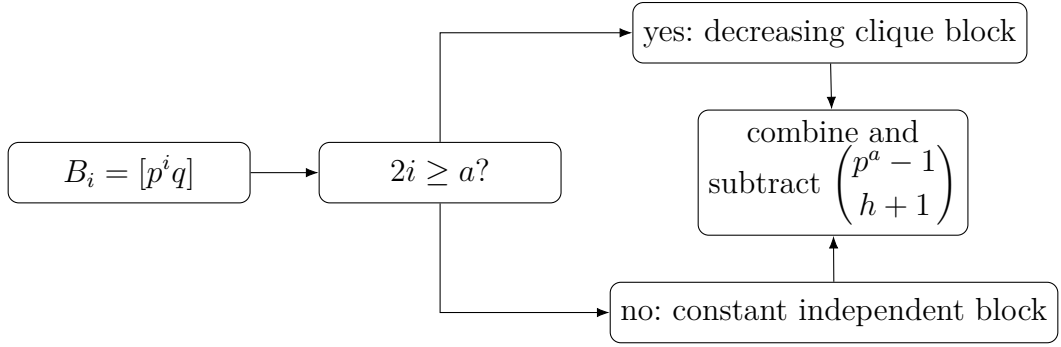
\begin{figure}[H]
\centering
\begin{tikzpicture}[node distance=9mm,>=Latex]
\tikzstyle{box}=[rectangle,draw,rounded corners,align=center,minimum width=32mm,minimum height=8mm]
\node[box] (Bi) {$B_i=[p^i q]$};
\node[box,right=of Bi] (test) {$2i\geq a$?};
\node[box,above right=of test,yshift=4mm, xshift=3mm] (dec) {yes:\ decreasing\ clique block};
\node[box,below right=of test, yshift=-4mm] (const) {no:\ constant\ independent block};
\node[box,right=20mm of test, xshift=-2mm] (sum) {combine and\\ subtract $\binom{p^a-1}{h+1}$};
\draw[->] (Bi)--(test);
\draw[->] (test)|-(dec);
\draw[->] (test)|-(const);
\draw[->] (dec)--(sum);
\draw[->] (const)--(sum);
\end{tikzpicture}
\caption{Decision rule for the $p^a q$ type blocks.}
\label{fig:two-prime-flow}
\end{figure}

\medskip

Figure \ref{fig:two-prime-flow} shows the diagram is the conceptual core of the section.  The same valuation class changes behavior exactly at the threshold $i=\lceil a/2\rceil$.

\medskip

\section{The square-free three-prime case}\label{sec:three-prime}
 The graph $\Gamma(\mathbb{Z}/pqr\mathbb{Z})$ is cochordal, and its projective dimension stated in \cite{DungVu2026} remains correct for $p<q<r$.  However, the type blocks for $[pq]$, $[pr]$, and $[qr]$ are constant blocks, not decreasing blocks, and the full Betti formula (Theorem 1.10 \cite{DungVu2026}) must be corrected accordingly.

\medskip

Let $p<q<r$ be primes and let $n=pqr$.  Denote the classes represented by $p,q,r$ by $A_p,A_q,A_r$, and the classes represented by $pq,pr,qr$ by $B_{pq},B_{pr},B_{qr}$.  Their sizes are
$ |A_p|=(q-1)(r-1), |A_q|=(p-1)(r-1), |A_r|=(p-1)(q-1), 
|B_{pq}|=r-1, |B_{pr}|=q-1,$ and $ |B_{qr}|=p-1.
$

\medskip

\noindent The following theorem gives the corrected type sequence for $pqr$.
\begin{theorem}\label{thm:three-prime-type}
The graph $\Gamma(\mathbb{Z}/pqr\mathbb{Z})$ is cochordal.  A valid constructible system has three constant blocks, 
$ T_{pq}=\bigl((p-1)(q-1),\ldots,(p-1)(q-1)\bigr) $
of length $r-1$, $
T_{pr}=\bigl(pr-p,\ldots,pr-p\bigr) $
of length $q-1$, and $
T_{qr}=\bigl(qr-1,\ldots,qr-1\bigr) $
of length $p-1$.  The type sequence is $T_{qr},T_{pr},T_{pq}$.
\end{theorem}

\begin{proof}
We construct the graph $\Gamma(\mathbb{Z}/pqr\mathbb{Z})$ by taking centers first from $B_{pq}$, then from $B_{pr}$, then from $B_{qr}$, or equivalently, the displayed type sequence is written in reverse construction order.  A vertex of $B_{pq}$ is adjacent to all vertices of $A_r$ and to no previous same-class vertex, because $(pq)^2$ is not divisible by $pqr$.  Hence every cover in the first construction block has size $|A_r|=(p-1)(q-1)$.

\medskip

A vertex of $B_{pr}$ is adjacent to all vertices of $A_q$ and all vertices of the previously constructed class $B_{pq}$, but not to previous vertices of $B_{pr}$.  The cover size is
$$
|A_q|+|B_{pq}|=(p-1)(r-1)+(r-1)=pr-p.
$$
A vertex of $B_{qr}$ is adjacent to all vertices of $A_p$, $B_{pq}$, and $B_{pr}$, but not to previous vertices of $B_{qr}$.  The cover size is
$$
|A_p|+|B_{pq}|+|B_{pr}|=(q-1)(r-1)+(r-1)+(q-1)=qr-1.
$$
These covers contain only true neighbors.  They also cover the graph built before each step. In the first block there are no earlier center classes. In the second block the only earlier edges involve $A_r$ and $B_{pq}$, and $B_{pq}$ is included. In the third block all earlier edges have an endpoint in $B_{pq}$, $B_{pr}$, or $A_p$, all of which are included.  Thus the system is valid.
\end{proof}

\medskip

The correction is visible in each block, as none of $B_{pq}$, $B_{pr}$, or $B_{qr}$ is a clique.  Therefore none of the three blocks should decrease by one along same-class vertices.

\medskip

\noindent The following theorem gives the corrected Betti formula for the square-free three-prime case.
\begin{theorem}\label{thm:three-prime-betti}
Let $G=\Gamma(\mathbb{Z}/pqr\mathbb{Z})$ with $p<q<r$, and let $I=I(G)$.  For every $h\geq 1$,
$$
\begin{aligned}
\beta_{h,h+1}(S/I)=\beta_h(S/I)
={}&\binom{qr+p-2}{h+1}-\binom{qr-1}{h+1} +\binom{pr+q-2}{h+1}-\binom{pr-1}{h+1}\\
&+\binom{pq+r-2}{h+1}-\binom{pq-1}{h+1} -\binom{p+q+r-3}{h+1}.
\end{aligned}
$$
\end{theorem}

\begin{proof}
We us the constant-block case of Theorem \ref{prop:block-evaluation} to the three blocks in Theorem \ref{thm:three-prime-type}.  For $T_{qr}$ there are no blocks to its left, its repeated value is $qr-1$, and its length is $p-1$, thereby giving
$$
\binom{qr+p-2}{h+1}-\binom{qr-1}{h+1}.
$$
For $T_{pr}$ the left offset is $p-1$, its repeated value is $pr-p$, and its length is $q-1$, thereby giving
$$
\binom{pr+q-2}{h+1}-\binom{pr-1}{h+1}.
$$
For $T_{pq}$ the left offset is $(p-1)+(q-1)$, its repeated value is $(p-1)(q-1)$, and its length is $r-1$, thereby giving
$$
\binom{pq+r-2}{h+1}-\binom{pq-1}{h+1}.
$$
The total number of center steps is $p+q+r-3$, so the global subtraction is $\binom{p+q+r-3}{h+1}$, and the result follows.
\end{proof}

\medskip

The above formula replaces the shorter repeated-binomial expression.  The two expressions have the same largest endpoint when $p<q<r$, but they do not have the same Betti numbers.

\medskip

\noindent The following corollary gives projective dimension.
\begin{corollary}\label{cor:three-prime-pd}
If $p<q<r$, then
$
\operatorname{pd}(S/I(\Gamma(\mathbb{Z}/pqr\mathbb{Z})))=qr+p-3,$ and $ 
\operatorname{reg}(S/I(\Gamma(\mathbb{Z}/pqr\mathbb{Z})))=1. $
\end{corollary}

\begin{proof}
The shifted endpoints of the three constant blocks are $qr+p-3$, $pr+q-3$, and $pq+r-3$.  Since $p<q<r$, we have 
$ qr+p-3>pr+q-3 $
and $
qr+p-3>pq+r-3. $
Now, the projective dimension follows from Theorem \ref{thm:type-formula}, and regularity follows from linearity.
\end{proof}

\medskip

Thus the projective dimension alone is not sensitive enough to detect the error in the earlier Betti formula (Theorem 1.10 \cite{DungVu2026}).  The full linear strand is necessary.

\medskip

\noindent The following example gives the smallest square-free three-prime computation.
\begin{example}\label{prop:z30}
	For $n=30$, the class $[6]$ has four vertices, $[10]$ has two vertices, and $[15]$ has one vertex.  These are the three center classes.  No two distinct vertices inside $[6]$ are adjacent, because the product of two multiples of $6$ need not contain the prime factor $5$, see Figure \ref{fig:z30-compressed}. The actual graph $\Gamma(\mathbb{Z}/30\mathbb{Z})$ is show in Figure \ref{fig 1}.
So, for $n=30=2\cdot 3\cdot 5$, and from Theorem \ref{thm:three-prime-betti}, the corrected Betti sequence is
$$
(38,\ 211,\ 654,\ 1441,\ 2457,\ 3332,\ 3597,\ 3058,\ 2013,\ 1002,\ 364,\ 91,\ 14,\ 1).
$$
The projective dimension is $14$.
\end{example}
 Substitute $p=2$, $q=3$, and $r=5$ in Theorem \ref{thm:three-prime-betti}.  The three block contributions are
$$
\binom{15}{h+1}-\binom{14}{h+1},\quad
\binom{11}{h+1}-\binom{9}{h+1},\quad
\binom{9}{h+1}-\binom{5}{h+1},
$$
and the global subtraction is $\binom{7}{h+1}$.  Evaluation for $h=1,\ldots,14$ gives the displayed sequence. 
 The first entry $38$ is the actual number of edges.  However from Theorem 1.10 from \cite{DungVu2026}, the first term is $66$, so it fails already at homological degree one. In fact, by Theorem 1.10 \cite{DungVu2026}, we have
 $$ (1,\ 66,\ 293,\ 828,\ 1701,\ 2730,\ 3535,\ 3704,\ 3097,\ 2022,\ 1003,\ 364,\ 91,\ 14,\ 1).$$

\begin{table}[htbp]
	\centering
	\ttfamily 
	\resizebox{\textwidth}{!}{
		\begin{tabular}{r c c c c c c c c c c c c c c c}
			& 0 & 1  & 2   & 3   & 4    & 5    & 6    & 7    & 8    & 9    & 10   & 11  & 12 & 13 & 14 \\
			total: & 1 & 38 & 211 & 654 & 1441 & 2457 & 3332 & 3597 & 3058 & 2013 & 1002 & 364 & 91 & 14 & 1  \\
			0: & 1 & .  & .   & .   & .    & .    & .    & .    & .    & .    & .    & .   & .  & .  & .  \\
			1: & . & 38 & 211 & 654 & 1441 & 2457 & 3332 & 3597 & 3058 & 2013 & 1002 & 364 & 91 & 14 & 1  \\
		\end{tabular}%
	}
	\normalfont 
	\caption{Graded Betti numbers for the ideal of $\Gamma(\mathbb{Z}/30\mathbb{Z})$.}
	\label{tab:betti_numbers_wide}
\end{table}
Table \ref{tab:betti_numbers_wide} gives the Betti table obtained by Macaulay2 \cite{Macaulay2}, which agrees with our calculation, and discards calculation of Theorem 1.10 \cite{DungVu2026}.
\medskip
The above example explains why the type blocks are constant.  The covers grow only by positional shifts, not by adding previous vertices from the same center class.

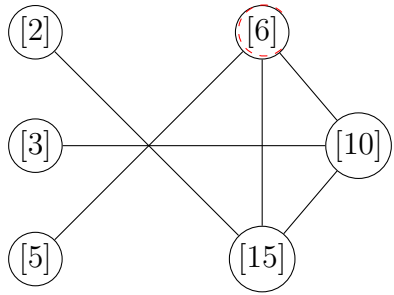
\begin{figure}[H]
\centering
\begin{tikzpicture}[scale=0.75, every node/.style={circle,draw,minimum size=7mm,inner sep=1pt}]
\node (A2) at (0,2) {$[2]$};
\node (A3) at (0,0) {$[3]$};
\node (A5) at (0,-2) {$[5]$};
\node (B6) at (4,2) {$[6]$};
\node (B10) at (5.7,0) {$[10]$};
\node (B15) at (4,-2) {$[15]$};
\draw (B6)--(A5);
\draw (B10)--(A3);
\draw (B15)--(A2);
\draw (B6)--(B10);
\draw (B6)--(B15);
\draw (B10)--(B15);
\draw[dashed,red] ($(B6)+(0.35,0.35)$) arc (45:315:0.45);
\end{tikzpicture}
\caption{Compressed class pattern for $\Gamma(\mathbb{Z}/30\mathbb{Z})$.  The dashed loop indicates that no loop or same-class clique is present in the ordinary graph.}
\label{fig:z30-compressed}
\end{figure}
Figure \ref{fig:z30-compressed} shows the compressed adjacency pattern, but the dashed mark warns against replacing a class by a clique.  Lifting from the compressed graph to the ordinary graph is exactly where the correction enters.
\begin{figure}[H]
	\centering
	\begin{tikzpicture}[
		scale=0.95,
		every node/.style={circle, draw, minimum size=5mm, inner sep=1pt, font=\small},
		edge/.style={draw, thin}
		]
		
		
		\node (v2)  at (-6, 3.5) {2};
		\node (v4)  at (-6, 2.5) {4};
		\node (v8)  at (-6, 1.5) {8};
		\node (v14) at (-6, 0.5) {14};
		\node (v16) at (-6,-0.5) {16};
		\node (v22) at (-6,-1.5) {22};
		\node (v26) at (-6,-2.5) {26};
		\node (v28) at (-6,-3.5) {28};
		
		\node (v3)  at (-2, 3) {3};
		\node (v9)  at (-2, 2) {9};
		\node (v21) at (-2, 1) {21};
		\node (v27) at (-2, 0) {27};
		
		\node (v5)  at (2, 2.5) {5};
		\node (v25) at (2, 1.5) {25};
		
		\node (v6)  at (6, 3) {6};
		\node (v12) at (6, 2) {12};
		\node (v18) at (6, 1) {18};
		\node (v24) at (6, 0) {24};
		
		\node (v10) at (2,-1.5) {10};
		\node (v20) at (2,-2.5) {20};
		
		\node (v15) at (-2,-2) {15};
		
		
		\foreach \a in {2,4,8,14,16,22,26,28}
		\draw[edge] (v\a) -- (v15);
		
		\foreach \b in {3,9,21,27}
		\foreach \e in {10,20}
		\draw[edge] (v\b) -- (v\e);
		
		\foreach \c in {5,25}
		\foreach \d in {6,12,18,24}
		\draw[edge] (v\c) -- (v\d);
		
		\foreach \d in {6,12,18,24}
		\foreach \e in {10,20}
		\draw[edge] (v\d) -- (v\e);
		
		\foreach \d in {6,12,18,24}
		\draw[edge] (v\d) -- (v15);
		
		\foreach \e in {10,20}
		\draw[edge] (v\e) -- (v15);
		
	\end{tikzpicture}
	\caption{Zero divisor graph $\Gamma(\mathbb{Z}/30\mathbb{Z})$.}
	\label{fig 1}
\end{figure}
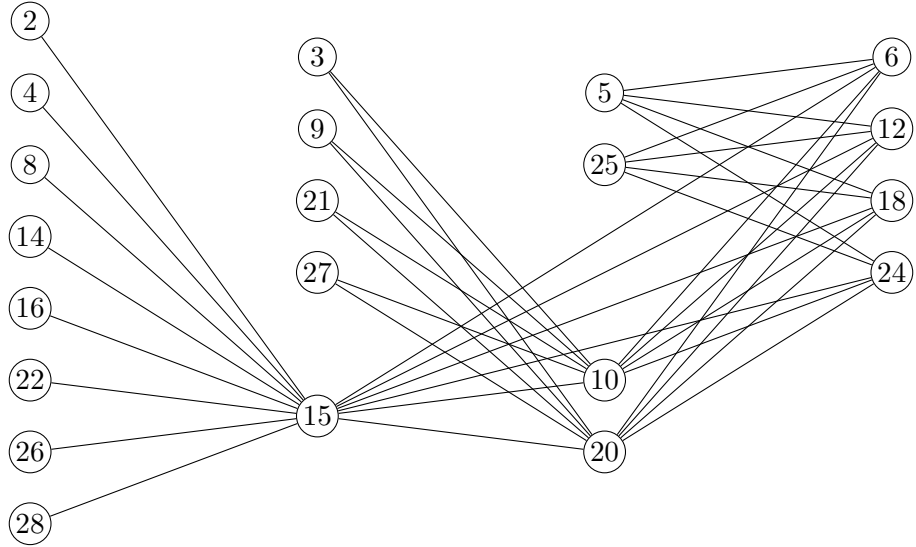
\begin{table}[H]
	\centering
	\small
	\renewcommand{\arraystretch}{1.2}
	\begin{tabularx}{\textwidth}{c c c X}
		\toprule
		$n$
		&
		\makecell{Theorem \ref{thm:three-prime-betti}\\$\beta_1$}
		&
		\makecell{Uncorrected repeated-block value\\
			(Theorem 1.10 \cite{DungVu2026})}
		&
		\makecell{Conclusion}
		\\
		\midrule
		$30$ & $38$  & $66$  & Uncorrected value overcounts edges. \\
		$42$ & $56$  & $108$ & Same-class overcount persists. \\
		$70$ & $106$ & $182$ & Larger $r$ amplifies the error. \\
		\bottomrule
	\end{tabularx}
	\caption{Comparison of the first Betti number of $\Gamma(\mathbb{Z}/pqr\mathbb{Z})$.}
	\label{tab:three-prime-beta-one-comparison}
\end{table}

\medskip

Table \ref{tab:three-prime-beta-one-comparison} compares the first Betti number with the value obtained from treating constant blocks as decreasing blocks.  Since $\beta_1$ must equal the number of edges, the discrepancy is decisive.

\medskip

\noindent The following open problem indicates a natural direction beyond the present complete classification.
\begin{openproblem}\label{op:noncochordal}
For $n$ outside the three cochordal families, determine effective formulae or sharp bounds for the nonlinear Betti numbers of $S/I(\Gamma(\mathbb{Z}/n\mathbb{Z}))$ in terms of the prime factorization of $n$.
\end{openproblem}

\medskip

The problem is open because induced matchings force nonlinear syzygies, and the type-sequence method no longer applies directly.  Hochster's formula remains available, but direct simplicial homology over all induced subgraphs is usually too large for closed arithmetic expressions.

\medskip

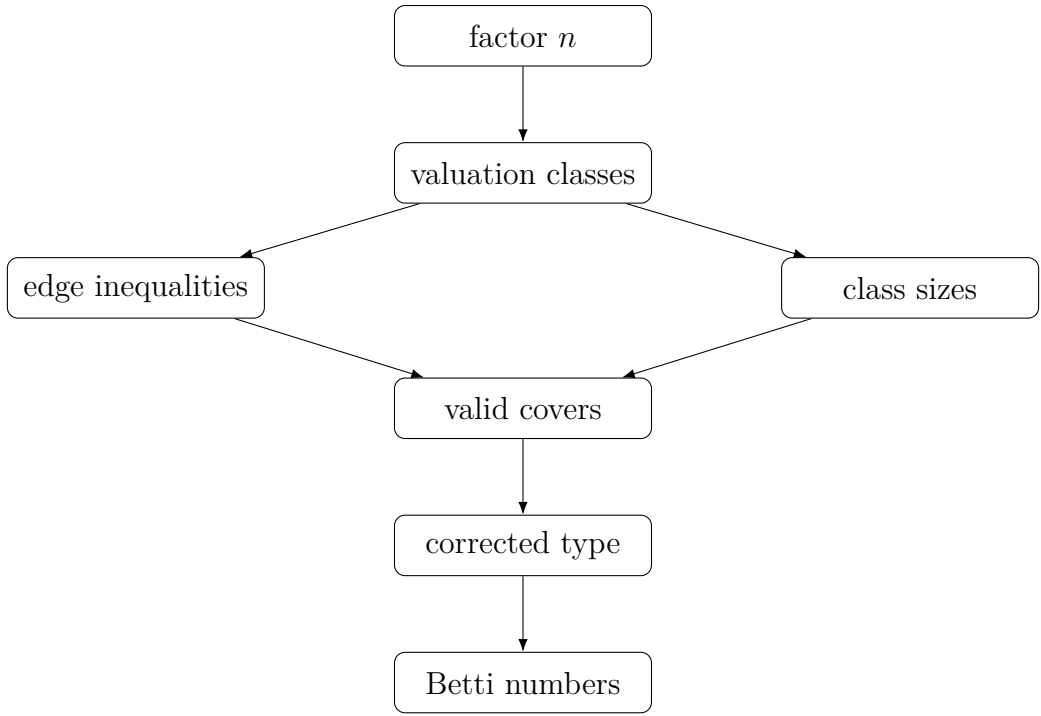
\begin{figure}[H]
	\centering
	\begin{tikzpicture}[>=Latex,node distance=10mm]
		\tikzstyle{box}=[rectangle,draw,rounded corners,align=center,minimum width=34mm,minimum height=8mm]
		\node[box] (factor) {factor $n$};
		\node[box,below=of factor] (classes) {valuation classes};
		\node[box,below left=of classes,xshift=-10mm] (edge) {edge inequalities};
		\node[box,below right=of classes,xshift=10mm] (sizes) {class sizes};
		\node[box,below=of classes,yshift=-13mm] (covers) {valid covers};
		\node[box,below=of covers] (type) {corrected type};
		\node[box,below=of type] (betti) {Betti numbers};
		\draw[->] (factor)--(classes);
		\draw[->] (classes)--(edge);
		\draw[->] (classes)--(sizes);
		\draw[->] (edge)--(covers);
		\draw[->] (sizes)--(covers);
		\draw[->] (covers)--(type);
		\draw[->] (type)--(betti);
	\end{tikzpicture}
	\caption{Expanded validation path from arithmetic data to homological data.}
	\label{fig:expanded-validation-path}
\end{figure}

\medskip

The diagram in Figure \ref{fig:expanded-validation-path} makes explicit that the corrected proof has two independent inputs: the edge inequalities and the class-size counts.  Both are needed before the type formula can be applied safely.

\medskip

\begin{table}[H]
	\begin{center}
\begin{tabular}{cccc}
\toprule
Family & block types &  correction term & output degree range\\
\midrule
$p^a$ & decreasing only & $\binom{p^{a-m}-1}{h+1}$ & $1\leq h\leq p^{a-1}-2$\\
$p^a q$ & mixed &  $\binom{p^a-1}{h+1}$ & $1\leq h\leq \operatorname{pd}$\\
$pqr$ & constant only &  $\binom{p+q+r-3}{h+1}$ & $1\leq h\leq qr+p-3$\\
\bottomrule
\end{tabular}
\end{center}
\caption{Correction terms for three cochordal cases of $\Gamma(\mathbb{Z}/n\mathbb{Z})$.}
\label{tab 5}
\end{table}
 
Table \ref{tab 5} compares the three cochordal cases of $\Gamma(\mathbb{Z}/n\mathbb{Z})$.  The correction term is different in each family because the number of center steps is different.

\medskip

\section{Hilbert series of the corrected quotient rings}\label{sec:hilbert-series}

 In this section, we determine the Hilbert series of $\Gamma(\mathbb{Z}/n\mathbb{Z}).$ Let $G=\Gamma(\mathbb{Z}/n\mathbb{Z})$ be one of the cochordal zero-divisor graphs considered above, and as usual $S=\mathbb{K}[x_v:v\in V(G)]$.  Also order of $G$ is $N$.  Since $I(G)$ has a $2$-linear resolution in the cochordal cases, the Hilbert series is determined by the linear Betti numbers.

\medskip

\noindent The following theorem gives the uniform Hilbert-series formula.
\begin{theorem}[Hilbert series from the corrected linear strand]\label{thm:hilbert-uniform}
Let $G=\Gamma(\mathbb{Z}/n\mathbb{Z})$ be cochordal, let $I=I(G)$, and let $B_h(n)=\beta_h(S/I)$ denote the corrected Betti number supplied by Theorems \ref{thm:prime-power-betti}, \ref{thm:two-prime-betti} and \ref{thm:three-prime-betti}.  If $D=\operatorname{pd}(S/I)$ and $N=|V(G)|$, then
$$
\operatorname{Hilb}_{S/I}(t)=\frac{1+\sum_{h=1}^{D}(-1)^h B_h(n)t^{h+1}}{(1-t)^N}.
$$
If $G$ has no edge, the same formula is interpreted as $\operatorname{Hilb}_{S/I}(t)=\frac{1}{(1-t)^N}$.
\end{theorem}

\begin{proof}
	Since $I=I(G)$ has a $2$-linear resolution, all nonzero graded Betti numbers
	of $S/I$ occur in bidegrees $(h,h+1)$ for $h\geq 1$.  Thus the minimal graded
	free resolution has the form
	{\footnotesize$$
	0\longrightarrow S(-D-1)^{B_D(n)}
	\longrightarrow \cdots
	\longrightarrow S(-4)^{B_3(n)}
	\longrightarrow S(-3)^{B_2(n)}
	\longrightarrow S(-2)^{B_1(n)}
	\longrightarrow S
	\longrightarrow S/I
	\longrightarrow 0.
	$$}
	Here $D=\operatorname{pd}(S/I)$, and $B_h(n)=\beta_h(S/I)$ denotes the total Betti number in
	homological degree $h$. Hilbert series is additive on short exact sequences and, equivalently,
	alternating on a finite free resolution.  Since
	$ \operatorname{Hilb}_S(t)=\frac{1}{(1-t)^N}$
	and $
	\operatorname{Hilb}_{S(-j)}(t)=\frac{t^j}{(1-t)^N}, $
	the above resolution gives
	$$
	\operatorname{Hilb}_{S/I}(t)
	=
	\frac{1}{(1-t)^N}
	+
	\sum_{h=1}^{D}(-1)^h B_h(n)\frac{t^{h+1}}{(1-t)^N}.
	$$
	Now, combining the terms over the common denominator $(1-t)^N$, we have
	$$
	\operatorname{Hilb}_{S/I}(t)
	=
	\frac{1+\sum_{h=1}^{D}(-1)^hB_h(n)t^{h+1}}{(1-t)^N}.
	$$
	 If $G$ has no edge, then $I(G)=0$.  The minimal free resolution is simply
	$$
	0\longrightarrow S\longrightarrow S/I\longrightarrow 0,
	$$
	so $S/I=S$ and therefore $
	\operatorname{Hilb}_{S/I}(t)=\operatorname{Hilb}_S(t)=\frac{1}{(1-t)^N}. $
\end{proof}
\medskip

The above theorem is often the fastest route from the corrected Betti formulae to enumerative data.  The numerator is sensitive to both corrections made in this paper, as changing a low block or omitting the global subtraction changes the coefficient of $t^2$ and therefore the number of quadratic generators.

\medskip

\noindent The following proposition gives the number of variables in the three arithmetic families.
\begin{proposition}\label{prop:vertex-counts}
The denominator exponent $N=|V(\Gamma(\mathbb{Z}/n\mathbb{Z}))|$ is as follows:
$$
N(p^a)=p^{a-1}-1\quad (a\geq 2),
$$
$$
N(p^a q)=p^{a-1}(p+q-1)-1,
$$
and
$$
N(pqr)=pq+pr+qr-p-q-r.
$$
For $n=p$ prime, $N=0$.
\end{proposition}

\begin{proof}
The vertices are the nonzero zero-divisors.  For $p^a$, these are the nonzero multiples of $p$, so there are $p^{a-1}-1$.  For $p^a q$, subtract the units and zero from all residues:
$$
p^a q-\varphi(p^a q)-1=p^a q-p^{a-1}(p-1)(q-1)-1=p^{a-1}(p+q-1)-1.
$$
For $pqr$, similarly,
$$
pqr-(p-1)(q-1)(r-1)-1=pq+pr+qr-p-q-r.
$$
If $n=p$ is prime, there are no nonzero zero-divisors.
\end{proof}

\medskip

The proposition supplies the denominator in unreduced form.  Reduction to a denominator of the form $(1-t)^{\dim S/I}$ is possible, but the unreduced expression is the direct output of the minimal free resolution and is the most convenient for computation.

\medskip

\noindent The following corollary writes the three family-specific Hilbert series in closed form.
\begin{corollary}\label{cor:family-hilbert}
In the prime-power, two-prime, and square-free three-prime cases, the Hilbert series are obtained by substituting the corresponding corrected Betti functions into Theorem \ref{thm:hilbert-uniform}.  Explicitly,
$$
\operatorname{Hilb}_{p^a}(t)=\frac{1+\sum_{h=1}^{p^{a-1}-2}(-1)^h\left(\sum_{j=m}^{a-1}p^{a-j-1}(p-1)\binom{p^j-2}{h}-\binom{p^{a-m}-1}{h+1}\right)t^{h+1}}{(1-t)^{p^{a-1}-1}}
$$
for $a\geq 2$, with the evident zero-edge convention.  The corresponding formulae for $p^a q$ and $pqr$ are obtained by replacing the parenthesized coefficient with the right sides of Theorems \ref{thm:two-prime-betti} and \ref{thm:three-prime-betti} and using the denominator exponents from Proposition \ref{prop:vertex-counts}.
\end{corollary}

\begin{proof}
	The assertion is obtained by combining the uniform Hilbert-series formula
	with the corrected Betti formulae already proved for the three arithmetic
	families. For the prime-power case, let
	$ m=\left\lceil \frac{a}{2}\right\rceil . $ Then by Theorem \ref{thm:prime-power-betti}, the corrected Betti number is
	$$
	B_h(p^a)
	=
	\sum_{j=m}^{a-1}p^{a-j-1}(p-1)\binom{p^j-2}{h}
	-
	\binom{p^{a-m}-1}{h+1}.
	$$
	Also, by Proposition \ref{prop:vertex-counts},
	$ N=p^{a-1}-1. $
	Substituting these two expressions into Theorem \ref{thm:hilbert-uniform} gives
	$$
	\operatorname{Hilb}_{p^a}(t)=
	\frac{
		1+\sum_{h=1}^{p^{a-1}-2}(-1)^h
		\left(
		\sum_{j=m}^{a-1}p^{a-j-1}(p-1)\binom{p^j-2}{h}
		-
		\binom{p^{a-m}-1}{h+1}
		\right)t^{h+1}
	}
	{(1-t)^{p^{a-1}-1}}.
	$$
	 The same argument applies to the two-prime and three-prime cases.  Namely,
	one replaces $B_h(p^a)$ by the corrected coefficient supplied by Theorem
	\ref{thm:two-prime-betti} or by Theorem \ref{thm:three-prime-betti}, and we use
	the denominator exponent from Proposition \ref{prop:vertex-counts}.  The upper summation
	limit is the corresponding projective dimension, as recorded in Corollaries
	\ref{cor:two-prime-pd} and \ref{cor:three-prime-pd}.  The zero-edge cases are covered by the convention in
	Theorem \ref{thm:hilbert-uniform}.
\end{proof}
\medskip

The above corollary is intentionally written with the same corrected coefficients as the Betti theorems.  This prevents a common error, using corrected projective dimension but an uncorrected numerator.

\medskip

\noindent The following proposition records concrete Hilbert series for the examples computed earlier.\\
For $n=8$, Corollary \ref{cor:prime-power-a3} gives $B_1=2$ and $B_2=1$, while Corollary \ref{prop:vertex-counts} gives $N=3$.  Thus, we have
$$
\operatorname{Hilb}_{\Gamma(\mathbb{Z}/8\mathbb{Z})}(t)=\frac{1-2t^2+t^3}{(1-t)^3},
$$
 For $n=12$, Exercise  \ref{ex:z12} gives $(B_1,B_2,B_3,B_4)=(8,14,9,2)$ and $N=7$.   Thus, we have
$$
\operatorname{Hilb}_{\Gamma(\mathbb{Z}/12\mathbb{Z})}(t)=\frac{1-8t^2+14t^3-9t^4+2t^5}{(1-t)^7},
$$
For $n=30$, Proposition \ref{prop:z30} gives the listed Betti sequence and Proposition \ref{prop:vertex-counts} gives $N=21$.  Substitution into Theorem \ref{thm:hilbert-uniform}, we have
$$
\operatorname{Hilb}_{\Gamma(\mathbb{Z}/30\mathbb{Z})}(t)=\frac{1-38t^2+211t^3-654t^4+1441t^5-2457t^6+\cdots+t^{15}}{(1-t)^{21}},
$$
where the omitted middle coefficients are determined by Proposition \ref{prop:z30}.

\medskip

These examples show how the Hilbert numerator alternates according to homological degree.  In particular, the coefficient of $t^2$ is always $-|E(G)|$, so the Hilbert series retains the same edge-count diagnostic used earlier.

\medskip

\begin{table}[H]
	\begin{center}
\begin{tabular}{cccc}
\toprule
$n$ & $N$ & corrected Betti sequence & Hilbert numerator\\
\midrule
$8$ & $3$ & $(2,1)$ & $1-2t^2+t^3$\\
$12$ & $7$ & $(8,14,9,2)$ & $1-8t^2+14t^3-9t^4+2t^5$\\
$16$ & $7$ & $(7,16,20,15,6,1)$ & $1-7t^2+16t^3-20t^4+15t^5-6t^6+t^7$\\
\bottomrule
\end{tabular}
\end{center}
\caption{Betti table and Hilbert series for some values of $\Gamma(\mathbb{Z}/n\mathbb{Z}).$}
\label{tab 6}
\end{table}

\medskip

Table \ref{tab 6} compares denominator size, Betti data, and numerator data.  The examples also show that rings with the same number of variables can have very different Hilbert numerators.

\medskip

\noindent The following proposition extracts the Hilbert function from the same numerator.
\begin{proposition}\label{prop:hilbert-function-coefficients}
Let $G$ be one of the cochordal zero-divisor graphs considered in this paper, let $N=|V(G)|$, and let $\beta_i=\beta_i(S/I(G))$.  For every $d\geq0$,
$$
\dim_{\mathbb{K}}(S/I(G))_d=\binom{N+d-1}{d}+\sum_{\substack{i\geq1\\ d\geq i+1}}(-1)^i\beta_i\binom{N+d-i-2}{d-i-1}.
$$
In particular,
$ \dim_{\mathbb{K}}(S/I(G))_0=1,
\dim_{\mathbb{K}}(S/I(G))_1=N, $
and $
\dim_{\mathbb{K}}(S/I(G))_2=\binom{N+1}{2}-|E(G)|. $
\end{proposition}

\begin{proof}
By Theorem \ref{thm:hilbert-uniform}, the Hilbert series is
$$
\frac{1+\sum_{i\geq1}(-1)^i\beta_i t^{i+1}}{(1-t)^N}.
$$
Since $(1-t)^{-N}=\sum_{d\geq0}\binom{N+d-1}{d}t^d$, multiplying the numerator by this power series and reading the coefficient of $t^d$ gives the required formula.
\end{proof}

\medskip

The above proposition gives a direct way to check the Hilbert series without recomputing a resolution.  The degree-two coefficient is especially useful because it counts all quadratic monomials except the edge monomials.

\medskip

\noindent The next proposition records the dimension interpretation used when comparing Hilbert series and Cohen--Macaulayness.
\begin{proposition}\label{prop:dimension-independent-sets}
For any finite simple graph $G$,
$$
\dim S/I(G)=\alpha(G),
$$
where $\alpha(G)$ is the independence number of $G$.  Equivalently, the height of $I(G)$ is the minimum size of a vertex cover of $G$.
\end{proposition}

\begin{proof}
The minimal primes of $I(G)$ are precisely the ideals generated by complements of maximal independent sets, or equivalently by minimal vertex covers.  Therefore the largest dimension of a quotient by a minimal prime is the largest size of an independent set.  This gives $\dim S/I(G)=\alpha(G)$, and the equivalent height statement follows from $\operatorname{ht} I(G)=|V(G)|-\alpha(G)$.
\end{proof}

\medskip

The proposition is standard, but it clarifies why the denominator $(1-t)^N$ in Theorem \ref{thm:hilbert-uniform} is not the final Krull dimension.  Cancellation in the rational function may reduce the denominator to $(1-t)^{\alpha(G)}$ after simplification.

\medskip

\noindent The following corollary gives a simplified denominator after cancellation.
\begin{corollary}\label{cor:reduced-hilbert-denominator}
After cancelling the largest possible power of $1-t$ from the Hilbert numerator in Theorem \ref{thm:hilbert-uniform}, the Hilbert series of $S/I(G)$ has denominator $(1-t)^{\alpha(G)}$.
\end{corollary}

\begin{proof}
For every standard graded algebra, the order of the pole of its Hilbert series at $t=1$ is the Krull dimension.  By Proposition \ref{prop:dimension-independent-sets}, this dimension is $\alpha(G)$ for an edge ideal.  Hence the reduced denominator has the stated exponent.
\end{proof}

\medskip
The above form is sometimes preferable for asymptotic questions.  The unreduced form is better for computation from Betti numbers, while the reduced form displays the dimension.

\medskip

\section{Cohen--Macaulay properties}\label{sec:cohen-macaulay}
 The quotient rings considered in this paper have linear resolutions, but they are almost never Cohen--Macaulay.  The result below separates linearity from Cohen--Macaulayness and gives a complete classification inside the cochordal residue-class families.

\medskip

We use two standard facts from Stanley--Reisner theory.  First, $S/I(G)$ is Cohen--Macaulay only if the independence complex of $G$ is pure, or equivalently, $G$ must be well-covered.  Second, by the Auslander--Buchsbaum formula, $\operatorname{depth}(S/I(G))=N-\operatorname{pd}(S/I(G))$ for $N=|V(G)|$, see \cite{BrunsHerzog1993,Villarreal2015}.

\medskip

\noindent The following theorem gives the complete Cohen--Macaulay classification for the cochordal families.
\begin{theorem}\label{thm:cm-classification}
Let $n\geq 2$ and suppose that $\Gamma(\mathbb{Z}/n\mathbb{Z})$ is cochordal.  Then $S/I(\Gamma(\mathbb{Z}/n\mathbb{Z}))$ is Cohen--Macaulay if and only if $n$ is prime or $n=p^2$ for some prime $p$.
\end{theorem}

\begin{proof}
	In order to prove that $S/I(G)$ is not Cohen--Macaulay, it
	is enough to exhibit two maximal independent sets of different sizes.
	 First suppose that $n=p$ is prime.  Then $\mathbb{Z}/p\mathbb{Z}$ is a field, so it has no
	nonzero zero-divisors.  Hence $\Gamma(\mathbb{Z}/p\mathbb{Z})$ has no vertices and
	$I(\Gamma(\mathbb{Z}/p\mathbb{Z}))=0$.  The quotient is just the ground field $\mathbb{K}$, and is
	therefore Cohen--Macaulay. Now suppose that $n=p^2$.  The nonzero zero-divisors modulo $p^2$ are precisely
	the nonzero multiples of $p$,   and therefore
	$ \Gamma(\mathbb{Z}/p^2\mathbb{Z})=K_{p-1}. $ For $N=p-1$, the minimal vertex covers of $K_N$ are exactly the sets obtained by deleting
	one vertex from $V(K_N)$.  Hence every minimal vertex cover has size $N-1=p-2$,
	so $
	\operatorname{ht} I(K_N)=p-2. $
	By Corollary \ref{cor:prime-power-pd},   we also have
	$
	\operatorname{pd}(S/I(K_N))=p-2.
	$
	The Auslander--Buchsbaum formula gives
	$$
	\operatorname{depth}(S/I(K_N))
	=
	N-\operatorname{pd}(S/I(K_N))
	=
	(p-1)-(p-2)
	=
	1.
	$$
	On the other hand,
	$$
	\dim(S/I(K_N))
	=
	N-\operatorname{ht} I(K_N)
	=
	(p-1)-(p-2)
	=
	1.
	$$
	Thus $
	\operatorname{depth}(S/I(K_N))=\dim(S/I(K_N)), $
	and $S/I(K_N)$ is Cohen--Macaulay.

\medskip

Now assume $n=p^a$ with $a\geq 3$.  Any vertex in $V_{a-1}$ is adjacent to every other vertex of $\Gamma(\mathbb{Z}/p^a\mathbb{Z})$, so a singleton from $V_{a-1}$ is a maximal independent set.  On the other hand, $V_1$ contains at least two vertices and is independent, since $2<a$.  Extending $V_1$ to a maximal independent set gives a maximal independent set of size at least two.  Thus the graph is not well-covered, and $S/I$ is not Cohen--Macaulay.

\medskip

Next let $n=p^a q$ with distinct primes.  If $a=1$, then $\Gamma(\mathbb{Z}/pq\mathbb{Z})=K_{q-1,p-1}$, whose two bipartition classes are maximal independent sets of different sizes because $p\neq q$.  Hence the graph is not well-covered.  If $a\geq 2$, the union $A_1\cup\cdots\cup A_a$ is a maximal independent set of size $(q-1)p^{a-1}$.  Also, $B_0\cup\{b\}$, where $b$ is one fixed vertex of $B_{a-1}$, is a maximal independent set of size $p^{a-1}(p-1)+1$.  These two sizes are different, since equality would imply $(q-p)p^{a-1}=1$, impossible for $a\geq 2$ and distinct primes.  Therefore $S/I$ is not Cohen--Macaulay.

\medskip

Finally let $n=pqr$ with $p<q<r$.  The set $A_p\cup A_q\cup A_r$ is a maximal independent set.  The set $B_{qr}\cup A_q\cup A_r$ is also maximal independent, as every excluded vertex in $A_p$, $B_{pq}$, or $B_{pr}$ is adjacent to an included vertex.  Their sizes differ by
$$
|A_p|-|B_{qr}|=(q-1)(r-1)-(p-1)>0.
$$
Thus the independence complex is not pure, so $S/I$ is not Cohen--Macaulay.   The cochordal zero-divisor graphs are exactly the graphs arising from $n=p$, $n=p^a$, $n=p^aq$, and $n=pqr$ in the classified cases earlier.  Among
	them, the only Cohen--Macaulay quotients occur for $n=p$ and $n=p^2$.  
\end{proof}
\medskip

The above theorem shows that the linearity obtained from cochordality is much weaker than Cohen--Macaulayness.  Except for the empty graph and complete graphs arising from $p^2$, the independence complexes have facets of different dimensions.

\medskip

\noindent The following corollary records the depth in the Cohen--Macaulay cases and the general depth formula used for comparison.
\begin{corollary}\label{cor:depth-defect}
For every cochordal case with $N$ vertices,
$$
\operatorname{depth}(S/I(\Gamma(\mathbb{Z}/n\mathbb{Z})))=N-\operatorname{pd}(S/I(\Gamma(\mathbb{Z}/n\mathbb{Z}))).
$$
In particular, for $n=p^2$ this depth equals $1$, while for the non-Cohen--Macaulay cases the Cohen--Macaulay defect is
$ \dim(S/I)-\operatorname{depth}(S/I)>0. $
\end{corollary}

\begin{proof}
The first equality is the Auslander--Buchsbaum formula applied to the standard graded polynomial ring $S$.  In the case $n=p^2$, $N=p-1$ and $\operatorname{pd}=p-2$, so the depth is $1$, equal to the dimension.  In all remaining cases, Theorem \ref{thm:cm-classification} says that the quotient is not Cohen--Macaulay, and therefore the dimension is strictly larger than the depth.
\end{proof}

\medskip

Corollary \ref{cor:depth-defect} is useful computationally because the projective dimension has already been computed.  It also quantifies the failure of Cohen--Macaulayness without requiring the full independence complex.

\medskip

\begin{example}\label{ex:cm-checks}
For $n=9$, the graph is a single edge $K_2$, so $S/I$ is the hypersurface $\mathbb{K}[x_3,x_6]/(x_3x_6)$ and is Cohen--Macaulay.  For $n=8$, the graph is a two-edge star, the singleton consisting of the center is maximal independent, while the two leaves form another maximal independent set, so the quotient is not Cohen--Macaulay.  For $n=15$, the graph is $K_{4,2}$, whose two parts have different sizes, so the quotient is not Cohen--Macaulay.
\end{example}

\medskip

\noindent The next result records the corresponding unmixedness statement.  It is weaker than Cohen--Macaulayness in general, but in the present families the same exceptions occur.
\begin{proposition}\label{prop:unmixed-classification}
Let $n\geq2$ and suppose that $\Gamma(\mathbb{Z}/n\mathbb{Z})$ is cochordal.  Then $I(\Gamma(\mathbb{Z}/n\mathbb{Z}))$ is unmixed if and only if $n$ is prime or $n=p^2$ for some prime $p$.
\end{proposition}

\begin{proof}
For an edge ideal, unmixedness is equivalent to all minimal vertex covers having the same cardinality, or equivalently to all maximal independent sets having the same cardinality.  The cases $n=p$ and $n=p^2$ are respectively the empty graph and a complete graph, so they are unmixed.  In every other cochordal case, Theorem \ref{thm:cm-classification} gives two maximal independent sets of different sizes.  Hence the corresponding minimal vertex covers have different sizes, and the edge ideal is not unmixed.
\end{proof}

\medskip

Proposition \ref{prop:unmixed-classification} strengthens the negative part of Theorem \ref{thm:cm-classification}.  The failure is not a subtle depth phenomenon, as it is already visible at the level of minimal primes.

\medskip

\begin{example}\label{ex:minimal-primes-expanded}
For $n=8$, the graph $\Gamma(\mathbb{Z}/n\mathbb{Z})$ is the star with center $4$ and leaves $2,6$.  The maximal independent sets are $\{4\}$ and $\{2,6\}$, so the minimal primes are $(x_2,x_6)$ and $(x_4)$, of heights $2$ and $1$.

\medskip

For $n=15$, the graph $\Gamma(\mathbb{Z}/n\mathbb{Z})$ is $K_{4,2}$.  The two bipartition classes give minimal primes of heights $2$ and $4$.  Thus the quotient has a linear resolution but is not equidimensional.
\end{example}

\medskip

The above examples display the obstruction in the simplest possible algebraic language.  Different minimal-prime heights rule out Cohen--Macaulayness before any depth computation is needed.

\medskip

\begin{table}[H]
	\begin{center}
\begin{tabular}{ccccc}
\toprule
$n$ & graph type & linear? & unmixed? & Cohen--Macaulay?\\
\midrule
$p$ & empty & yes & yes & yes\\
$p^2$ & complete & yes & yes & yes\\
$8$ & star $K_{1,2}$ & yes & no & no\\
$15$ & complete bipartite $K_{4,2}$ & yes & no & no\\
$30$ & three-block cochordal graph & yes & no & no\\
\bottomrule
\end{tabular}
\end{center}
\caption{Homological data for values of $n$ in $\Gamma(\mathbb{Z}/n\mathbb{Z})$.}
\label{tab 7}
\end{table}

\medskip

Table \ref{tab 7} emphasizes the main qualitative conclusion.  Cochordality is a resolution property, whereas Cohen--Macaulayness is a purity and depth property.

\section{A chain-ring structure for the graphs}\label{sec:chain-template}
 The purpose of this section is to separate the homological computation from the particular presentation of a finite chain ring.  The same layered graph occurs for $\mathbb{Z}/p^L\mathbb{Z}$, for $\mathbb F_q[x]/(x^L)$ when $q$ is a prime power, and for other finite chain rings with residue field of size $q$.  The construction below keeps the Dung and Vu \cite{DungVu2026} type method, but it evaluates the global correction term explicitly, this is the point at which the corrected formula differs from an uncorrected summation of block contributions.

\medskip

It is useful to isolate the graph-theoretic structure shared by the two families studied later.  Throughout this section $q\geq 2$ is an integer and $L\geq 2$.  In applications, $q$ will be the cardinality of a residue field and $L$ will be the nilpotency index of the maximal ideal.

\medskip

\begin{definition}[Layered chain graph]\label{def:chain-graph}
	Let \(q\ge2\) be an integer and let \(L\ge2\). Define a layered graph $C(q,L)$ as follows.  Its vertex set is a disjoint union
	$ V(C(q,L))=V_1\sqcup V_2\sqcup\cdots\sqcup V_{L-1}, $
	where $
	|V_k|=(q-1)q^{L-k-1} $ for $1\leq k\leq L-1. $
	Two distinct vertices $x\in V_k$ and $y\in V_\ell$ are adjacent in $C(q,L)$ if and only if $k+\ell\geq L$. Figure \ref{fig 1} shows $C(q,L)$ with $q=3$ and $L=4$,  where the blue layer has $18$ vertices, the green layer has $6$ vertices, and the red layer has $2$ vertices.  The edge rule is $k+\ell\geq 4$.
\end{definition}
\begin{figure}[H]
	\centering
	\begin{tikzpicture}[
		scale=0.86,
		v1/.style={circle, draw=black!80, fill=blue!20, minimum size=4mm, inner sep=0pt, font=\small},
		v2/.style={circle, draw=black!80, fill=green!20, minimum size=4mm, inner sep=0pt, font=\small},
		v3/.style={circle, draw=black!80, fill=red!20, minimum size=4mm, inner sep=0pt, font=\small},
		edge13/.style={draw=blue!40, thin, opacity=0.5},
		edge23/.style={draw=orange!80, thick, opacity=0.8},
		edge22/.style={draw=green!60!black, thick, opacity=0.7},
		edge33/.style={draw=red, ultra thick}
		]
		\def\R{5.5}
		\foreach \i/\ang in {1/95,2/85,3/75,4/64,5/52,6/40,7/30,8/20,9/10,10/-5,11/-18,12/-30,13/-40,14/-50,15/-60,16/-70,17/-80,18/-90}{
			\coordinate (c1-\i) at ({\R*cos(\ang)},{\R*sin(\ang)});
		}
		\foreach \i/\ang in {1/130,2/155,3/175,4/195,5/215,6/235}{
			\coordinate (c2-\i) at ({\R*cos(\ang)},{\R*sin(\ang)});
		}
		\coordinate (c3-1) at (-1.8,5.5);
		\coordinate (c3-2) at (-1.8,-5.5);
		\foreach \i in {1,...,18}{\foreach \j in {1,2}{\draw[edge13] (c1-\i)--(c3-\j);}}
		\foreach \i in {1,...,6}{\foreach \j in {1,2}{\draw[edge23] (c2-\i)--(c3-\j);}}
		\foreach \i in {1,...,5}{
			\pgfmathtruncatemacro{\next}{\i+1}
			\foreach \j in {\next,...,6}{\draw[edge22] (c2-\i)--(c2-\j);}
		}
		\draw[edge33] (c3-1)--(c3-2);
		\node[left, font=\Large\bfseries] at (6.85,0) {$V_1$};
		\node[left, font=\Large\bfseries] at (-5.95,0) {$V_2$};
		\node[left, font=\Large\bfseries] at (-1,6) {$V_3$};
		\foreach \i in {1,...,18}{\node[v1] at (c1-\i) {};}
		\foreach \i in {1,...,6}{\node[v2] at (c2-\i) {};}
		\node[v3] at (c3-1) {};
		\node[v3] at (c3-2) {};
	\end{tikzpicture}
	\caption{The graph $C(3,4)$.}
	\label{fig:chain-C34}
\end{figure}
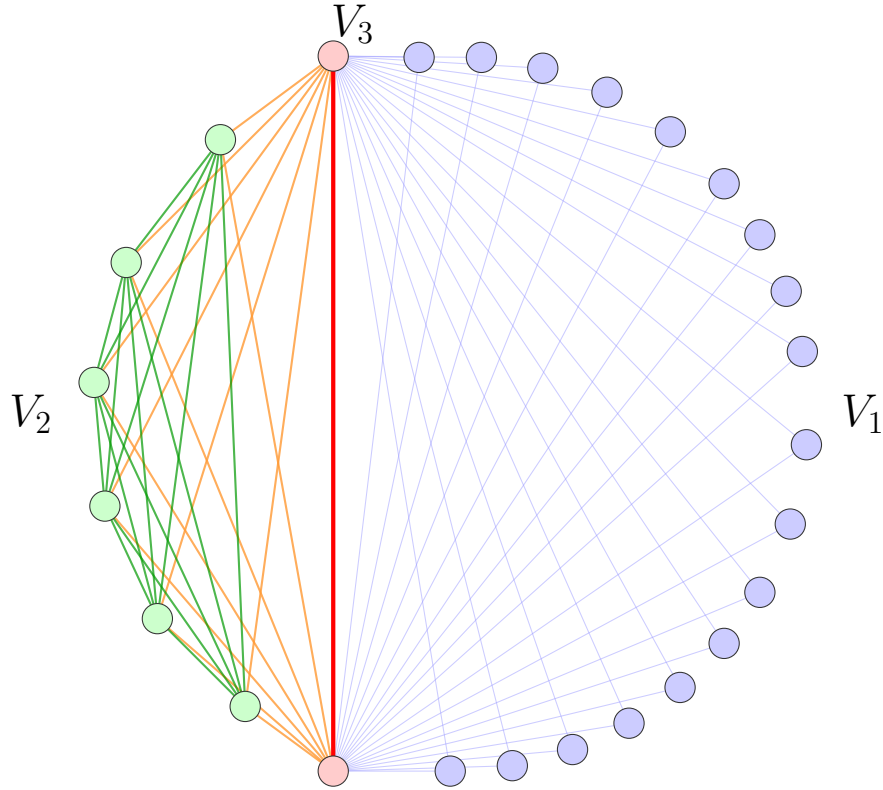
This definition is modeled on a finite chain ring with residue field of size $q$ and nilpotency index $L$.  In that setting $V_k$ consists of the nonzero elements of valuation $k$, and the product of elements of valuations $k$ and $\ell$ is zero precisely when $k+\ell\geq L$. The number of vertices is
$$
|V(C(q,L))|=\sum_{k=1}^{L-1}(q-1)q^{L-k-1}=q^{L-1}-1.
$$
This identity will be used repeatedly when projective dimension and depth are computed from the Betti formula.

\medskip

\medskip

The next lemma gives the explicit cochordal constructible system for $C(q,L)$.  The statement is written in the same language as as in Definition \ref{def:constructible}, but the proof spells out the cover condition because this is the step at which false same-layer edges can otherwise enter the calculation.

\medskip

\begin{lemma}\label{lem:chain-cochordal-system}
	Let $G=C(q,L)$ with $h=\lceil L/2\rceil$, and for $h\leq i\leq L-1$, let
	$ n_i=|V_i|=(q-1)q^{L-i-1} $
	and
	$ M_i=\sum_{s=L-i}^{i-1}|V_s|=q^i-q^{L-i}, $
	where the sum is interpreted as $0$ when $L-i>i-1$.  With vertex ordering $
	V_i=\{v_{i,1},\ldots,v_{i,n_i}\},$ and for $1\leq j\leq n_i$, let
	$$
	U_{i,j}=\left(\bigcup_{s=L-i}^{i-1}V_s\right)\cup\{v_{i,1},\ldots,v_{i,j-1}\}.
	$$
	Then the ordered construction with centers
	$$
	(i,j)=(h,1),(h,2),\ldots,(h,n_h),(h+1,1),\ldots,(L-1,n_{L-1})
	$$
	and covers $U_{i,j}$ is a cochordal constructible system for $G$, and $C(q,L)$ is cochordal.
\end{lemma}

\begin{proof}
	For fixed $i\geq h$ and let the current center be $v_{i,j}$.  A vertex in $V_i$ is adjacent exactly to the vertices in layers $V_s$ satisfying $s\geq L-i$.  If $2i\geq L$, it is also adjacent to the other vertices in $V_i$.  Since $i\geq h$, the inequality $2i\geq L$ holds, so the same-layer vertices that have already appeared are true neighbors of $v_{i,j}$.  Hence every vertex listed in $U_{i,j}$ is adjacent to $v_{i,j}$, and the star with center $v_{i,j}$ and leaf set $U_{i,j}$ introduces only genuine edges of $G$.
	
	\medskip
	
	It remains to verify the vertex-cover condition.  Suppose the construction has reached the pair $(i,j)$, and let $e=\{x,y\}$ be an edge already constructed.  If one endpoint of $e$ lies in the current layer $V_i$, then it must be one of $v_{i,1},\ldots,v_{i,j-1}$, and hence it belongs to $U_{i,j}$.  Otherwise both endpoints lie in layers that occurred before $V_i$, or in lower layers connected to an earlier center.  If $x\in V_a$ and $y\in V_b$ with $a\leq b$.  Since $e$ is an edge, $a+b\geq L$.  Because the larger layer already used as a center is at most $i$, we have $b\leq i$.  The inequality $a+b\geq L$ then gives $a\geq L-b\geq L-i$.  Thus at least one endpoint, indeed the smaller endpoint $x$ unless $b=i$ is the current same layer, lies in $\bigcup_{s=L-i}^{i-1}V_s$ or among the earlier vertices of $V_i$.  Therefore every previously constructed edge meets $U_{i,j}$.
	
	\medskip
	
	Finally, every edge of $G$ appears exactly once, as if an edge joins $V_a$ and $V_b$ with $a<b$, it is introduced when the center in $V_b$ is processed. If it joins two vertices in the same layer $V_i$, it is introduced when the later of the two vertices in the fixed order is processed.  Hence the construction produces precisely $G$.  The cochordality follows from the constructible characterization in Theorem \ref{thm:constructible-characterization}.
\end{proof}

\medskip

The next result records the type sequence determined by the construction.

\medskip

\begin{corollary}\label{cor:type-chain}
	With notation as in Theorem \ref{lem:chain-cochordal-system}, the corresponding type list is the concatenation
	$ (s_{L-1},s_{L-2},\ldots,s_h), $ where
	$$
	s_i=\bigl(q^i-q^{L-i-1}-1,
	q^i-q^{L-i-1}-2,
	\ldots,
	q^i-q^{L-i}\bigr).
	$$
	The length of $s_i$ is $n_i=(q-1)q^{L-i-1}$.
\end{corollary}

\begin{proof}
	For fixed $i$, the size of $U_{i,j}$ is
	$$
	|U_{i,j}|=M_i+j-1=q^i-q^{L-i}+j-1.
	$$
	The type sequence is written in reverse order from the construction order.  Hence the block contributed by $V_i$ is the decreasing sequence
	$$
	M_i+n_i-1,\ M_i+n_i-2,\ \ldots,\ M_i.
	$$
	Since $n_i=(q-1)q^{L-i-1}$, its first term is
	$$
	q^i-q^{L-i}+(q-1)q^{L-i-1}-1=q^i-q^{L-i-1}-1,
	$$
	and its last term is $q^i-q^{L-i}$. 
\end{proof}

\medskip

For the graph $C(q,L),$ there is, however, a second structural point which is worth recording.  The same graph is a threshold graph.  Thus the valuation layers used in the construction are not only algebraic layers, as they are also threshold-weight classes.

\medskip

We recall the standard definition.  A finite simple graph $H$ is called a threshold graph if there are real numbers $w(v)$ for $v\in V(H)$ and a real number $T$ such that, for any two distinct vertices $u,v$, we have
$$
\{u,v\}\in E(H)\quad\Longleftrightarrow\quad w(u)+w(v)\ge T.
$$
Equivalently, threshold graphs are precisely the graphs that can be built from one vertex by repeatedly adding either an isolated vertex or a dominating vertex.  They are also characterized as the graphs with no induced $P_4$, $C_4$, or $2K_2$, see \cite{ChvatalHammer,maPeledThreshold}.

\medskip

The following result gives the promised threshold representation of $C(q,L)$.

\medskip

\begin{theorem}\label{thm:CqL-threshold}
	For $q\ge2$ and $L\ge2$, the graph $C(q,L)$ is a threshold graph.  More precisely, if $x\in V_k$, with $w(x)=k$, and take the threshold $T=L$.  Then, for distinct vertices $x\in V_k$ and $y\in V_\ell$,
	$$
	\{x,y\}\in E(C(q,L))\quad\Longleftrightarrow\quad w(x)+w(y)\ge T.
	$$
\end{theorem}

\begin{proof}
	By definition, the vertex set of $C(q,L)$ is the disjoint union
	$ V(C(q,L))=V_1\sqcup V_2\sqcup\cdots\sqcup V_{L-1}, $
	where $|V_k|=(q-1)q^{L-k-1}$.  The adjacency rule is also part of the definition of threshold, as if $x\in V_k$ and $y\in V_\ell$ are distinct, then
	$$
	\{x,y\}\in E(C(q,L))\quad\Longleftrightarrow\quad k+\ell\ge L.
	$$
	Now assign to every vertex in the $k$-th layer the weight $k$.  Thus $w(x)=k$ for $x\in V_k$ and $w(y)=\ell$ for $y\in V_\ell$.  Thus,
	$ w(x)+w(y)=k+\ell,$ and therefore
	$$
	\{x,y\}\in E(C(q,L))
	\quad\Longleftrightarrow\quad
	k+\ell\ge L
	\quad\Longleftrightarrow\quad
	w(x)+w(y)\ge T,
	$$
	where $T=L$.  This is exactly the weighted definition of a threshold graph.  Hence $C(q,L)$ is threshold.
\end{proof}

\medskip

The proof shows that no artificial weights are being introduced.  The threshold weight of a vertex is precisely the index of the valuation layer in which the vertex lies.  In the chain-ring examples, this is the same integer that measures the order of divisibility by the chosen generator of the maximal ideal.

\medskip

It is also useful to record the corresponding statement for the complement.

\medskip

\begin{lemma}\label{lem:CqL-complement-threshold}
	The complement $\overline{C(q,L)}$ is a threshold graph.  More explicitly, if $x\in V_k$, then the weights
	$ w'(x)=-k $
	and the threshold $T'=1-L$ give a threshold representation of $\overline{C(q,L)}$.
\end{lemma}

\begin{proof}
	Let $x\in V_k$ and $y\in V_\ell$ be distinct vertices.  In the complement, $x$ and $y$ are adjacent exactly when they are not adjacent in $C(q,L)$.  By the defining rule of $C(q,L)$, this is equivalent to
	$ k+\ell<L. $
	Since $k+\ell$ and $L$ are integers, the preceding strict inequality is the same as
	$ k+\ell\le L-1$ or 
	$ -k-\ell\ge 1-L. $
	With $w'(x)=-k$, $w'(y)=-\ell$, and $T'=1-L$, the last inequality becomes
	$ w'(x)+w'(y)\ge T'. $
	Thus two distinct vertices are adjacent in $\overline{C(q,L)}$ exactly when the sum of their new weights is at least $T'$.  This proves that the complement is threshold.
\end{proof}

\medskip

The next corollary explains why the threshold viewpoint is compatible with the cochordal theory used earlier.

\medskip

\begin{corollary}\label{cor:CqL-threshold-cochordal}
	For all $q\ge2$ and $L\ge2$, the graph $C(q,L)$ is chordal and cochordal.  Consequently, whenever $I(C(q,L))\ne0$, the edge ideal $I(C(q,L))$ has a $2$-linear resolution.
\end{corollary}

\begin{proof}
	A threshold graph has no induced $P_4$, $C_4$, or $2K_2$.  In particular, it has no induced $C_4$.  It also has no induced cycle of length at least $5$, because any induced cycle with at least five vertices contains four consecutive vertices inducing a $P_4$.  Hence a threshold graph has no induced cycle of length at least $4$, and therefore it is chordal.
	
	\medskip
	
	By Theorem~\ref{thm:CqL-threshold}, the graph $C(q,L)$ is threshold, so it is chordal.  By Lemma~\ref{lem:CqL-complement-threshold}, the complement $\overline{C(q,L)}$ is also threshold, so it is chordal as well.  Thus $C(q,L)$ is cochordal.
	
	\medskip
	
	Finally, from Theorem \ref{thm:froberg}, the edge ideal of a graph has a $2$-linear resolution if and only if the complement of the graph is chordal.  Since $\overline{C(q,L)}$ is chordal, the edge ideal $I(C(q,L))$ has a $2$-linear resolution when it is nonzero.
\end{proof}

\medskip

We next spell out the nested-neighborhood structure.  This is often the most convenient internal test for recognizing the threshold property in a layered graph.

\medskip

\begin{proposition}\label{prop:CqL-split-nested}
	Let $h=\lceil L/2\rceil$ and let 
	$ A=V_1\sqcup\cdots\sqcup V_{h-1},$ and $B=V_h\sqcup\cdots\sqcup V_{L-1}. $
	Then $A$ is an independent set and $B$ is a clique.  Moreover, the neighborhoods are nested by layer,  if $1\le i<i'<h$, then
	$ N(V_i)\cap B\subseteq N(V_{i'})\cap B, $
	and if $h\le j<j'\le L-1$, then the adjacency range of $V_{j'}$ contains the adjacency range of $V_j$.
\end{proposition}

\begin{proof}
	First suppose $x\in V_i$ and $y\in V_j$ with $i,j<h$.  If $L=2a$, then $h=a$, so $i+j\le2a-2<L$.  If $L=2a+1$, then $h=a+1$, so $i+j\le2a<L$.  In both cases $i+j<L$, and hence $x$ and $y$ are not adjacent.  Therefore $A$ is independent.
	Now suppose $x\in V_i$ and $y\in V_j$ are distinct vertices with $i,j\ge h$.  Then $i+j\ge2h\ge L$, so $x$ and $y$ are adjacent.  Therefore $B$ is a clique.
	For the first nesting assertion, take $1\le i<i'<h$ and let $z\in B$ be adjacent to a vertex of $V_i$.  If $z\in V_s$, then $i+s\ge L$.  Since $i'>i$, we have $i'+s\ge i+s+1>L$, and in particular $i'+s\ge L$.  Thus $z$ is adjacent to every vertex in $V_{i'}$, and hence $N(V_i)\cap B\subseteq N(V_{i'})\cap B$.
	
	\medskip
	
	For the second assertion, assume $h\le j<j'\le L-1$.  If a vertex $z\in V_s$ is adjacent to a vertex in $V_j$, then $s+j\ge L$.  Since $j'>j$, it follows that $s+j'\ge L$.  Therefore every layer that is adjacent to $V_j$ is also adjacent to $V_{j'}$, apart from the ordinary convention that a vertex is not adjacent to itself.  Thus the neighborhood ranges increase with the layer index.
\end{proof}

\medskip

\medskip

\begin{example}\label{ex:threshold-C34-short}
	For $q=3$ and $L=4$, the graph $C(3,4)$ has layer sizes
	$ |V_1|=18, 
	|V_2|=6, $ and $
	|V_3|=2. $
	The threshold weights are $1$, $2$, and $3$, and the threshold is $T=4$.  Thus vertices in $V_1$ are adjacent precisely to vertices in $V_3$, the layer $V_2$ is a clique, and the layer $V_3$ is adjacent to every other vertex.  This is the same graph displayed earlier in Figure~\ref{fig 1}, the present interpretation only changes the language from valuation layers to threshold weights.
\end{example}

\medskip

\begin{table}[H]
	\centering
	\begin{tabular}{c c c c c}
		\toprule
		$q$ & $L$ & layer sizes & threshold weights & threshold $T$ \\
		\midrule
		$2$ & $2$ & $1$ & $1$ & $2$ \\
		$2$ & $3$ & $2,1$ & $1,2$ & $3$ \\
		$2$ & $4$ & $4,2,1$ & $1,2,3$ & $4$ \\
		$3$ & $4$ & $18,6,2$ & $1,2,3$ & $4$ \\
		$q$ & $L$ & $(q-1)q^{L-k-1}$ & $k$ on $V_k$ & $L$ \\
		\bottomrule
	\end{tabular}
	\caption{Threshold behaviour of graphs of the form $C(q,L)$.}
	\label{tab:threshold-data-only}
\end{table}
Table \ref{tab:threshold-data-only} gives the threshold data for small graphs of the form $C(q,L)$.  The table records only the new threshold description; the homological formulae are those already proved above.
\medskip

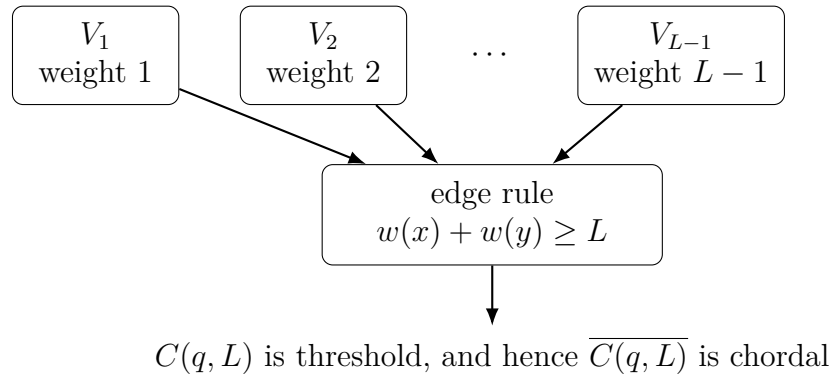
\begin{figure}[H]
	\centering
	\begin{tikzpicture}[
		box/.style={draw, rounded corners, align=center, inner sep=6pt, minimum width=22mm},
		arr/.style={-Latex, thick}
		]
		\node[box] (v1) {$V_1$\\ weight $1$};
		\node[box, right=8mm of v1] (v2) {$V_2$\\ weight $2$};
		\node[right=7mm of v2] (dots) {$\cdots$};
		\node[box, right=7mm of dots] (vl) {$V_{L-1}$\\ weight $L-1$};
		\node[box, below=12mm of dots, minimum width=45mm] (rule) {edge rule\\ $w(x)+w(y)\ge L$};
		\draw[arr] (v1) -- (rule);
		\draw[arr] (v2) -- (rule);
		\draw[arr] (vl) -- (rule);
		\node[below=8mm of rule, align=center] (conclusion) {$C(q,L)$ is threshold, and hence $\overline{C(q,L)}$ is chordal};
		\draw[arr] (rule) -- (conclusion);
	\end{tikzpicture}
	\caption{The threshold interpretation of $C(q,L)$.}
	\label{fig:threshold-block-diagram-reduced}
\end{figure}

\medskip

Figure~\ref{fig:threshold-block-diagram-reduced} summarizes the new threshold representation of $C(q,L)$. The valuation layer is the graph weight, and the nilpotency index $L$ is the threshold..  The same integer that indexes the layer $V_k$ is used as the threshold weight, so the algebraic filtration and the threshold-graph structure coincide.

\medskip
\begin{remark}\label{rem:zero-type}
	When $L$ is even and $i=L/2$, the last entry of $s_i$ is $0$.  This entry corresponds to an empty star and does not change the graph.  If one insists that type sequences contain only positive entries, this terminal zero may be deleted.  The Betti numbers are unchanged because adding a final zero adds $\binom{k-1}{r}$ to the shifted type sum and changes the global correction from $\binom{k-1}{r+1}$ to $\binom{k}{r+1}$, since Pascal's identity shows that the two changes cancel.
\end{remark}

\medskip

The next theorem is the  closed Betti formula for the universal chain-layer graph.  It is the formula that will be specialized in the two ring families below.

\medskip

\begin{theorem}\label{thm:chain-betti}
	Let $G=C(q,L)$ be a graph with order $N=q^{L-1}-1$, and let $I=I(G)\subseteq S=K[x_1,\ldots,x_N]$.  Then, for every $r\geq 1$,
	$$
	\beta_r(S/I)=\beta_{r,r+1}(S/I)=
	\sum_{j=\lceil L/2\rceil}^{L-1}(q-1)q^{L-j-1}\binom{q^j-2}{r}
	-\binom{q^{\lfloor L/2\rfloor}-1}{r+1}.
	$$
	All other graded Betti numbers of $S/I$, except $\beta_{0,0}(S/I)=1$, are zero.  In particular,
	$ \operatorname{pd}_S(S/I)=q^{L-1}-2.$
	If $I\neq 0$, then $I$ has a $2$-linear resolution and $\operatorname{reg}(S/I)=1$, equivalently $\operatorname{reg}(I)=2$.
\end{theorem}

\begin{proof}
	Let $h=\lceil L/2\rceil$, and let $K$ be the total number of entries in the type list of Corollary \ref{cor:type-chain}.  Then
	$$
	K=\sum_{j=h}^{L-1}(q-1)q^{L-j-1}=q^{L-h}-1=q^{\lfloor L/2\rfloor}-1.
	$$
	For a fixed block indexed by $j$, the number of type entries lying to its left is
	$$
	K_j=\sum_{s=j+1}^{L-1}(q-1)q^{L-s-1}=q^{L-j-1}-1.
	$$
	The entries in the block are
	$$
	M_j+n_j-1,\ M_j+n_j-2,\ldots,\ M_j,
	$$
	where $M_j=q^j-q^{L-j}$ and $n_j=(q-1)q^{L-j-1}$.
	
	\medskip
	
	In the  formula of Theorem \ref{thm:type-formula}, the offset increases by one as one moves from left to right through the type list.  Therefore the $u$th entry in the $j$th block, counted from the left inside that block, is
	$ M_j+n_j-u $ 	and its offset is $K_j+u-1$.  The upper argument of the corresponding binomial coefficient is
	$$
	(M_j+n_j-u)+(K_j+u-1)=M_j+n_j+K_j-1=q^j-q^{L-j}+(q-1)q^{L-j-1}-1+q^{L-j-1}-1=q^j-2.
	$$
	Thus every entry in the same block contributes the same binomial coefficient, and the whole block contributes
	$$
	n_j\binom{q^j-2}{r}=(q-1)q^{L-j-1}\binom{q^j-2}{r}.
	$$
	Summing over $j=h,\ldots,L-1$ and subtracting the global correction $\binom K{r+1}$ gives the required formula.
	
	\medskip
	
	The largest possible homological degree comes from the block $j=L-1$, where the upper binomial argument is $q^{L-1}-2=N-1$.  The correction term has upper argument $q^{\lfloor L/2\rfloor}-1$, which is strictly smaller when $L\geq 3$, and the case $L=2$ is immediate.  Hence the last nonzero total Betti number occurs in homological degree $q^{L-1}-2$.  The graded positions and the linearity statement follow from cochordality and Theorem \ref{thm:type-formula}.
\end{proof}

\medskip

The next proposition computes the independence number and the height of the edge ideal.  It will later give the Krull dimension of the quotient rings.

\medskip

\begin{proposition}\label{prop:chain-alpha}
	Let $G=C(q,L)$.  Then
	$$
	\alpha(G)=
	\begin{cases}
		q^{2a-1}-q^a+1, & L=2a,\\[2pt]
		q^{2a}-q^a, & L=2a+1.
	\end{cases}
	$$
	Consequently,
	$$
	\operatorname{height} I(G)=|V(G)|-\alpha(G)=
	\begin{cases}
		q^a-2, & L=2a,\\[2pt]
		q^a-1, & L=2a+1.
	\end{cases}
	$$
\end{proposition}

\begin{proof}
	Let $F$ be an independent set and let $s$ be the largest layer index represented in $F$.  If $F$ contains a vertex from $V_s$, then every other layer represented in $F$ must be among $V_1,\ldots,V_{L-s-1}$, because a vertex in $V_k$ with $k+s\geq L$ would be adjacent to the chosen vertex in $V_s$.  Moreover, if $2s\geq L$, then $F$ contains at most one vertex from $V_s$.
	
	\medskip
	
	Assume first that $L=2a$.  The set
	$ V_1\cup V_2\cup\cdots\cup V_{a-1} $ 	is independent, and one may add one vertex from $V_a$.  Hence
	$$
	\alpha(G)\geq \sum_{k=1}^{a-1}(q-1)q^{2a-k-1}+1=q^{2a-1}-q^a+1.
	$$
	If an independent set uses no layer $V_s$ with $s\geq a$, it is contained in $V_1\cup\cdots\cup V_{a-1}$ and is smaller by one.  If it uses some $V_s$ with $s\geq a$, then it contains at most one vertex from $V_s$ and may use only the lower layers $V_1,\ldots,V_{2a-s-1}$; this number is maximized at $s=a$.  Therefore the displayed lower bound is sharp.
	
	\medskip
	
	Now assume that $L=2a+1$.  The union $V_1\cup\cdots\cup V_a$ is independent because $a+a<2a+1$, and its size is
	$$
	\sum_{k=1}^{a}(q-1)q^{2a+1-k-1}=q^{2a}-q^a.
	$$
	If an independent set uses a layer $V_s$ with $s\geq a+1$, then it loses all layers $V_{2a+1-s},\ldots,V_a$ and gains at most one vertex from $V_s$.  This cannot increase the size, since the omitted layer $V_a$ alone has $q-1\geq 1$ vertices and the lower layers have nonnegative size.  Thus the asserted value of $\alpha(G)$ follows.  The height formula is the standard equality between the height of an edge ideal and the vertex-cover number.
\end{proof}

\medskip

The independence polynomial ($ F_G(y)=\sum_{F\in \operatorname{Ind}(G)}y^{|F|} $, where $\operatorname{Ind}(G)$ represents the set of all independent sets of the graph $G$) gives an independent check on the Hilbert series and on the dimension obtained from Proposition \ref{prop:chain-alpha}.

\medskip

\begin{proposition}\label{prop:chain-independence-polynomial}
	Let $G=C(q,L)$ and let $n_k=(q-1)q^{L-k-1}$.  The independence polynomial 
	is
	$$
	F_G(y)=1+\sum_{s=1}^{L-1}A_s(y)\prod_{k=1}^{\min\{s-1,L-s-1\}}(1+y)^{n_k},
	$$
	where
	$$
	A_s(y)=
	\begin{cases}
		(1+y)^{n_s}-1, & 2s<L,\\[2pt]
		n_s y, & 2s\geq L.
	\end{cases}
	$$
	Consequently,
	$ \operatorname{Hilb}_{S/I(G)}(t)=F_G\left(\frac{t}{1-t}\right),$ or equivalently, if $F_G(y)=\sum_{r\geq 0}f_{r-1}y^r$, then
	$$
	\operatorname{Hilb}_{S/I(G)}(t)=\sum_{r\geq 0}f_{r-1}\frac{t^r}{(1-t)^r}
	=\frac{1}{(1-t)^N}\sum_{r=0}^{\alpha(G)}f_{r-1}t^r(1-t)^{N-r}.
	$$
\end{proposition}

\begin{proof}
	Every nonempty independent set has a unique largest valuation layer, say $V_s$.  If $2s<L$, then $V_s$ has no internal edges, so any nonempty subset of $V_s$ may be chosen, this contributes $(1+y)^{n_s}-1$.  If $2s\geq L$, then $V_s$ is a clique, so at most one vertex may be chosen from $V_s$, this contributes $n_s y$.
	
	After a choice in the largest layer $V_s$ has been made, lower layers may be chosen freely exactly from those $V_k$ with $k<s$ and $k+s<L$.  These are precisely the layers with
	$ k\leq \min\{s-1,L-s-1\}. $
	The product in the formula records these free choices.  Summing over $s$ and adding the empty independent set proves the formula for $F_G(y)$.
	
	\medskip
	
	The Hilbert-series identity is the usual Stanley--Reisner formula for the independence complex.  Since $S/I(G)\cong K[\operatorname{Ind}(G)]$, each independent set $F$ contributes $t^{|F|}/(1-t)^{|F|}$ to the Hilbert series.  This gives $F_G(t/(1-t))$, and rewriting with the common denominator $(1-t)^N$ gives the final expression.
\end{proof}

\medskip

\begin{table}[H]
	\centering
	\begin{tabular}{c c c c c}
		\toprule
		$L$ & $q$ & $N=q^{L-1}-1$ & $\alpha(C(q,L))$ & $\operatorname{pd}(S/I)$ \\
		\midrule
		$3$ & $2$ & $3$ & $2$ & $2$ \\
		$4$ & $2$ & $7$ & $5$ & $6$ \\
		$4$ & $3$ & $26$ & $19$ & $25$ \\
		$5$ & $2$ & $15$ & $12$ & $14$ \\
		$5$ & $3$ & $80$ & $72$ & $79$ \\
		\bottomrule
	\end{tabular}
	\caption{Some homological invariants of $C(q,L)$.}
	\label{tab:chain-samples}
\end{table}
Table \ref{tab:chain-samples} gives the sample numerical values for the universal chain graph $C(q,L)$.  The projective dimension is always one less than the number of vertices, while the dimension of the quotient is the independence number.

\section{The ring \texorpdfstring{$\mathbb{Z}_{2^m}[i]$}{Z_2^m[i]} and its zero-divisor graph}\label{sec:gaussian}
 This section records that the Gaussian residue ring $\mathbb{Z}_{2^m}[i]$ has the same zero-divisor graph as the integer residue ring $\mathbb{Z}/2^{2m}\mathbb{Z}$.  The observation is elementary, but it is useful because it lets the corrected type calculus be applied without recomputing a new family of Betti splittings.

\medskip

Let \(\mathbb{Z}[i]\) be the ring of Gaussian integers.  It is a Euclidean domain and hence a principal ideal domain.  The rational prime \(2\) factors as $2=-i(1+i)^2.$ 
If \(\pi=1+i\), then $(2^m)=(\pi^{2m})$ up to a unit, and therefore
\[
R_m=\mathbb{Z}_{2^m}[i]\cong \mathbb{Z}[i]/(2^m)\cong \mathbb{Z}[i]/(\pi^{2m}).
\]
Thus \(R_m\) is a finite local principal ideal ring, indeed a finite chain ring, with maximal ideal
$\mathfrak{m}=(\overline\pi),$ and $\mathfrak{m}^{2m}=0\ne \mathfrak{m}^{2m-1}.$ Thus, every nonzero element of \(R_m\) has the form \(u\pi^k\), where \(u\in R_m^\times\) and \(0\le k\le 2m-1\).  Now, we define
\[
v(x)=
\begin{cases}
	k, & x=u\pi^k\ne0\text{ with }u\in R_m^\times,\\[2pt]
	+\infty, & x=0.
\end{cases}
\]
Then \(x\) is a unit if and only if \(v(x)=0\), and \(x\in Z^*(R_m)\) if and only if \(1\le v(x)\le 2m-1\).

Thus, an element $x\in R_m$ is a unit if and only if $v(x)=0$, and a nonzero
zero divisor if and only if $1\le v(x)\le 2m-1$. For $k=0,1,\dots,2m$, we define the ideals $\mathfrak m^k = (\pi^k)\subseteq R_m.$  The standard facts about finite chain rings (or a direct counting argument)
gives $|\mathfrak m^k| = 2^{2m-k}$ for $0\le k\le 2m.$ We have the following lemma.	

\begin{lemma}\label{lem:Vk-gaussian}
	For $k=0,1,\dots,2m$ consider $V_k = \{x\in R_m \mid v(x)=k\}
	= \mathfrak m^k\setminus \mathfrak m^{k+1}.$ 
	Then $|V_0|=2^{2m-1}$ and, for $1\le k\le 2m-1$, $|V_k| = 2^{2m-k-1}.$ In particular, the set of nonzero zero divisors of $R_m$ is $Z^{*}(R_m) = \bigcup_{k=1}^{2m-1} V_k$ and its cardinality
	\[
	|Z^{*}(R_m)| = \sum_{k=1}^{2m-1} 2^{2m-k-1}
	= 2^{2m-1}-1.
	\]
\end{lemma}
\begin{proof}
	As $|\mathfrak m^k|=2^{2m-k}$ for all $k$. So, for $0\le k\le 2m-1$, we have
	\[
	|V_k| = |\mathfrak m^k|-|\mathfrak m^{k+1}|
	= 2^{2m-k}-2^{2m-k-1}
	= 2^{2m-k-1}.
	\]
	Now, summing from $k=1$ to $2m-1$, we obtain
	\[
	|Z^{*}(R_m)|
	= \sum_{k=1}^{2m-1} 2^{2m-k-1}
	= \sum_{j=0}^{2m-2} 2^{j}
	= 2^{2m-1}-1,
	\]
	where we used the substitution $j=2m-k-1$.
\end{proof}

\medskip

The next lemma gives the zero-product rule between layers.

\medskip

\begin{lemma}\label{lem:adjacency-gaussian}
	Let $x\in V_k$ and $y\in V_\ell$, where $1\leq k,\ell\leq 2m-1$ and $x\neq y$.  Then
	$$
	xy=0\quad\Longleftrightarrow\quad k+\ell\geq 2m.
	$$
	Thus
	$$
	E(\Gamma(R_m))=\bigl\{\{x,y\}\mid x\in V_k,\ y\in V_\ell,\ x\neq y,\ k+\ell\geq 2m\bigr\}.
	$$
\end{lemma}

\begin{proof}
	Let $x=u\pi^k$ and $y=w\pi^\ell$ with $u,w\in R_m^\times$.  Then
	$ xy=uw\pi^{k+\ell}. $
	The product $uw$ is a unit.  Hence $xy$ is zero in $\mathbb{Z}[i]/(\pi^{2m})$ if and only if $\pi^{2m}$ divides $\pi^{k+\ell}$, and this is equivalent to $k+\ell\geq 2m$.
\end{proof}

\medskip

In particular, for $k\geq m$, the induced subgraph on $V_k$ is complete, because two distinct vertices in $V_k$ have valuation sum $2k\geq 2m$.

\medskip

We now give the basic example that will be used again to check the Betti and Hilbert formulas.

\medskip

\begin{example}\label{ex:gaussian-m2-layers}
	 If $m=2$, then	$R_2=\mathbb Z[i]/(4)\ \cong\ \mathbb Z[i]/(\pi^4),$ $|R_2|=2^{2m}=16,$ and $|Z^*(R_2)|=2^{2m-1}-1=7.$	The chain of ideals is $R_2=\mathfrak m^0 \supset \mathfrak m^1 \supset \mathfrak m^2 \supset \mathfrak m^3 \supset \mathfrak m^4=0,$ and	$|\mathfrak m^k|=2^{4-k}$ for $k=0,1,2,3,4$. Concretely (writing elements as residue classes $a+bi\pmod 4$), divisibility by $\pi=1+i$ is equivalent to
	$a\equiv b\pmod 2$. So, we have $V_1=\mathfrak m\setminus \mathfrak m^2=\{\,\pm 1\pm i\,\},$ with $ |V_1|=4,$
	$V_2=\mathfrak m^2\setminus \mathfrak m^3=\{\,2,\ 2i\,\},$ with $|V_2|=2,$ and $V_3=\mathfrak m^3\setminus \mathfrak m^4=\{\,2+2i\,\},$ with $|V_3|=1,$ (see Figure \ref{chain}). For example: $1+i=\pi$ so $v(1+i)=1$; $2=(-i)\pi^2$ so $v(2)=2$; and $2+2i=(-i)\pi^3$ so $v(2+2i)=3$.  As $2m=4$, so two distinct vertices $x\in V_k$, $y\in V_\ell$ are adjacent if and only if $k+\ell\ge 4$. Thus, $V_1$ has no internal edges, as $1+1<4$, $V_2$ is a clique ($2+2=4$), so $\{2,2i\}$ is an edge, and every vertex in $V_3$ is adjacent to every vertex in $V_1$ and $V_2$ ($3+1=4$, $3+2>4$), see Figure \ref{zero div graph 1}. 
\end{example}

\medskip

\begin{figure}[H]
	\centering
	\begin{tikzpicture}[
		box/.style={draw, rounded corners, inner sep=6pt, align=center},
		arr/.style={-Latex, thick}
		]
		\node[box] (m0) {$\mathfrak m^0=R_2$\\$|\mathfrak m^0|=2^{4}=16$};
		\node[box, below=10mm of m0] (m1) {$\mathfrak m^1=(\pi)$\\$|\mathfrak m^1|=2^{3}=8$};
		\node[box, below=10mm of m1] (m2) {$\mathfrak m^2=(\pi^2)=(2)$\\$|\mathfrak m^2|=2^{2}=4$};
		\node[box, below=10mm of m2] (m3) {$\mathfrak m^3=(\pi^3)$\\$|\mathfrak m^3|=2^{1}=2$};
		\node[box, below=10mm of m3] (m4) {$\mathfrak m^4=(\pi^4)=0$\\$|\mathfrak m^4|=1$};
		\draw[arr] (m0)--(m1);
		\draw[arr] (m1)--(m2);
		\draw[arr] (m2)--(m3);
		\draw[arr] (m3)--(m4);
		\node[right=12mm of m1, align=left] {\small
			$V_k=\mathfrak m^k\setminus\mathfrak m^{k+1}$\\[2pt]
			$|V_0|=8,\ |V_1|=4,\ |V_2|=2,\ |V_3|=1$};
	\end{tikzpicture}
	\caption{The ideal chain of $R_2=\mathbb{Z}[i]/(4)\cong \mathbb{Z}[i]/(\pi^4)$ and the sizes $|\mathfrak m^k|=2^{4-k}$.}
	\label{chain}
\end{figure}
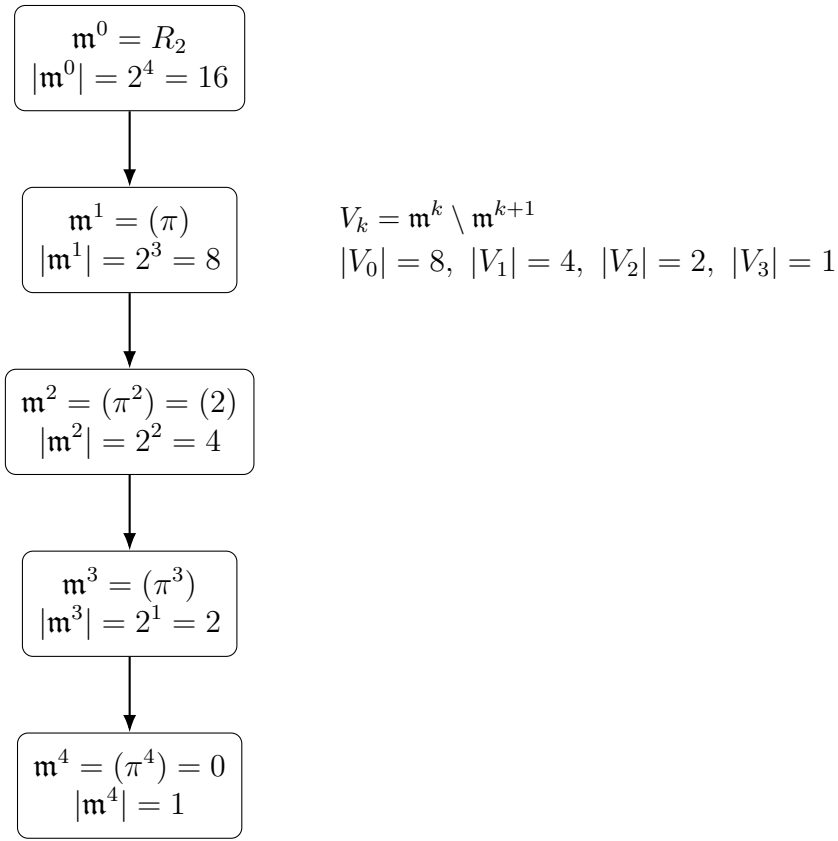
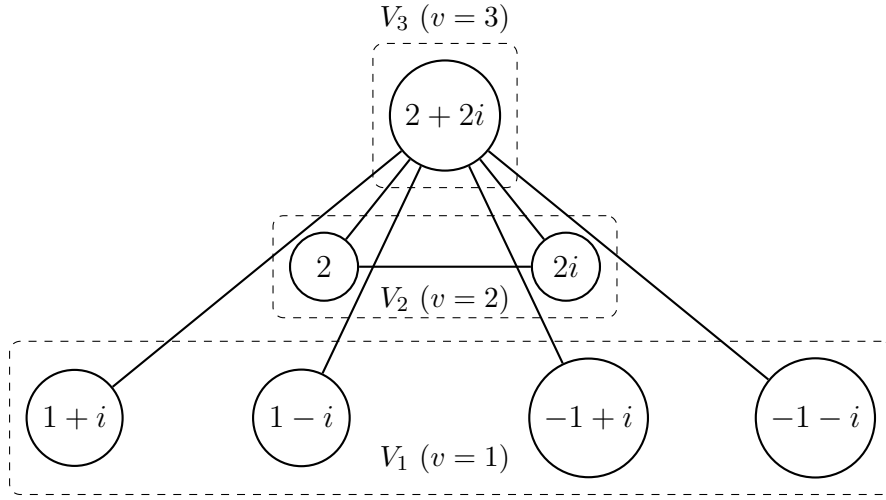
\begin{figure}[H]
	\centering
	\begin{tikzpicture}[
		v1/.style={circle, draw, thick, minimum size=9mm},
		v2/.style={circle, draw, thick, minimum size=9mm},
		v3/.style={circle, draw, thick, minimum size=9mm},
		lab/.style={font=\small},
		ed/.style={thick},
		grp/.style={draw, rounded corners, dashed, inner sep=6pt}
		]
		\node[v3] (a) at (0,4) {$2+2i$};
		\node[lab, above=2mm of a] {$V_3\ (v=3)$};
		\node[v2] (b1) at (-1.6,2) {$2$};
		\node[v2] (b2) at (1.6,2) {$2i$};
		\node[lab, below=1mm of $(b1)!0.5!(b2)$] {$V_2\ (v=2)$};
		\node[v1] (c1) at (-4.9,0) {$1+i$};
		\node[v1] (c2) at (-1.9,0) {$1-i$};
		\node[v1] (c3) at (1.9,0) {$-1+i$};
		\node[v1] (c4) at (4.9,0) {$-1-i$};
		\node[lab, below=2mm of $(c1)!0.5!(c4)$] {$V_1\ (v=1)$};
		\node[grp, fit=(a)] {};
		\node[grp, fit=(b1)(b2)] {};
		\node[grp, fit=(c1)(c2)(c3)(c4)] {};
		\draw[ed] (a)--(b1);
		\draw[ed] (a)--(b2);
		\draw[ed] (a)--(c1);
		\draw[ed] (a)--(c2);
		\draw[ed] (a)--(c3);
		\draw[ed] (a)--(c4);
		\draw[ed] (b1)--(b2);
		\node[align=center, font=\small] at (0,6.5) {$\{x,y\}\in E(\Gamma(R_2))$ if and only if $v(x)+v(y)\geq 4$. Here $V_2$ is a clique, $V_1$ is independent,\\ and $V_3$ connects to all.};
	\end{tikzpicture}
	\caption{The zero-divisor graph $\Gamma(R_2)$ on $Z^*(R_2)=V_1\sqcup V_2\sqcup V_3$.}
	\label{zero div graph 1}
\end{figure}

\medskip

If $A = \mathbb Z/2^{2m}\mathbb Z$, then in $A$ every nonzero element can be written 	uniquely as $u2^k$ with $u$ a unit and $0\le k\le 2m-1$. Define the $2$–adic valuation $w$ by $w(u2^k)=k$ and $w(0)=+\infty$. Then the zero divisors of $A$ are the elements with $1\le w(a)\le 2m-1$. Consider the sets
\[
W_k = \{a\in A \mid w(a)=k\},\qquad\text{for}\qquad 1\le k\le 2m-1.
\]
Then it is clear that $|W_k| = 2^{2m-k-1}$, and for
$a\in W_k$, $b\in W_\ell$ with $a\neq b$, we get  $ab=0$ if and only if $k+\ell\ge 2m.$	Thus, $\Gamma(A)$ has the same valuation layer sizes and the same adjacency
rule as $\Gamma(R_m)$. We consider, it more formally in the following result.	

\begin{proposition}\label{prop:gaussian-iso}
	For every $m\geq 1$, there is a graph isomorphism
	$$
	\Phi_m:\Gamma(R_m)\longrightarrow \Gamma(\mathbb{Z}/2^{2m}\mathbb{Z}).
	$$
\end{proposition}

\begin{proof}
	Let $A=\mathbb{Z}/2^{2m}\mathbb{Z}$.  Every nonzero element of $A$ has a unique form $u2^k$ with $u$ a unit and $0\leq k\leq 2m-1$.  Let
	$ W_k=\{a\in A\mid w(a)=k\},$ for $ 1\leq k\leq 2m-1, $
	where $w$ is the $2$-adic valuation.  Then
	$ |W_k|=2^{2m-k-1}=|V_k|. $
	Choose a bijection $\varphi_k:V_k\to W_k$ for each $k$, and define $\Phi_m(x)=\varphi_k(x)$ when $x\in V_k$.
	
	\medskip
	
	The map $\Phi_m$ is a bijection from $Z^*(R_m)$ to $Z^*(A)$.  If $x\in V_k$ and $y\in V_\ell$ with $x\neq y$, then Lemma \ref{lem:adjacency-gaussian} and the identical valuation rule in $A$ give
	$$
	xy=0\quad\Longleftrightarrow\quad k+\ell\geq 2m\quad\Longleftrightarrow\quad \Phi_m(x)\Phi_m(y)=0.
	$$
	Thus adjacency is preserved and reflected, so $\Phi_m$ is a graph isomorphism.
\end{proof}

\medskip

\section{Cochordality, type sequence, and Betti numbers for \texorpdfstring{$\mathbb{Z}_{2^m}[i]$}{Z_2^m[i]}}\label{sec:gaussian-betti}
 The formulas in this section are not new because of a different graph construction, rather, they are new as corrected evaluations of the chain-type formula with $q=2$ and $L=2m$.  The global correction term $-\binom{2^m-1}{r+1}$ is essential and is retained throughout.

\medskip

The preceding section identifies $\Gamma(R_m)$ with the layered graph $C(2,2m)$.  Hence the general type list in Corollary \ref{cor:type-chain} specializes immediately.

\medskip

\begin{theorem}\label{thm:gaussian-type}
	Let $G_m=\Gamma(R_m)$, where $R_m=\mathbb{Z}_{2^m}[i]$ and $m\geq 1$.  For $m\leq i\leq 2m-1$, define
	$$
	s_i=\bigl(2^i-2^{2m-i-1}-1,
	2^i-2^{2m-i-1}-2,
	\ldots,
	2^i-2^{2m-i}\bigr).
	$$
	Then $G_m$ is cochordal and has type list
	$ (s_{2m-1},s_{2m-2},\ldots,s_m). $
	If the last entry is $0$, it may be deleted without changing the graph or the Betti numbers.
\end{theorem}

\begin{proof}
	By Lemmas \ref{lem:Vk-gaussian} and \ref{lem:adjacency-gaussian}, the graph $G_m$ is exactly $C(2,2m)$: the layer $V_k$ has size $2^{2m-k-1}$ and the edge relation is $k+\ell\geq 2m$.  Applying Corollary \ref{cor:type-chain} with $q=2$ and $L=2m$ gives the stated type list.  The statement about the terminal zero is Remark \ref{rem:zero-type}.
\end{proof}

\medskip

The next theorem gives all graded Betti numbers of the quotient by the edge ideal.

\medskip

\begin{theorem}\label{thm:gaussian-betti}
	Let $G_m=\Gamma(\mathbb{Z}_{2^m}[i])$, let $N=|V(G_m)|=2^{2m-1}-1$, and let $I_m=I(G_m)\subseteq S=K[x_1,\ldots,x_N]$.  For every $r\geq 1$,
	$$
	\beta_r(S/I_m)=\beta_{r,r+1}(S/I_m)=
	\sum_{j=m}^{2m-1}2^{2m-j-1}\binom{2^j-2}{r}
	-\binom{2^m-1}{r+1}.
	$$
	Moreover,
	$ \operatorname{pd}_S(S/I_m)=2^{2m-1}-2. $
	If $m\geq 2$, then $I_m\neq 0$, $I_m$ has a $2$-linear resolution, and $\operatorname{reg}(S/I_m)=1$.  For $m=1$, the graph has one vertex and no edge, so $I_1=(0)$ and $S/I_1\cong K[x]$.
\end{theorem}

\begin{proof}
	Apply Theorem \ref{thm:chain-betti} with $q=2$ and $L=2m$.  Then $N=2^{2m-1}-1$ and $\lfloor L/2\rfloor=m$, giving the displayed Betti formula and projective dimension.  The graph is edgeless exactly for $m=1$, since then $Z^*(R_1)$ has one element.  For $m\geq 2$, the graph is nonempty and cochordal, so the edge ideal has a $2$-linear resolution by Theorem \ref{thm:type-formula}.
\end{proof}

\medskip

The next corollary translates the quotient Betti numbers into the Betti table of the edge ideal itself.

\medskip

\begin{corollary}\label{cor:gaussian-graded-betti-ideal}
	For $m\geq 2$, the graded Betti numbers of the edge ideal $I_m$, viewed as an $S$-module, are determined by
	$ \beta_{r,r+2}(I_m)=\beta_{r+1,r+2}(S/I_m)\quad\text{for }r\geq 0, $
	and all other graded Betti numbers of $I_m$ vanish.  Thus the minimal free resolution of $I_m$ is $2$-linear.
\end{corollary}

\begin{proof}
	The short exact sequence
	$$
	0\longrightarrow I_m\longrightarrow S\longrightarrow S/I_m\longrightarrow 0
	$$
	identifies the positive homological Betti numbers of $S/I_m$ with the Betti numbers of $I_m$ shifted by one homological degree.  Since the quotient has nonzero graded Betti numbers only in positions $(r,r+1)$, the ideal has nonzero graded Betti numbers only in positions $(r,r+2)$.
\end{proof}

\medskip

\begin{example}\label{ex:m2-betti}
	For $m=2$, the graph has $N=7$ vertices.  Theorem \ref{thm:gaussian-betti} gives
	$$
	\beta_r(S/I_2)=2\binom{2}{r}+\binom{6}{r}-\binom{3}{r+1}.
	$$
	Therefore
	$$
	(\beta_1,\beta_2,\beta_3,\beta_4,\beta_5,\beta_6)=(7,16,20,15,6,1).
	$$
\end{example}

\medskip

\begin{table}[H]
	\centering
	\begin{tabular}{c c c c c}
		\toprule
		$m$ & $N$ & type blocks & $\operatorname{pd}(S/I_m)$ & first Betti numbers \\
		\midrule
		$1$ & $1$ & $(0)$ & $0$ & none \\
		$2$ & $7$ & $(6),(1,0)$ & $6$ & $7,16,20$ \\
		$3$ & $31$ & $(30),(14,13),(4,3,2,1)$ & $30$ & $91,1320,12090$ \\
		\bottomrule
	\end{tabular}
	\caption{The first values of the Betti formula.}
	\label{tab:gaussian-betti-values}
\end{table}
Table \ref{tab:gaussian-betti-values} gives the Gaussian Betti formula.  The displayed type lists include the harmless terminal zero when it occurs.
\medskip

\section{Cohen--Macaulayness and dimension for \texorpdfstring{$\mathbb{Z}_{2^m}[i]$}{Z_2^m[i]}}\label{sec:gaussian-CM}
 The projective dimension obtained from the corrected Betti formula almost reaches the number of variables.  This forces depth one in every nontrivial Gaussian case, and therefore separates the homological linearity of the edge ideal from the Cohen--Macaulay property of the quotient.

\medskip

In this section we examine the Cohen--Macaulayness of $S/I_m$, where $I_m=I(G_m)$ and $G_m=\Gamma(\mathbb{Z}_{2^m}[i])$.  We use the standard graded convention: a quotient is Cohen--Macaulay precisely when its depth equals its Krull dimension, see, for example, \cite{BrunsHerzog1993,HerzogHibi2011}.

\medskip

The polynomial ring $S$ has
$ N=|V(G_m)|=2^{2m-1}-1 $
variables.  From Theorem \ref{thm:gaussian-betti},
$ \operatorname{pd}_S(S/I_m)=2^{2m-1}-2 $
for every $m\geq 1$, with the value $0$ in the degenerate case $m=1$.  Auslander--Buchsbaum \cite{AuslanderBuchsbaum1957}
\begin{equation}\label{eq 1}
	\operatorname{depth}_S(S/I_m)
	= n - \operatorname{pd}_S(S/I_m)
	= (2^{2m-1}-1) - (2^{2m-1}-2) = 1.
\end{equation}
Thus depth alone does not detect the exceptional case; the dimension must also be computed.

\medskip

Next, we compute the Krull dimension of $S/I_m$.  As $I_m$ is a squarefree	monomial ideal generated by quadratic monomials corresponding to the edges of	$G_m$. So, its height equals the minimal size of a vertex cover of $G_m$, and hence $\dim(S/I_m) = n - \operatorname{ht}(I_m)\geq 2$  as soon as $G_m$ admits an independent set of size~$2$. For $m\ge2$, it follows immediately.  Take any two distinct vertices	$x,y\in V_1$.  Then $v(x)=v(y)=1$, so $v(x) + v(y) = 1+1 = 2 < 2m,$ and by the adjacency rule $x$ and $y$ are not adjacent.  Thus $\{x,y\}$ is an
independent set, and so there exists a vertex cover of size at most $n-2$, and whence for all $m\geq 2$, we obtain
\[
\dim(S/I_m) = n - \operatorname{ht}(I_m) \;\ge\; n-(n-2) = 2.
\]
For $m=1$, the situation is degenerate, $G_1$ has a single vertex and no	edges, so $I_1=(0)$ and $S/I_1\cong K[x]$ is a polynomial ring in one	variable, of dimension~$1$. Therefore, we have
\[
\dim(S/I_m)=
\begin{cases}
	1, & m=1,\\[2pt]
	\ge 2, & m\ge2.
\end{cases}
\]

We recall that a standard graded $K$–algebra $R$ is Cohen--Macaulay if $\operatorname{depth}(R)=\dim(R).$ For $R=S/I_m$, by Equation \eqref{eq 1}, $\operatorname{depth}(S/I_m)=1$, for all
$m\ge1$, while
\[
\dim(S/I_m)=1 \quad\text{if } m=1,\qquad\text{and}\qquad
\dim(S/I_m)\ge2 \quad\text{if } m\ge2.
\]

The next theorem gives the exact height and dimension.

\medskip

\begin{theorem}\label{thm:gaussian-dim-height}
	Let $G_m=\Gamma(\mathbb{Z}_{2^m}[i])$, and let $I_m=I(G_m)\subseteq S$.  Then
	$ \dim(S/I_m)=\alpha(G_m)=2^{2m-1}-2^m+1, $
	and
	$ \operatorname{height}(I_m)=2^m-2. $
\end{theorem}

\begin{proof}
	The graph $G_m$ is $C(2,2m)$.  Applying the even case of Proposition \ref{prop:chain-alpha} with $q=2$ and $L=2m$ gives
	$$
	\alpha(G_m)=2^{2m-1}-2^m+1.
	$$
	For an edge ideal, the Krull dimension of the Stanley--Reisner quotient equals the independence number of the graph.  Since $|V(G_m)|=2^{2m-1}-1$, so the height is
	$$
	\operatorname{height}(I_m)=|V(G_m)|-\alpha(G_m)=2^m-2.
	$$
\end{proof}

\medskip

The next theorem determines the Cohen--Macaulay cases.

\medskip

\begin{theorem}\label{thm:gaussian-CM}
	The quotient $S/I_m$ is Cohen--Macaulay if and only if $m=1$.
\end{theorem}

\begin{proof}
	If $m=1$, then $G_1$ has one vertex and no edge.  Hence $I_1=(0)$ and $S/I_1\cong K[x]$, which is Cohen--Macaulay.
	 Assume $m\geq 2$. Then by Theorem \ref{thm:gaussian-betti},
	$$
	\operatorname{pd}_S(S/I_m)=2^{2m-1}-2.
	$$
	Since $S$ has $N=2^{2m-1}-1$ variables, so Auslander--Buchsbaum implies
	$ \operatorname{depth}(S/I_m)=1. $
	On the other hand, Theprem \ref{thm:gaussian-dim-height} gives
	$$
	\dim(S/I_m)=2^{2m-1}-2^m+1.
	$$
	For $m\geq 2$, this dimension is at least $5$, and we have
	$$
	\operatorname{depth}(S/I_m)=1<\dim(S/I_m),
	$$
	so $S/I_m$ is not Cohen--Macaulay.
\end{proof}

\medskip

The following corollary describes the largest facets of the independence complex.

\medskip

\begin{corollary}\label{cor:gaussian-facets}
	For $m\geq 1$, the maximum independent sets of $G_m$ are precisely the sets
	$$
	V_1\cup V_2\cup\cdots\cup V_{m-1}\cup\{z\},\qquad z\in V_m.
	$$
	Consequently the number of facets of $\operatorname{Ind}(G_m)$ of maximum dimension is
	$ |V_m|=2^{m-1}. $
\end{corollary}

\begin{proof}
	This is the equality case in the even part of the proof of Proposition \ref{prop:chain-alpha}.  The layers $V_1,\ldots,V_{m-1}$ may all be chosen, and exactly one vertex from $V_m$ may be added because $V_m$ is a clique.  A layer of valuation larger than $m$ excludes at least all of $V_m$ and cannot compensate for that loss, so it cannot produce a larger independent set.
\end{proof}

\medskip

\begin{table}[H]
	\centering
	\begin{tabular}{c c c c c c}
		\toprule
		$m$ & $N$ & $\operatorname{height}(I_m)$ & $\dim(S/I_m)$ & $\operatorname{depth}(S/I_m)$ & CM? \\
		\midrule
		$1$ & $1$ & $0$ & $1$ & $1$ & yes \\
		$2$ & $7$ & $2$ & $5$ & $1$ & no \\
		$3$ & $31$ & $6$ & $25$ & $1$ & no \\
		$4$ & $127$ & $14$ & $113$ & $1$ & no \\
		\bottomrule
	\end{tabular}
	\caption{Depth, dimension, and Cohen--Macaulay behavior for the Gaussian family $\Gamma(\mathbb{Z}_{2^m}[i])$.}
	\label{tab:gaussian-CM}
\end{table}
Table \ref{tab:gaussian-CM} gives the numerical data for the depth, dimension, and Cohen--Macaulay behavior for the Gaussian family $\Gamma(\mathbb{Z}_{2^m}[i])$.  The quotient has a linear resolution in the nontrivial cases, but it is not Cohen--Macaulay once $m\geq 2$.
\medskip

\section{Hilbert series for \texorpdfstring{$\mathbb{Z}_{2^m}[i]$}{Z_2^m[i]}}\label{sec:gaussian-hilbert}
 The Hilbert series here is obtained twice, first from the independence polynomial and then from the corrected Betti numerator.  Agreement of these two expressions gives a useful validation of the corrected global subtraction term in Theorem \ref{thm:gaussian-betti}.

\medskip

Let $\Delta_m=\operatorname{Ind}(G_m)$ be the independence complex of $G_m$.  Since $I_m=I(G_m)$, the quotient $S/I_m$ is the Stanley--Reisner ring $K[\Delta_m]$. Let $ n_k=|V_k|=2^{2m-k-1},$ for $ 1\leq k\leq 2m-1. $
By Proposition \ref{prop:chain-independence-polynomial}, the independence polynomial of $G_m$ is
$$
F_{G_m}(y)=1+\sum_{s=1}^{2m-1}A_s(y)\prod_{k=1}^{\min\{s-1,2m-s-1\}}(1+y)^{2^{2m-k-1}},
$$
where
$$
A_s(y)=
\begin{cases}
	(1+y)^{2^{2m-s-1}}-1, & 2s<2m,\\[2pt]
	2^{2m-s-1}y, & 2s\geq 2m.
\end{cases}
$$
Therefore
$ \operatorname{Hilb}_{S/I_m}(t)=F_{G_m}\left(\frac{t}{1-t}\right).$ Equivalently, if
$ F_{G_m}(y)=\sum_{r=0}^{\alpha(G_m)}f_{r-1}y^r, $
then
$$
\operatorname{Hilb}_{S/I_m}(t)=\frac{1}{(1-t)^N}\sum_{r=0}^{\alpha(G_m)}f_{r-1}t^r(1-t)^{N-r},\qquad N=2^{2m-1}-1.
$$

\medskip

The coefficients $f_{r-1}$ can also be written directly by valuation-layer choices.  If $a_k$ is the number of chosen vertices from $V_k$, then
$$
f_{r-1}=\sum\prod_{k=1}^{2m-1}\binom{2^{2m-k-1}}{a_k},
$$
where the sum runs over all vectors $(a_1,\ldots,a_{2m-1})$ satisfying
$$
\sum_{k=1}^{2m-1}a_k=r,
\qquad
0\leq a_k\leq 2^{2m-k-1}, k+\ell<2m\quad\text{whenever }a_k>0\text{ and }a_\ell>0.
$$
In particular, $a_k\leq 1$ whenever $k\geq m$.

\medskip

\begin{remark}\label{rem:gaussian-hilbert-betti-check}
	The Hilbert series may also be recovered from the minimal free resolution:
	$$
	\operatorname{Hilb}_{S/I_m}(t)=
	\frac{1+\sum_{r\geq 1}(-1)^r\beta_r(S/I_m)t^{r+1}}{(1-t)^N}.
	$$
	The shift is $r+1$, not $r+2$, because this formula is for the quotient $S/I_m$.  This expression is often the fastest way to check a numerical Hilbert numerator once the corrected Betti numbers are known.
\end{remark}

\medskip

\begin{example}\label{ex:m2-hilbert}
	For $m=2$, we have $|V_1|=4$, $|V_2|=2$, and $|V_3|=1$.  The independence polynomial is
	$$
	F_{G_2}(y)=1+\bigl((1+y)^4-1\bigr)+2y(1+y)^4+y=1+7y+14y^2+16y^3+9y^4+2y^5.
	$$
	Thus
	$$
	\operatorname{Hilb}_{S/I_2}(t)=\frac{1-7t^2+16t^3-20t^4+15t^5-6t^6+t^7}{(1-t)^7}.
	$$
	This agrees with the Betti numerator obtained from Example \ref{ex:m2-betti}:
	$$
	1-7t^2+16t^3-20t^4+15t^5-6t^6+t^7.
	$$
\end{example}

\medskip

\begin{table}[H]
	\centering
	\resizebox{\textwidth}{!}{%
		\begin{tabular}{c c c c}
			\toprule
			$m$ & independence polynomial beginning & denominator exponent & numerator beginning \\
			\midrule
			$1$ & $1+y$ & $1$ & $1$ \\
			$2$ & $1+7y+14y^2+16y^3+\cdots$ & $7$ & $1-7t^2+16t^3-20t^4+\cdots$ \\
			$3$ & $1+31y+374y^2+\cdots$ & $31$ & $1-91t^2+1320t^3-12090t^4+\cdots$ \\
			\bottomrule
		\end{tabular}}
	\caption{Initial Hilbert-series data for $\Gamma(\mathbb{Z}_{2^m}[i])$.}
	\label{tab:gaussian-hilbert-values}
\end{table}

Table \ref{tab:gaussian-hilbert-values} gives the initial Hilbert-series data for $\Gamma(\mathbb{Z}_{2^m}[i])$.  The independence-polynomial expansion and the Betti numerator give the same Hilbert series.
\medskip

\section{Zero-divisor graphs of \texorpdfstring{$\mathbb{Z}_p[x]/(x^c)$}{Z_p[x]/(x^c)}}\label{sec:poly-ring}
 This section shows that the truncated polynomial ring $\mathbb F_p[x]/(x^c)$ yields exactly the chain graph $C(p,c)$.  Thus the corrected formula for prime powers is not special to integer residue rings, as it is a valuation-layer phenomenon for finite principal chain rings.

\medskip
Let $p$ be a prime number and $c\in\mathbb N$ with $c\ge 2$.  Consider the ring $R_{p,c} = \mathbb Z_p[x]/\langle x^c\rangle \cong \mathbb F_p[x]/(x^c),$ where $\mathbb Z_p\cong\mathbb F_p$ is the field with $p$ elements, and let its zero divisor graph be $G_{p,c}=\Gamma(R_{p,c}).$  The ring $R_{p,c}$ is a standard example of a finite local principal ideal
(chain) ring.  Every element $f(x)\in R_{p,c}$ can be written uniquely as a polynomial of degree $<c$, that is,
\[
f(x) = a_0 + a_1 x + \cdots + a_{c-1}x^{c-1},\qquad a_i\in\mathbb F_p.
\]
The ideal generated by the image of $x$, $\mathfrak m = (x) \subset R_{p,c},$ 	is the unique maximal ideal, and $\mathfrak m^c=(0)\neq\mathfrak m^{c-1}$.  In	particular, $R_{p,c}$ is a local ring of length $c$ with residue field
$R_{p,c}/\mathfrak m\cong\mathbb F_p$.  Every nonzero $f\in R_{p,c}$ can be written uniquely in the form 	$f(x) = x^k u(x),$ where $0\le k\le c-1$ and $u(x)\in R_{p,c}$ is a unit (that is, $u(0)\ne 0$).  We
define the $x$--adic valuation by
\[
v(f) =
\begin{cases}
	k & \text{if } f(x)=x^k u(x)\text{ with }u\text{ a unit } (\text{or } u\in R_{p,c}^\times),\\[2pt]
	+\infty & \text{if }f=0.
\end{cases}
\]
Then $f$ is a unit if and only if $v(f)=0$, and $f$ is a nonzero zero divisor if and only if
$1\le v(f)\le c-1$.	For $k=0,1,\dots,c$, the powers of the maximal ideal are $\mathfrak m^k = (x^k),$ for $0\le k\le c,$	and each $\mathfrak m^k$ consists of all equivalence classes of polynomials 	divisible by $x^k$. 

\medskip
The following lemma counts the valuation layers.
\begin{lemma}\label{lem:pc-mk-revised}
	For $0\le k\le c$, $|\mathfrak m^k| = p^{c-k}$, and for $0\le k\le c-1$, the set $V_k = \{f\in R_{p,c} \mid v(f)=k\}
	= \mathfrak m^k\setminus \mathfrak m^{k+1}$ has cardinality	$|V_k| = p^{c-k-1}(p-1)$.  In particular, the nonzero zero divisors of $R_{p,c}$ are $Z^{*}(R_{p,c})
	= \bigcup_{k=1}^{c-1} V_k$ and $|Z^{*}(R_{p,c})|
	= \sum_{k=1}^{c-1} p^{c-k-1}(p-1)
	= p^{c-1}-1.$ 
\end{lemma}

\begin{proof}
	The ideal $\mathfrak m^k=(x^k)$ consists precisely of the classes of polynomials
	of the form $x^k g(x)$ with $\deg g\le c-k-1$. So, $g(x)$ has $c-k$ coefficients,	each with $p$ choices, and hence $|\mathfrak m^k|=p^{c-k}$.  Therefore, we have
	\[
	|V_k| = |\mathfrak m^k|-|\mathfrak m^{k+1}|
	= p^{c-k}-p^{c-k-1}
	= p^{c-k-1}(p-1).
	\]
	Summing from $k=1$ to $c-1$, and then with substitution $j=c-k-1$,  we obtain
	\[
	|Z^{*}(R_{p,c})|
	= \sum_{k=1}^{c-1} p^{c-k-1}(p-1)
	= (p-1)\sum_{j=0}^{c-2} p^j
	= p^{c-1}-1.
	\]
\end{proof}

\medskip
Let $f,g\in R_{p,c}$ be nonzero zero divisors with $v(f)=k$ and $v(g)=\ell$, for $1\le k,\ell\le c-1$.  The valuation is additive on	products up to the nilpotency index $c$, so we have
\[
v(fg) =
\begin{cases}
	k+\ell & \text{if }k+\ell\le c-1,\\
	\ge c  & \text{if }k+\ell\ge c.
\end{cases}
\]
Also, in the latter case $x^c\mid fg$, so $fg=0$ in $R_{p,c}$.  Equivalently,  the following result gives the description of adjacency in $G_{p,c}$.

\medskip

\begin{lemma}\label{lem:pc-adj-revised}
	Let $1\leq k,\ell\leq c-1$ and let $f\in V_k$, $g\in V_\ell$ with $f\neq g$.  Then
	$ fg=0$ if and only if $ k+\ell\geq c. $
	Equivalently,
	$$
	E(G_{p,c})=\bigl\{\{f,g\}\mid f\in V_k,\ g\in V_\ell,\ f\neq g,\ k+\ell\geq c\bigr\}.
	$$
\end{lemma}

\begin{proof}
	Let $f=x^k u$ and $g=x^\ell w$ with $u,w\in R_{p,c}^\times$.  Then
	$ fg=x^{k+\ell}uw. $
	Since $uw$ is a unit, this product is zero in $\mathbb F_p[x]/(x^c)$ if and only if $x^c$ divides $x^{k+\ell}$, which is equivalent to $k+\ell\geq c$.
\end{proof}

\medskip
In particular, for $k\ge\lceil c/2\rceil$ the induced subgraph on $V_k$ is
complete, because if $f,g\in V_k$ with $f\neq g$, then $k+\ell=2k\ge c$ and
$fg=0$. Consider the prime power ring $A=\mathbb Z/p^c\mathbb Z$.  Every nonzero element
$a\in A$ can be written uniquely as $a=u p^k$ with $u$ a unit and
$0\le k\le c-1$.  Define the $p$--adic valuation $w$ by $w(up^k)=k$ and
$w(0)=+\infty$.  The nonzero zero divisors in $A$ are the elements with
$1\le w(a)\le c-1$. For $1\le k\le c-1$, consider
\[
W_k = \{a\in A \mid w(a)=k\}.
\]
Then it is clear  that $|W_k| = p^{c-k-1}(p-1), 1\le k\le c-1,$ and for $a\in W_k$, $b\in W_\ell$ with $a\neq b$, we obtain $ab=0$  if and only if $k+\ell\ge c.$	In other words, the zero divisor graph $\Gamma(A)$ has the same valuation layer sizes and the same adjacency rule as $G_{p,c}$.

\medskip

The same layer sizes and edge rule of $\mathbb{Z}_p[x]/\langle x^c\rangle$ occur in $\mathbb{Z}/p^c\mathbb{Z}$.

\medskip

\begin{proposition}\label{prop:pc-iso-revised}
	For every prime $p$ and every integer $c\geq 2$, there is a graph isomorphism
	$$
	\Psi_{p,c}:\Gamma(\mathbb{Z}_p[x]/\langle x^c\rangle)\longrightarrow \Gamma(\mathbb{Z}/p^c\mathbb{Z}).
	$$
\end{proposition}

\begin{proof}
	Let $A=\mathbb{Z}/p^c\mathbb{Z}$, and let $W_k$ be the set of nonzero elements of $A$ with $p$-adic valuation $k$.  Then
	$ |W_k|=p^{c-k-1}(p-1)=|V_k| $
	for $1\leq k\leq c-1$.  Choose a bijection $\psi_k:V_k\to W_k$ for each $k$, and define
	$$
	\Psi_{p,c}(f)=\psi_k(f)\quad\text{if }f\in V_k.
	$$
	This is a bijection on the nonzero zero divisors.
	
	\medskip
	
	If $f\in V_k$ and $g\in V_\ell$ with $f\neq g$, then Lemma \ref{lem:pc-adj-revised} and the usual valuation rule in $\mathbb{Z}/p^c\mathbb{Z}$ give
	$$
	fg=0\quad\Longleftrightarrow\quad k+\ell\geq c\quad\Longleftrightarrow\quad \Psi_{p,c}(f)\Psi_{p,c}(g)=0.
	$$
	Thus $\{f,g\}$ is an edge of $G_{p,c}$ if and only if
	$\{\Psi_{p,c}(f),\Psi_{p,c}(g)\}$ is an edge of $\Gamma(A)$, thereby it shows that
	$\Psi_{p,c}$ is a graph isomorphism.
\end{proof}

\medskip

\begin{example}\label{ex:pc-23}
	Let $p=2$ and $c=3$.  Then $R_{2,3}=\mathbb F_2[x]/(x^3)$ has two elements of valuation $1$, namely $x$ and $x+x^2$, and one element of valuation $2$, namely $x^2$.  The graph is the path $P_3$, with $x^2$ adjacent to the two valuation-one vertices and no edge between them.
\end{example}

\medskip

\section{Cochordality and Betti numbers for \texorpdfstring{$\mathbb{Z}_p[x]/(x^c)$}{Z_p[x]/(x^c)}}\label{sec:pc-betti-section}
 The following results are the corrected chain-ring specialization of the type formula.  They agree with the prime-power case of the corrected paper after replacing $q$ by the prime $p$ and the nilpotency index by $c$.

\medskip

The graph $\Gamma(R_{p,c})$ is the layered graph $C(p,c)$.  The following statements are therefore immediate specializations of Section \ref{sec:chain-template}.

\medskip

\begin{theorem}\label{thm:pc-type-revised}
	Let $G_{p,c}=\Gamma(\mathbb{Z}_p[x]/(x^c))$, where $p$ is prime and $c\geq 2$.  For $\lceil c/2\rceil\leq i\leq c-1$, define
	$$
	s_i=\bigl(p^i-p^{c-i-1}-1,
	p^i-p^{c-i-1}-2,
	\ldots,
	p^i-p^{c-i}\bigr).
	$$
	Then $G_{p,c}$ is cochordal and has type list
	$$
	(s_{c-1},s_{c-2},\ldots,s_{\lceil c/2\rceil}).
	$$
	If a terminal zero occurs, it may be deleted without changing the graph or the Betti numbers.
\end{theorem}

\begin{proof}
	By Lemmas \ref{lem:pc-mk-revised} and \ref{lem:pc-adj-revised}, the graph $G_{p,c}$ is $C(p,c)$.  Therefore Corollary \ref{cor:type-chain} gives the stated type list, and Lemma \ref{rem:zero-type} gives the terminal-zero convention.
\end{proof}

\medskip

The next theorem gives the Betti numbers of the quotient by the edge ideal.

\medskip

\begin{theorem}\label{thm:pc-betti-revised}
	Let $G_{p,c}=\Gamma(\mathbb{Z}_p[x]/(x^c))$, let $N=|V(G_{p,c})|=p^{c-1}-1$, and let $I_{p,c}=I(G_{p,c})\subseteq S=K[x_1,\ldots,x_N]$.  For every $r\geq 1$,
	$$
	\beta_r(S/I_{p,c})=\beta_{r,r+1}(S/I_{p,c})=
	\sum_{j=\lceil c/2\rceil}^{c-1}p^{c-j-1}(p-1)\binom{p^j-2}{r}
	-\binom{p^{\lfloor c/2\rfloor}-1}{r+1}.
	$$
	Moreover, 	$ 	\operatorname{pd}_S(S/I_{p,c})=p^{c-1}-2. $
	If $I_{p,c}\neq 0$, then $I_{p,c}$ has a $2$-linear resolution and $\operatorname{reg}(S/I_{p,c})=1$.  The only case with $I_{p,c}=0$ and $c\geq 2$ is $(p,c)=(2,2)$.
\end{theorem}

\begin{proof}
	Apply Theorem \ref{thm:chain-betti} with $q=p$ and $L=c$.  This gives the Betti formula, the projective dimension, and the linearity statement.  The graph has no edge only if it has one vertex, because for $c\geq 2$ the top layer $V_{c-1}$ is adjacent to every other nonzero zero divisor and is a clique when it has more than one vertex.  The equality $p^{c-1}-1=1$ holds exactly when $(p,c)=(2,2)$.
\end{proof}

\medskip

The following consequence is immediate.
\begin{corollary}\label{cor:pc-ideal-betti}
	When $I_{p,c}\neq 0$, the graded Betti numbers of the edge ideal itself satisfy
	$$
	\beta_{r,r+2}(I_{p,c})=\beta_{r+1,r+2}(S/I_{p,c})\quad\text{for }r\geq 0,
	$$
	and all other graded Betti numbers of $I_{p,c}$ vanish.
\end{corollary}

\begin{proof}
	The result follows from the short exact sequence $0\to I_{p,c}\to S\to S/I_{p,c}\to 0$ and the fact that $S/I_{p,c}$ has nonzero graded Betti numbers only in bidegrees $(r,r+1)$.
\end{proof}

\medskip

\begin{example}\label{ex:complete-c2}
	Let $c=2$.  Then all nonzero zero divisors have valuation $1$, and any two distinct such elements multiply to zero.  Hence
	$ G_{p,2}\cong K_{p-1}. $
	The formula in Theorem \ref{thm:pc-betti-revised} becomes
	$$
	\beta_r(S/I_{p,2})=(p-1)\binom{p-2}{r}-\binom{p-1}{r+1}=r\binom{p-1}{r+1},
	$$
	which is the usual Betti sequence of the quotient by the edge ideal of a complete graph.
\end{example}

\medskip

\begin{example}\label{ex:pc-24-betti}
	Let $p=2$ and $c=4$.  Then $G_{2,4}$ is the same graph as $\Gamma(\mathbb{Z}_{2^2}[i])$ in Example \ref{ex:m2-betti}.  The Betti formula gives
	$$
	\beta_r(S/I_{2,4})=2\binom{2}{r}+\binom{6}{r}-\binom{3}{r+1},
	$$
	so
	$$
	(\beta_1,\beta_2,\beta_3,\beta_4,\beta_5,\beta_6)=(7,16,20,15,6,1).
	$$
	This confirms the graph isomorphism between the two chain-ring presentations in this numerical case.
\end{example}

\section{Cohen--Macaulayness and Hilbert series for \texorpdfstring{$\mathbb{Z}_p[x]/(x^c)$}{Z_p[x]/(x^c)}}\label{sec:pc-CM-Hilbert}
 This final section completes the parallel with the Gaussian family.  The quotient is Cohen--Macaulay exactly in the complete-graph case $c=2$, and the Hilbert series is governed by the independence polynomial of the chain graph $C(p,c)$.
With these notations, the valuation layers or subsets $V_k$ ($1\le k\le c-1$) are defined as in
Lemma~\ref{lem:pc-mk-revised}, with $|V_k| = p^{c-k-1}(p-1),$ and	$|V(G_{p,c})| = p^{c-1}-1.$ 
\medskip

Let $I_{p,c}=I(G_{p,c})\subseteq S$, where $G_{p,c}=\Gamma(\mathbb{Z}_p[x]/(x^c))$ and $S$ has
$ N=p^{c-1}-1 $
variables.  We first compute the exact dimension and height.

\medskip

\begin{theorem}\label{thm:pc-dim-height}
	Let $G_{p,c}=\Gamma(\mathbb{Z}_p[x]/(x^c))$, and let $I_{p,c}=I(G_{p,c})$.  If $c=2a$, then
	$$
	\dim(S/I_{p,c})=\alpha(G_{p,c})=p^{2a-1}-p^a+1,
	\qquad
	\operatorname{height}(I_{p,c})=p^a-2.
	$$
	If $c=2a+1$, then
	$$
	\dim(S/I_{p,c})=\alpha(G_{p,c})=p^{2a}-p^a,
	\qquad
	\operatorname{height}(I_{p,c})=p^a-1.
	$$
\end{theorem}

\begin{proof}
	The graph $G_{p,c}$ is $C(p,c)$ by Lemmas \ref{lem:pc-mk-revised} and \ref{lem:pc-adj-revised}.  Therefore Proposition \ref{prop:chain-alpha} gives the displayed independence number and height.  The equality $\dim(S/I_{p,c})=\alpha(G_{p,c})$ follows because $S/I_{p,c}$ is the Stanley--Reisner ring of the independence complex of $G_{p,c}$.
\end{proof}

\medskip

The Cohen--Macaulay property is now determined exactly.

\medskip

\begin{theorem}\label{thm:pc-CM-revised}
	For $c\geq 2$, the quotient $S/I_{p,c}$ is Cohen--Macaulay if and only if $c=2$.  If one also includes $c=1$, then $R_{p,1}\cong\mathbb F_p$ has no nonzero zero divisors and the corresponding quotient is trivially Cohen--Macaulay.
\end{theorem}

\begin{proof}
	Assume first that $c=2$, then $G_{p,2}\cong K_{p-1}$. So its independence number is $1$, and $\dim(S/I_{p,2})=1$.  Also Theorem \ref{thm:pc-betti-revised} gives
	$ \operatorname{pd}_S(S/I_{p,2})=p-2. $
	Since $S$ has $p-1$ variables, so Auslander--Buchsbaum gives
	$$
	\operatorname{depth}(S/I_{p,2})=(p-1)-(p-2)=1.
	$$
	Thus depth equals dimension, and $S/I_{p,2}$ is Cohen--Macaulay.
	
	 Now assume $c\geq 3$.  Again by Theorem \ref{thm:pc-betti-revised}, we get
	$
	\operatorname{pd}_S(S/I_{p,c})=p^{c-1}-2.
	$
	As $S$ has $N=p^{c-1}-1$ variables, we obtain
	$$
	\operatorname{depth}(S/I_{p,c})=N-\operatorname{pd}_S(S/I_{p,c})=1.
	$$
	On the other hand, Theorem \ref{thm:pc-dim-height} gives $\dim(S/I_{p,c})\geq 2$ for every $c\geq 3$.  Hence depth is strictly smaller than dimension, so $S/I_{p,c}$ is not Cohen--Macaulay.
\end{proof}

\medskip

We now compute the Hilbert series.  Let $\Delta_{p,c}=\operatorname{Ind}(G_{p,c})$.  Then $S/I_{p,c}=K[\Delta_{p,c}]$.  Also, set
$ n_k=|V_k|=p^{c-k-1}(p-1),$ for $1\leq k\leq c-1. $ By Proposition \ref{prop:chain-independence-polynomial}, the independence polynomial is
$$
F_{G_{p,c}}(y)=1+\sum_{s=1}^{c-1}A_s(y)\prod_{k=1}^{\min\{s-1,c-s-1\}}(1+y)^{p^{c-k-1}(p-1)},
$$
where
$$
A_s(y)=
\begin{cases}
	(1+y)^{p^{c-s-1}(p-1)}-1, & 2s<c,\\[2pt]
	p^{c-s-1}(p-1)y, & 2s\geq c.
\end{cases}
$$
Therefore, 
$
\operatorname{Hilb}_{S/I_{p,c}}(t)=F_{G_{p,c}}\left(\frac{t}{1-t}\right).
$
If
$
F_{G_{p,c}}(y)=\sum_{r=0}^{\alpha(G_{p,c})}f_{r-1}y^r,
$
then
$$
\operatorname{Hilb}_{S/I_{p,c}}(t)=\frac{1}{(1-t)^N}\sum_{r=0}^{\alpha(G_{p,c})}f_{r-1}t^r(1-t)^{N-r},\qquad N=p^{c-1}-1.
$$
Equivalently, for $r\geq 0$, the coefficient $f_{r-1}$ is
$$
f_{r-1}=\sum\prod_{k=1}^{c-1}\binom{p^{c-k-1}(p-1)}{a_k},
$$
where the sum is over all vectors $(a_1,\ldots,a_{c-1})$ satisfying
$$
\sum_{k=1}^{c-1}a_k=r,
\qquad
0\leq a_k\leq p^{c-k-1}(p-1), ~k+\ell<c\quad\text{whenever }a_k>0\text{ and }a_\ell>0.
$$
In particular, if $k\geq \lceil c/2\rceil$, then automatically $a_k\leq 1$.

\medskip

\begin{remark}\label{rem:pc-hilbert-betti-check}
	The Betti-number expression for the Hilbert series is
	$$
	\operatorname{Hilb}_{S/I_{p,c}}(t)=
	\frac{1+\sum_{r\geq 1}(-1)^r\beta_r(S/I_{p,c})t^{r+1}}{(1-t)^N}.
	$$
	This gives a useful check on computations obtained from the independence polynomial.  The agreement depends on using the corrected Betti formula with the global subtraction term.
\end{remark}

\medskip
\begin{example}[The Cohen--Macaulay case \(c=2\)]
	For \(c=2\), the graph is \(K_{p-1}\).  Its independent sets are the empty set and the singletons, so
	\[
	F_{G_{p,2}}(y)=1+(p-1)y.
	\]
	Thus
	\[
	H_{S/I_{p,2}}(t)=1+(p-1)\frac{t}{1-t}
	=\frac{1+(p-2)t}{1-t}.
	\]
	For \(p=3\), this gives
	\[
	H_{S/I_{3,2}}(t)=\frac{1+t}{1-t}
	=\frac{1-t^2}{(1-t)^2},
	\]
	which agrees with \(I_{3,2}=(x_1x_2)\subset K[x_1,x_2]\).
\end{example}

\begin{example}[A non-Cohen--Macaulay case: \(p=2,c=3\)]
	Let \(R_{2,3}=\mathbb{Z}_2[x]/(x^3)\).  Then \(|V_1|=2\) and \(|V_2|=1\).  The graph is a path of length two: the unique vertex in \(V_2\) is adjacent to both vertices in \(V_1\), while the two vertices in \(V_1\) are not adjacent.  Thus
	\[
	F_{G_{2,3}}(y)=1+3y+y^2.
	\]
	Consequently
	\[
	H_{S/I_{2,3}}(t)
	=1+3\frac{t}{1-t}+\left(\frac{t}{1-t}\right)^2
	=\frac{1-2t^2+t^3}{(1-t)^3}.
	\]
	The Betti formula gives \(\beta_1=2\) and \(\beta_2=1\), hence the same numerator
	\[
	1-2t^2+t^3.
	\]
	Moreover \(\operatorname{depth}(S/I_{2,3})=1\) while \(\dim(S/I_{2,3})=2\), so this quotient is not Cohen--Macaulay.
\end{example}

\medskip

\begin{table}[H]
	\centering
	\begin{tabular}{c c c c c}
		\toprule
		$(p,c)$ & $N$ & $\dim(S/I_{p,c})$ & $\operatorname{depth}(S/I_{p,c})$ & CM? \\
		\midrule
		$(2,2)$ & $1$ & $1$ & $1$ & yes \\
		$(3,2)$ & $2$ & $1$ & $1$ & yes \\
		$(2,3)$ & $3$ & $2$ & $1$ & no \\
		$(2,4)$ & $7$ & $5$ & $1$ & no \\
		$(3,3)$ & $8$ & $6$ & $1$ & no \\
		\bottomrule
	\end{tabular}
	\caption{Cohen--Macaulay behavior for the truncated polynomial family.}
	\label{tab:pc-CM-values}
\end{table}
Table \ref{tab:pc-CM-values} gives the numerical values of Cohen--Macaulay behavior for the truncated polynomial family.  Only the complete-graph case $c=2$ has depth equal to dimension.

\section{Conclusion and future work}\label{sec:conclusion} 
In this paper, we developed a uniform chain-ring framework for studying zero-divisor graphs through their valuation layers.  The central object was the layered graph $C(q,L)$, whose adjacency relation is determined by the inequality $k+\ell\ge L$.  This simple rule captures the zero-divisor graph of a finite chain ring with residue field of order $q$ and nilpotency index $L$, and it allows the graph-theoretic and homological computations to be carried out without treating each ring element separately.

\medskip

Using the cochordal constructible-system method, we proved that $C(q,L)$ is cochordal and obtained an explicit type sequence.  Substituting this type sequence into the corrected Dung and Vu \cite{DungVu2026} formula gave a closed expression for all graded Betti numbers of the quotient by the edge ideal.  In particular, for $I=I(C(q,L))\subseteq S$, we obtained a $2$-linear resolution whenever $I\ne0$, together with the projective dimension $ \operatorname{pd}_S(S/I)=q^{L-1}-2. $
The correction term in the Betti formula is essential, as it records the global overlap created by the ordered star construction and prevents overcounting.

\medskip

The general theory was applied first to the Gaussian quotient $ R_m=\mathbb Z_{2^m}[i]\cong \mathbb Z[i]/(1+i)^{2m}. $
Since $R_m$ is a finite chain ring with residue field of size $2$ and nilpotency index $2m$, its zero-divisor graph is precisely $C(2,2m)$.  This identification gives the type sequence, the complete Betti table, the projective dimension, and the regularity of $S/I(\Gamma(R_m))$.  We also computed the independence number $\alpha(\Gamma(R_m))=2^{2m-1}-2^m+1$ and height $\operatorname{height} I(\Gamma(R_m))=2^m-2.$ Consequently, $S/I(\Gamma(R_m))$ is Cohen--Macaulay only for $m=1$.
 The same method was then applied to the truncated polynomial rings $ R_{p,c}=\mathbb Z_p[x]/(x^c). $
Here the valuation layers are the sets $\mathfrak m^k\setminus\mathfrak m^{k+1}$, and their sizes are $p^{c-k-1}(p-1)$.  The adjacency condition is again exactly $k+\ell\ge c$, so $\Gamma(R_{p,c})\cong C(p,c)$.  This gives closed formulae for the graded Betti numbers, projective dimension, independence number, height, and Hilbert series.  The Cohen--Macaulay classification is especially clean: for $c\ge2$, the quotient $S/I(\Gamma(R_{p,c}))$ is Cohen--Macaulay if and only if $c=2$.
The graph $\Gamma(\mathbb Z_{2^m}[i])$ is $C(2,2m)$, while $\Gamma(\mathbb Z_p[x]/(x^c))$ is $C(p,c)$.  Thus both are threshold graphs with valuation weights, and their previously computed invariants are threshold-graph invariants as well.
\medskip

The Hilbert-series computations provide an additional check on the Betti-number formulae.  On one hand, the Hilbert series is obtained from the independence polynomial of the graph as 
$ H_{S/I(G)}(t)=F_G\left(\frac{t}{1-t}\right). $
On the other hand, the same series is recovered from the minimal free resolution:
$$
H_{S/I(G)}(t)=
\frac{1+\sum_{r\ge1}(-1)^r\beta_r(S/I(G))t^{r+1}}{(1-t)^{|V(G)|}}.
$$
The agreement between these two descriptions confirms the consistency of the corrected Betti expressions and the combinatorial description of the independence complex.

\medskip

Several directions remain open.  First, it would be useful to extend the method beyond chain rings to finite principal ideal rings with more than one maximal ideal, where the valuation-layer picture is replaced by a multi-indexed divisor structure.  Second, one may study finer invariants of the independence complexes arising here, such as shellability, sequential Cohen--Macaulayness, and levelness.  Third, the explicit type sequences suggest possible algorithms for computing Betti tables of larger classes of cochordal zero-divisor graphs without using Hochster's formula directly.  Finally, the comparison between graph isomorphism classes of zero-divisor graphs and ring-theoretic isomorphism classes remains a natural problem: the examples in this paper show that different rings may produce the same zero-divisor graph, but a systematic description of this phenomenon is still incomplete.

\medskip

\section*{Declarations}
\noindent \textbf{Data Availability:} There is no data associated with this article.

\medskip

\noindent \textbf{Funding:} The authors did not receive support from any organization for the submitted work.

\medskip

\noindent \textbf{Conflict of interest:} The authors have no competing interests to declare that are relevant to the content of this article.

\medskip

\noindent\textbf{Note:} For any comments and suggestions regarding this article, please feel free to contact at \href{mailto:bilalahmadrr@gmail.com}{bilalahmadrr@gmail.com}.

\end{document}